\documentclass[a4paper,11pt,oneside]{amsart}

\usepackage{geometry} 
\usepackage[utf8]{inputenc} 
\usepackage[T1]{fontenc} 
\usepackage{amssymb,amsmath,amsfonts,amsxtra,amsthm} 
\usepackage{lmodern} 
\usepackage{cases} 
\usepackage{mathtools} 
\usepackage{tikz-cd}
\usepackage[all]{xy} 
\usepackage{hyperref} 
\usepackage{faktor} 
\usepackage{xargs} 
 \usepackage{mathrsfs} 

\theoremstyle{definition}
\newtheorem{definition}{Definition}[section]

\theoremstyle{plain}

\newtheorem{theorem}[definition]{Theorem}

\newtheorem{theoremintro}{Theorem}

\newtheorem{proposition}[definition]{Proposition}

\newtheorem{lemma}[definition]{Lemma}
\newtheorem{corollary}[definition]{Corollary}
\newtheorem*{factnonnumerote}{Fact}
\newtheorem{fact}[definition]{Fact}
\theoremstyle{remark}
\newtheorem{example}[definition]{Example}
\newtheorem{remark}[definition]{Remark}

\newcommand{\ensemblenombre}[1]{\ensuremath{\mathbb{#1}}}
\newcommand{\N}{\ensemblenombre{N}}

\newcommand{\R}{\ensemblenombre{R}}
\newcommand{\C}{\ensemblenombre{C}}
\newcommand{\abs}[1]{\ensuremath{\left\lvert#1\right\rvert}}
\newcommand{\norme}[1]{\ensuremath{\left\lVert#1\right\rVert}}
\newcommand{\enstq}[2]{\ensuremath{\left\{#1\mathrel{}\middle|\mathrel{}#2\right\}}}

\newcommand{\transp}[1]{\prescript{t}{}{#1}}

\newcommand{\intervalle}[4]{\ensuremath{\mathopen{#1}#2
		\mathclose{}\mathpunct{};#3
		\mathclose{#4}}}
\newcommand{\intervalleff}[2]{\intervalle{[}{#1}{#2}{]}}
\newcommand{\intervalleof}[2]{\intervalle{]}{#1}{#2}{]}}
\newcommand{\intervallefo}[2]{\intervalle{[}{#1}{#2}{[}}
\newcommand{\intervalleoo}[2]{\intervalle{]}{#1}{#2}{[}}

\newcommand{\restreinta}{\ensuremath{\mathclose{}|\mathopen{}}}

\newcommand{\GL}[1]{\ensuremath{\mathrm{GL}_{#1}\mathopen{(}\R\mathclose{)}}}
\newcommand{\GLplus}[1]{\ensuremath{\mathrm{GL}^+_{#1}\mathopen{(}\R\mathclose{)}}}
\newcommand{\PGL}[1]{\ensuremath{\mathrm{PGL}_{#1}\mathopen{(}\R\mathclose{)}}}
\newcommand{\SL}[1]{\ensuremath{\mathrm{SL}_{#1}\mathopen{(}\R\mathclose{)}}}
\newcommand{\SLtilde}[1]{\ensuremath{\mathrm{\widetilde{SL}}_{#1}\mathopen{(}\R\mathclose{)}}}
\newcommand{\PSL}[1]{\ensuremath{\mathrm{PSL}_{#1}\mathopen{(}\R\mathclose{)}}}
\newcommand{\slR}[1]{\ensuremath{\mathfrak{sl}_{#1}}}
\newcommand{\gl}[1]{\ensuremath{\mathfrak{gl}_{#1}}}
\newcommand{\SO}[1]{\ensuremath{\mathrm{SO}\mathopen{(}#1\mathclose{)}}}
\newcommand{\so}[1]{\ensuremath{\mathfrak{so}\mathopen{(}#1\mathclose{)}}}
\newcommand{\biso}[2]{\ensuremath{\mathfrak{so}\mathopen{(}#1\mathpunct,#2\mathclose{)}}}
\newcommand{\Heis}[1]{\ensuremath{\mathrm{Heis}\mathopen{(}#1\mathclose{)}}}
\newcommand{\heis}[1]{\ensuremath{\mathfrak{heis}\mathopen{(}#1\mathclose{)}}}
\newcommand{\Aff}[1]{\ensuremath{\mathrm{Aff}\mathopen{(}\R^{#1}\mathclose{)}}}

\newcommand{\Sim}[1]{\ensuremath{\mathrm{Sim}\mathopen{(}\R^{#1}\mathclose{)}}}
\newcommand{\similitude}[1]{\ensuremath{\mathfrak{sim}\mathopen{(}\R^{#1}\mathclose{)}}}
\newcommand{\Diff}[2]{\mathopen{}\mathrm{D}_{#1}#2}

\newcommand{\Tan}[2]{\ensuremath{\mathrm{T}_{#1}#2}}
\newcommand{\Fitan}[1]{\ensuremath{\mathrm{T}#1}}

\newcommand{\RP}[1]{\ensuremath{\R\mathbf{P}^{#1}}}
\newcommand{\cc}{\mathcal{C}}
\newcommand{\torus}{\mathbf{t}}
\newcommand{\affine}{\mathbf{a}}
\newcommand{\Xa}{\mathbf{X}_a}
\newcommand{\X}{\mathbf{X}}
\newcommand{\G}{\mathbf{G}}
\newcommand{\g}{\slR{3}}
\newcommand{\h}{\mathfrak{h}}
\newcommand{\Pmin}{\mathbf{P}_{min}}
\newcommand{\pmin}{\mathfrak{p}_{min}}
\newcommand{\Ealpha}{E^\alpha}
\newcommand{\Etildealpha}{\tilde{E}^\alpha}
\newcommand{\Exalpha}{\mathcal{E}^\alpha}
\newcommand{\Ebeta}{E^\beta}
\newcommand{\Etildebeta}{\tilde{E}^\beta}
\newcommand{\Exbeta}{\mathcal{E}^\beta}

\newcommand{\Etildec}{\tilde{E}^c}
\newcommand{\Exc}{\mathcal{E}^c}
\newcommand{\Sm}{\mathcal{S}}
\newcommand{\Stilde}{\tilde{\mathcal{S}}}
\newcommand{\Sx}{\mathcal{S}_Y}
\newcommand{\Lm}{\mathcal{L}}
\newcommand{\Ltilde}{\widetilde{\mathcal{L}}}
\newcommand{\Lx}{\mathcal{L}_\X}

\newcommand{\piun}[1]{\pi_1({#1})}
\newcommand{\Calpha}{\mathcal{C}^\alpha}
\newcommand{\Cbeta}{\mathcal{C}^\beta}
\newcommand{\Salphabeta}{\mathcal{S}_{\alpha,\beta}}
\newcommand{\Sbetaalpha}{\mathcal{S}_{\beta,\alpha}}

\newcommand{\Ztilde}{\tilde{\mathcal{Z}}}

\newcommand{\PSLk}{\mathrm{PSL}_2^{(k)}\mathopen{(}\R\mathclose{)}}
\newcommand{\Ktot}{\mathcal{K}^{tot}}
\newcommand{\K}{\mathcal{K}}

\newcommand{\e}{\mathrm{e}}

\makeatletter
\DeclareRobustCommand*{\lfaktor}[3][]
{
   \ensuremath{ \mathpalette{\mfaktor@impl@}{{#1}{#2}{#3}} }
}
\newcommand*{\mfaktor@impl@}[2]{\mfaktor@impl#1#2}
\newcommand*{\mfaktor@impl}[4]{
   \settoheight{\faktor@zaehlerhoehe}{\ensuremath{#1#2{#3}}}%
   \settoheight{\faktor@nennerhoehe}{\ensuremath{#1#2{#4}}}%
      \raisebox{-0.5\faktor@zaehlerhoehe}{\ensuremath{#1#2{#3}}}%
      \mkern-4mu\diagdown\mkern-5mu%
      \raisebox{0.5\faktor@nennerhoehe}{\ensuremath{#1#2{#4}}}%
}
\makeatother

\AtBeginDocument{}

\DeclareMathOperator{\Image}{Im}
\DeclareMathOperator{\Ker}{Ker}
\DeclareMathOperator{\Vect}{Vect}
\DeclareMathOperator{\rg}{rk}
\DeclareMathOperator{\Aut}{Aut}

\DeclareMathOperator{\tr}{tr}
\DeclareMathOperator{\id}{id}

\DeclareMathOperator{\Der}{Der}
\DeclareMathOperator{\Nor}{Nor}
\DeclareMathOperator{\Cent}{C}
\DeclareMathOperator{\Ad}{Ad}
\DeclareMathOperator{\ad}{ad}
\DeclareMathOperator{\Lie}{Lie}
\DeclareMathOperator{\Stab}{Stab}
\DeclareMathOperator{\ev}{ev}
\DeclareMathOperator{\Mat}{Mat}
\DeclareMathOperator{\gr}{gr}
\DeclareMathOperator{\Kill}{Kill}

\DeclareMathOperator{\Rec}{Rec}

\DeclareMathOperator{\Lin}{\mathsf{L}}

\numberwithin{equation}{section}
\numberwithin{figure}{section}

\title[Contact partially hyperbolic diffeomorphisms]{Partially hyperbolic diffeomorphisms and Lagrangian contact structures}
\author{Martin Mion-Mouton}
\date{February 24, 2020 (last revision: \today).}
\address{
Martin Mion-Mouton,
IRMA,
7 rue René Descartes 
67084 Strasbourg.
}
\email{martin.mionmouton@math.unistra.fr}

\hypersetup{hidelinks,
pagebackref=true,pdftoolbar=true
} 

\begin{document}

\maketitle

\begin{abstract} 
 In this paper, we classify the three-dimensional partially hyperbolic diffeomorphisms whose
 stable, unstable, and central distributions $E^s$, $E^u$, and $E^c$ are smooth, such that $E^s\oplus E^u$ is a contact distribution,         
 and whose non-wandering set equals the whole manifold. 
 We prove that up to a finite quotient or a finite power, they are smoothly conjugated
 either to the time-one map of 
 an algebraic contact-Anosov flow,
 or to an affine partially hyperbolic automorphism of a nil-manifold.
 The rigid geometric structure induced by the invariant distributions 
 plays a fundamental role in the proof. 
\end{abstract}


\section{Introduction}

In a lot of natural situations, a
differentiable dynamical system on a smooth manifold preserves a geometric structure on the tangent bundle,
defined by invariant distributions.
For instance, if it preserves a Borel measure, then Oseledet's theorem
provides an almost-everywhere defined splitting of the tangent bundle, 
given by the rates of expansion or contraction of the tangent vectors by the differentials of the dynamics.
\par Although invariant geometric structures naturally arise,
they are in general highly non-regular (Oseledet's decomposition is for instance only measurable), and
this lack of regularity allows a lot of flexibility of the dynamics:
former examples can be deformed in order to produce a lot of new ones.
In contrast, the smoothness of the invariant distributions puts a strong restriction on the system, and
the known examples with smooth (\emph{i.e.} $\cc^\infty$) distributions are in general ``very symmetric'':
typically, they arise from compact quotient of Lie groups, with action by affine automorphisms.
\par It is thus natural to ask to what extent the geometric structure preserved by the dynamics
makes the situation rigid, and especially \emph{why}. \\

\par Let us give a paradigmatic example of rigidity  
with the following result of Étienne Ghys concerning three-dimensional Anosov flows
(the statement proved by Ghys in \cite{ghys} is more precise than the one given below).
\begin{theorem}[\cite{ghys}]
\label{theoremeghys}
 Let $(\varphi^t)$ be an Anosov flow of a three-dimensional closed connected manifold.
 If the stable and unstable distributions of $(\varphi^t)$ are smooth, then:
 \begin{itemize}
  \item either $(\varphi^t)$ is smoothly conjugated to the suspension flow of a hyperbolic automorphism of the two-torus,
  \item or $(\varphi^t)$ is smoothly orbit equivalent 
  to a finite covering of the geodesic flow of a compact hyperbolic surface.
 \end{itemize}
\end{theorem}

\par We recall that a smooth non-singular flow $(\varphi^t)$ of a compact manifold $M$ is \emph{Anosov} if its differentials 
preserve two distributions $E^s$ and $E^u$, respectively called the stable and unstable distribution of $(\varphi^t)$,
satisfying $\Fitan{M}=E^s\oplus \R\frac{d\varphi^t}{dt}\oplus E^u$
and such that $E^s$ is \emph{uniformly contracted} by $(\varphi^t)$,
and $E^u$ \emph{uniformly expanded} by $(\varphi^t)$.
\par Under the smoothness assumption of $E^s$ and $E^u$, Ghys notices that the plane distribution $E^s\oplus E^u$ 
can only have two extreme geometrical behaviours:
either it integrates into a foliation, or it is a \emph{contact distribution} (\emph{i.e.} it is locally the kernel of a contact one-form).
In the first case, former results of Plante and Franks
conclude the proof, and lead to the suspension examples.
The work of Ghys in \cite{ghys} is therefore almost entirely devoted to three-dimensional \emph{contact-Anosov flows},
\emph{i.e.} when $E^s$ and $E^u$ are smooth, and $E^s\oplus E^u$ is  contact.
Under these geometrical assumptions, the pair $(E^s,E^u)$ is a \emph{rigid geometric structure} preserved by the Anosov flow, 
which makes the classification possible and leads to the finite coverings of geodesic flows. 
\par In this paper, we investigate the same kind of geometrical rigidity conditions, but for the discrete-time analogs
of Anosov flows that are the \emph{partially hyperbolic diffeomorphisms}. 

\subsection{Principal results}
We refer to \cite{crovisierpotrie} for a very complete introduction to partially hyperbolic diffeomorphisms,
for which we use the following definition.
\begin{definition}\label{definitionph}
A smooth diffeomorphism $f$ of a compact manifold $M$ is \emph{partially hyperbolic} if it preserves
a splitting $\Fitan{M}=E^s \oplus E^u \oplus E^c$ of the tangent bundle into three non-zero 
continuous distributions, 
satisfying the following dynamical conditions with respect to some Riemannian metric on $M$.
\begin{itemize}
\item The \emph{stable distribution} $E^s$ is \emph{uniformly contracted} by $f$, \emph{i.e.}
there is a non-zero integer $N$ such that
for any $x\in M$ and any unit vector $v^s\in E^s(x)$, 
\[
\norme{\Diff{x}{f^N}(v^s)}< 1.
\]
\item The \emph{unstable distribution} $E^u$ is \emph{uniformly expanded} by $f$, \emph{i.e.} uniformly contracted by $f^{-1}$.
\item The splitting is dominated, \emph{i.e.} 
there is a non-zero integer $N$ such that
for any $x\in M$, and any unit vectors $v^s\in E^s(x)$, $v^c\in E^c(x)$, and $v^u\in E^u(x)$, 
\[
\norme{\Diff{x}{f^N}(v^s)}<\norme{\Diff{x}{f^N}(v^c)}<\norme{\Diff{x}{f^N}(v^u)}.
\]
$E^c$ is called the \emph{central distribution}.
\end{itemize}
\end{definition}
The three invariant distributions of a partially hyperbolic diffeomorphism
have in general no reasons to be differentiable,
but we study in this paper the particular case when they are \emph{smooth},
\emph{i.e.} $\cc^\infty$,
and when $E^s\oplus E^u$ is furthermore a \emph{contact distribution}. \\

\par The (non-zero) time maps of the contact-Anosov flows appearing in Ghys Theorem \ref{theoremeghys} give us the first examples 
satisfying these geometrical conditions.
They have the following nice algebraic description (see \cite{ghys} for more details).
Let us denote by $\tilde{A}=\{a^t\}_{t\in\R}$ the one-parameter subgroup of the universal cover
$\SLtilde{2}$ of $\SL{2}$ generated by 
$
\left(
\begin{smallmatrix}
 1 & 0 \\
 0 & -1
\end{smallmatrix}
\right)
\in\mathfrak{sl}_2$.
Then for any cocompact lattice $\Gamma_0$ of $\SLtilde{2}$,
the flow $(R_{a^t})$ of right translations by $\tilde{A}$
on the quotient $\Gamma_0\backslash\SLtilde{2}$ is a finite covering of the geodesic flow of 
a compact hyperbolic surface (up to a constant rescaling of the time by a factor $\frac{1}{2}$),
and is thus Anosov.
Moreover, if a morphism $u\colon\Gamma_0\to\tilde{A}$
is such that the graph-group $\Gamma=\enstq{(\gamma,u(\gamma))}{\gamma\in\Gamma_0}$
acts freely, properly and cocompactly on $\SLtilde{2}$ by the action $(g,a)\cdot x=gxa$,
then $(R_{a^t})$ still induces an Anosov flow of the quotient $\Gamma\backslash\SLtilde{2}$,
which is a time-change of the former one (non-trivial if $u\neq\id$).
We will call these flows the \emph{three-dimensional algebraic contact-Anosov flows}.

\par In contrast, the following algebraic examples
are the time-map of none Anosov flow.
For $(\lambda,\mu)\in{\R^*}^2$,
we consider the automorphism 
\begin{equation}\label{equationdefinitionautomorphismHeis3}
\varphi_{\lambda,\mu}\colon
\begin{pmatrix}
1 & x & z \\
0 & 1 & y \\
0 & 0 & 1
\end{pmatrix}\in\Heis{3}
\mapsto
\begin{pmatrix}
1 & \lambda x & \lambda\mu z \\
0 & 1 & \mu y \\
0 & 0 & 1
\end{pmatrix}\in\Heis{3}
\end{equation}
of the Heisenberg group.
If $\varphi=\varphi_{\lambda,\mu}$, $g\in\Heis{3}$, $\Gamma$ is a cocompact lattice of $\Heis{3}$,
and $g\varphi(\Gamma)g^{-1}=\Gamma$, then $L_g\circ\varphi(\Gamma x)=\Gamma (g\varphi(x))$
is a well-defined diffeomorphism of the \emph{nil-$\Heis{3}$-manifold} $\Gamma\backslash\Heis{3}$.
If we moreover assume that either $\abs{\lambda}<1$ and $\abs{\mu}>1$, or the opposite, then
$L_g\circ \varphi$ is a partially hyperbolic diffeomorphism, whose
invariant distributions are smooth, and such that $E^s\oplus E^u$ is contact (see Paragraph \ref{soussoussectionstructHeis3}).
Concrete examples of cocompact lattices preserved by such automorphisms indeed exist, 
and we will call $L_g\circ \varphi$ a \emph{partially hyperbolic affine automorphism}. \\ 

\par The principal result of this paper is that, 
assuming all points are non-wandering,
there are no other examples than the two families we described precedently.
\begin{theoremintro}\label{theoremeprincipalPHdiffeos}
Let $M$ be a closed, connected and orientable three-dimensional manifold, 
and $f$ be a partially hyperbolic diffeomorphism of $M$ such that
\begin{itemize}
\item the stable, unstable, and central distributions $E^s$, $E^u$ and $E^c$ of $f$ are smooth,
\item $E^s\oplus E^u$ is a contact distribution,
\item and the non-wandering set $NW(f)$ equals $M$. 
\end{itemize}
Then we have the following description.
\begin{enumerate}
\item Either some finite power of $f$ is smoothly conjugated to a non-zero time-map 
 of a three-dimensional algebraic contact-Anosov flow,
\item or $f$ lifts by a smooth covering of order at most 4
to a partially hyperbolic affine automorphism of a nil-$\Heis{3}$-manifold.
\end{enumerate}
\end{theoremintro}

Actually, our geometrical conditions are so rigid that the uniformity of the contraction and the expansion of the diffeomorphism
will be obtained as a byproduct.
\begin{definition}\label{definitionweakly}
We will say that a distribution $E$ of a compact manifold $M$ is \emph{weakly contracted} by a diffeomorphism $f$, if for some 
Riemannian metric on $M$, we have for any $x\in M$: 
\[
\underset{n\to+\infty}{\lim}\norme{\Diff{x}{f^n}\restreinta_E}=0 \text{~or~} \underset{n\to-\infty}{\lim}\norme{\Diff{x}{f^n}\restreinta_E}=0.
\]
\end{definition}
We emphasize that the ``direction'' of weak contraction can \emph{a priori} change from point to point, and that
this notion 
is unchanged when replacing $f$ by $f^{-1}$.
\begin{theoremintro}\label{theoremeoptimal}
Let $M$ be a closed, connected and orientable three-dimensional manifold, endowed with a smooth splitting
$\Fitan{M}=E^\alpha\oplus E^\beta\oplus E^c$ such that $E^\alpha \oplus E^\beta$ is a contact distribution.
Let $f$ be a smooth diffeomorphism of $M$ that preserves this splitting, and such that
\begin{itemize}
 \item each of the distributions $E^\alpha$ and $E^\beta$ is weakly contracted by $f$,
 \item and $f$ has a dense orbit.
\end{itemize}
Then the conclusions of Theorem \ref{theoremeprincipalPHdiffeos} hold.
In particular, $f$ is a partially hyperbolic diffeomorphism.
\end{theoremintro}

Theorem \ref{theoremeprincipalPHdiffeos} 
will directly follow from
Theorem \ref{theoremeoptimal} by an argument of Brin,
as explained in Paragraph \ref{soussectionpreuvetheoremeA} at the end of this paper.
We also give in this paragraph the precise satement of Theorem \ref{theoremeprincipalPHdiffeos}, 
that does not use any domination hypothesis on $E^c$ (see Corollary \ref{corollairesansdomination}).
The rest of the paper is devoted to the proof of Theorem \ref{theoremeoptimal}. \\

\par The classification question for partially hyperbolic diffeomorphisms in dimension three has led to a lot of works
in the recent years, 
and significant progress has been made concerning the general case,
as can be seen for instance in the survey \cite{hammerlindlpotrie}.
Recently, different additional rigidity conditions have also been studied.
\par Carrasco, Pujals and Rodriguez-Hertz obtain in \cite{carrascopujalsrodriguezhertz} a classification result
under the smoothness assumption of invariant distributions.
On the contrary of Theorem \ref{theoremeprincipalPHdiffeos}, no additional geometrical condition is assumed, but the authors assume 
that the differential of the partially hyperbolic diffeomorphism 
is constant when read in the global frame given by three smooth vector fields generating these distributions. 
The geometric structure $(E^s,E^u,E^c)$ defined by such a partially hyperbolic diffeomorphism is in general not rigid,
and their result is obtained through dynamical arguments.
\par Beside the smoothness assumption on invariant distributions, 
Bonatti and Zhang obtain in \cite{bonattizhang} different rigidity results 
in the continuous category, under specific dynamical assumptions. 

\subsection{A rigid geometric structure preserved by partially hyperbolic diffeomorphisms}
\label{soussectionintrocomparaisonGhys}

\par Roughly speaking, a \emph{rigid geometric structure} is a structure with ``few automorphisms''.
More precisely, they are those smooth geometric structures whose Lie algebra of local Killing fields 
(\emph{i.e.} local vector fields whose flow preserves the structure) is everywhere finite-dimensional. 
\par As d'Ambra and Gromov pointed out in \cite{dambragromov}, 
it is natural to believe that 
rigid geometric structures preserved by rich dynamical systems have to be particularly peculiar:
\emph{``one does not expect rigid geometry to be accompanied by rich dynamics''} (\cite[\S 0.3 p.21]{dambragromov}).
It seems thus reasonable to look for classification results in these situations.
The general idea is that rich dynamical properties will imply strong restrictions
on the rigid geometric structure, inducing in return a rigidity of  
the dynamical system itself. 

\par Several rigid geometric structures can be preserved by a contact-Anosov flow $(\varphi^t)$.
First of all, $(\varphi^t)$ always preserves a contact one-form $\alpha$
defined by $\alpha(\frac{d\varphi^t}{dt})=1$ and $\alpha\restreinta_{E^s\oplus E^u}=0$.
If the dimension is $2n+1$,
the induced volume form $\alpha\wedge (d\alpha)^n$ is then preserved by $(\varphi^t)$,
\emph{i.e.} contact-Anosov flows are always \emph{conservative}.
For contact-Anosov flows of any odd dimension, $(\varphi^t)$ moreover preserves a natural linear connection on the tangent bundle,
initially defined by Kanai in \cite{kanai}.
An invariant connection of this kind allowed for example Benoist, Foulon and Labourie to obtain a classification result for 
contact-Anosov flows of any odd dimension in \cite{bfl}. 
\par While these invariant rigid geometric structures require the 
existence of a continuous one-parameter flow,
we study in this paper rigid geometric structures preserved by \emph{discrete-time} dynamics. \\

\par The transition from a flow to a diffeomorphism completely changes the situation.
From a dynamical point of view, 
partially hyperbolic diffeomorphisms of ``contact'' type do not anymore preserve a contact one-form, and
are thus (\emph{a priori}) not conservative (which explains the extra hypothesis on non-wandering points).
From a geometrical point of view, the difficulties that appear
are analog to the ones of a conformal geometry in contrast with a metric geometry,
for example the invariant Kanai connection does not anymore exist.
This situation requires to look for a new rigid geometric structure. 
\par A contact plane distribution is far from being rigid:
according to Darboux's theorem, they are all locally isomorphic.
A single smooth one-dimensional distribution in a contact plane distribution is still not
sufficient to make it rigid.
But if the stable and unstable distributions of the partially hyperbolic diffeomorphism are smooth and of contact sum,
then the \emph{pair} $(E^s,E^u)$ is a rigid geometric structure,
called a \emph{Lagrangian contact structure}. 

\par For this structure, the invariant Kanai connection will be replaced by another type of connection called a 
\emph{Cartan connection}, that defines a \emph{Cartan geometry}
(actually, this Cartan geometry partially appears in \cite{ghys},
but under the disguised form of ``the geometry of second-order ordinary differential equations'').
The strength of Cartan geometries is to link the Lagrangian contact structures with 
the \emph{homogeneous model space} $\X=\PGL{3}/\Pmin$ of complete flags of $\R^3$ 
(where $\Pmin$ is the subgroup of upper-triangular matrices). 
In particular, the
\emph{flat} Lagrangian contact structures, \emph{i.e.} the ones whose \emph{curvature} identically vanishes, 
are locally isomorphic to $\X$ (see Paragraphs \ref{soussoussoussectioncourbureCartangeometrie} and \ref{soussoussectionflatlagcontact}).
The geometry of $\X$ will thus play a prominent role in this paper.
\par In \cite{barbot}, Barbot also studies the geometry of $\X$
and the dynamics of $\PGL{3}$, but with a different approach. 
His purpose is among others to construct Anosov representations in $\PGL{3}$, and compact quotients of open subsets of $\X$.


\subsection{Organization of the paper}
This paper is organised in the following way.
Section \ref{sectionstructlagcontact}
introduces several notions and results about three-dimensional Lagrangian contact structures,
that will be used in the whole paper.
At the end of the paper in Paragraph \ref{soussectionpreuvetheoremeA}, we prove 
Theorem \ref{theoremeprincipalPHdiffeos} from Theorem \ref{theoremeoptimal},
and the rest of the paper is devoted to the proof of Theorem \ref{theoremeoptimal}.
In Section \ref{sectionSlochomsurOmega}, 
we begin this proof by showing that the triplet
$\Sm=(E^\alpha,E^\beta,E^c)$ is \emph{quasi-homogeneous}, \emph{i.e.}
locally homogeneous in restriction to a dense open subset $\Omega$ of $M$, 
and that its isotropy on $\Omega$ is non-trivial.
This implies that the Lagrangian contact structure $(E^\alpha,E^\beta)$ 
is \emph{flat}, \emph{i.e.} that $M$ has a $(\PGL{3},\X)$-structure.
In Section \ref{sectionlocalmodeleSm}, we refine this description,
proving that $\Sm\restreinta_\Omega$ is locally isomorphic
to one of two possible homogeneous models $(Y_\torus,\Sm_\torus)$ or $(Y_\affine,\Sm_\affine)$.
This relies on a technical classification of the underlying infinitesimal model, 
done in Section \ref{sectionclassificationmodeleinfinitesimal}.
A critical step is to show in Section \ref{sectionHYstructsurM}
that the open dense subset $\Omega$ is actually equal to $M$, implying
that $M$ has a $(H,Y)$-structure, with two possible models
$(H_\torus,Y_\torus)$ or $(H_\affine,Y_\affine)$.
We prove in Section \ref{sectionstructKleinienne}
that this $(H,Y)$-structure is complete, 
implying that $(M,\Sm)$ is a compact quotient $\Gamma\backslash Y$ of one of these two models, with $\Gamma$
a discrete subgroup of $H=\Aut(Y)$.
This description 
allows us to conclude the proof of Theorem \ref{theoremeoptimal} in Paragraph \ref{soussectionfinpreuveTheoremeB}.

\subsection*{Conventions and notations}
From now on, every differential geometric object will be supposed to be smooth (\emph{i.e.} $\cc^\infty$) if nothing is precised,
and the manifolds will be supposed to be boundaryless. 
\par The flow of a vector field $X$ is denoted by $(\varphi^t_X)$. 
The Lie algebra of a Lie group $G$ is denoted by $\mathfrak{g}$, 
and for any $v\in\mathfrak{g}$, we denote by $\tilde{v}$ the left-invariant vector field of $G$ generated by $v$.
If $\Theta\colon G\times M\to M$ is a smooth group action (on the left or the right) of $G$ on a manifold $M$, then the 
orbital map of the action at $x\in M$ is denoted by $\theta_x=\Theta(\cdot,x)$,
and we denote by $L_g=\Theta(g,\cdot)$ the translation by $g\in G$ if the action is on the left
(respectively by $R_g$ if the action is on the right). 
For any $v\in\mathfrak{g}$ we denote by $v^\dag$ the \emph{fundamental vector field
of the action generated by $v$}, defined for $x\in M$ by $v^\dag(x)=\Diff{e}{\theta_x}(v)$.

\subsection*{Acknowledgments}
I would like to thank Charles Frances for proposing this subject to me, and for the precious
advices that he offers me.

\section{Three-dimensional Lagrangian contact structures}\label{sectionstructlagcontact}

The rigid geometric structures that will be studied in the rest of this paper are the following.

\begin{definition}\label{definitionstructlagcontact}
A \emph{Lagrangian contact structure $\mathcal{L}$} on a three-dimensional manifold $M$ is a pair 
$\mathcal{L}=(E^\alpha,E^\beta)$ of transverse one-dimensional smooth distributions, such that $E^\alpha\oplus E^\beta$
is a contact distribution.
An \emph{enhanced Lagrangian contact structure $\mathcal{S}$} on $M$ is a triplet $\mathcal{S}=(E^\alpha,E^\beta,E^c)$
of one-dimensional smooth distributions such that $\Fitan{M}=E^\alpha\oplus E^\beta \oplus E^c$, and $E^\alpha\oplus E^\beta$
is a contact distribution. \\
A (local) isomorphism between two Lagrangian contact structures 
is a (local) diffeomorphism that individually preserves the distributions $\alpha$ and $\beta$, and
the (local) isomorphisms of enhanced Lagrangian contact structures preserve in addition the central distribution $E^c$.
\end{definition}

We first define what will be for us the most important
example of three-dimensional Lagrangian contact structure.

\subsection{Homogeneous model space}
\label{soussoussectionespacehomogenemodele}

We will call \emph{projective line} the projection in $\RP{2}$ of a plane of $\R^3$, and
we denote by $\RP{2}_*$ the set of projective lines of $\RP{2}$ (called the \emph{dual projective plane}). 
For any subset $Q$ of $\R^{n+1}$
we denote by $[Q]$ the projection in $\RP{n}$ of the linear subspace of $\R^{n+1}$ generated by $Q$.
\par A \emph{pointed projective line} is
a pair $(m,D)$ with $D\in\RP{2}_*$ and $m\in D$,
and we denote by 
\[
 \X=\enstq{(m,D)}{D\in\RP{2}_*,m\in D}\subset\RP{2}\times\RP{2}_*
\]
the \emph{space of pointed projective lines}.
In other words, $\X$ is the space of complete flags of $\R^3$.
We will denote in the whole paper  by
\[
\G=\PGL{3}
\]
the group of projective transformations of $\RP{2}$.
As the projective action of $\G$ on $\RP{2}$ and $\RP{2}_*$ preserves the incidence relation $m\in D$, 
it induces a natural diagonal action of $\G$ on 
$\X\subset\RP{2}\times\RP{2}_*$.
The action of $\G$ on $\X$ is \emph{transitive}, and
the stabilizer in $\G$ of the base-point $o=([e_1],[e_1,e_2])$ of $\X$ is the subgroup 
\[
\Stab_\G(o)=
\Pmin=
\left\{
\begin{bmatrix}
 * & * & * \\
 0 & * & * \\
 0 & 0 & *
\end{bmatrix}
\right\}
<\G
\]
of upper-triangular matrices. 
From now on, we will identify $\X$ and $\G/\Pmin$ by the orbital map
$\bar{\theta}_o\colon\G/\Pmin\to\X$ at $o$.
The homogeneous space $\X$ is a $\RP{1}$-bundle over $\RP{2}$ and $\RP{2}_*$ through the coordinate projections
\begin{equation}\label{equationpialphaetpibeta}
\pi_\alpha\colon (m,D)\in\X\mapsto m\in\RP{2} \text{~and~} \pi_\beta\colon (m,D)\in\X\mapsto D\in\RP{2}_*.
\end{equation}
For $x=(m,D)\in\X$, we will denote by $\Calpha(x)=\Calpha(m)$ 
(respectively $\Cbeta(x)=\Cbeta(D)$) the fiber of $x$ with respect to $\pi_\alpha$ (resp. $\pi_\beta$),
and we will call it the \emph{$\alpha$-circle} (resp. the \emph{$\beta$-circle}) of $x$.
We denote by
\[
\Exalpha=\Ker(\Diff{}{\pi_\alpha}) \text{~and~} \Exbeta=\Ker(\Diff{}{\pi_\beta}),
\]
the one-dimensional vertical distributions of these bundles,
tangent respectively to the foliations by $\alpha$ and $\beta$-circles.
The sum $\Exalpha\oplus\Exbeta$ is contact
and we will call $\Lx=(\Exalpha,\Exbeta)$ 
the \emph{standard Lagrangian contact structure} of $\X$.
\begin{lemma}
The group $\G$ is the group of automorphisms of the standard Lagrangian contact structure $\Lx$.
In particular, the structure $(\X,\Lx)$ is homogeneous.
\end{lemma}
\begin{proof}
First of all, the action of $\G$ preserves the foliations of $\X$ by $\alpha$ and $\beta$-circles, \emph{i.e.} preserves the 
structure $\Lx=(\Exalpha,\Exbeta)$. 
Conversely, if $f$ is a diffeomorphism of $\X$ that preserves $\Lx$, the fact that $f$ preserves the foliation
by $\alpha$-circles simply means that it induces a diffeomorphism $\bar{f}$ of $\RP{2}$ for which $f$ is a lift through 
the projection $\pi_\alpha$.
As $f$ moreover preserves the foliation by $\beta$-circles, $\bar{f}$ maps any projective line to a projective line. 
This implies that $\bar{f}$ is a projective transformation
according to a classical result of projective geometry
(proved for example in \cite[Theorem 7 p.32]{samuel}), \emph{i.e.} that $f$ is induced by the action of an element of $\G$. 
\end{proof}

\subsection{Lagrangian contact structures as Cartan geometries}\label{soussectiongeomcartanassociee}

We now introduce the Cartan geometries modelled on the homogeneous space $\G/\Pmin$,
and make the link with Lagrangian contact structures.
This notion will be our principal technical tool to deal with Lagrangian contact structures.
We refer the reader to \cite{sharpe} or \cite{capslovak} for further details about Cartan geometries in a more general context.

\subsubsection{Cartan geometries modelled on $\G/\Pmin$}

\begin{definition}\label{definitionSL3PminCartangeometrie}
A \emph{Cartan geometry $\mathcal{C}=(\hat{M},\omega)$ modelled on $\G/\Pmin$} 
on a three-dimensional manifold $M$
is the data 
of a $\Pmin$-principal bundle over $M$ denoted by $\pi\colon\hat{M}\to M$ and called the \emph{Cartan bundle},
together with a $\g$-valued one-form $\omega\colon\Fitan{\hat{M}}\to\g$ on $\hat{M}$ called the
\emph{Cartan connection},
that satisfies the three following properties:
\begin{enumerate}
  \item $\omega$ defines a \emph{parallelism of $\hat{M}$}, \emph{i.e.} 
  for any $\hat{x}\in\hat{M}$, $\omega_{\hat{x}}$ is a linear isomorphism from $\Tan{\hat{x}}{\hat{M}}$ to $\g$,
  \item $\omega$ \emph{reproduces the fundamental vector fields of the right action of $\Pmin$},
  \emph{i.e.} for any $v\in\g$ and $\hat{x}\in\hat{M}$ we have: 
  $v^\dag(\hat{x})=\frac{d}{dt}\restreinta_{t=0}\hat{x}\cdot \e^{tv}=\omega_{\hat{x}}^{-1}(v)$,
  \item and $\omega$ is \emph{$\Pmin$-equivariant}, \emph{i.e.} for any $p\in\Pmin$ and $\hat{x}\in\hat{M}$
  we have: $R_p^*\omega=\Ad(p)^{-1}\circ\omega$ (where $\Ad(p)$ stands for the adjoint action of $p$).
 \end{enumerate}
A \emph{(local) automorphism} $f$ of the Cartan geometry $\mathcal{C}$ 
between two open sets $U$ and $V$ of $M$ is a (local) diffeomorphism from $U$ to $V$ that lifts
to a $\Pmin$-equivariant (local) diffeomorphism $\hat{f}$ between $\pi^{-1}(U)$ and $\pi^{-1}(V)$,
such that $\hat{f}$ preserves the Cartan connexion $\omega$ (\emph{i.e.} $\hat{f}^*\omega=\omega$). 
\end{definition}

\begin{example}\label{exampleflatCartangeometrymodel}
The homogeneous model space $\X$ is endowed with the \emph{Cartan geometry of the model}
$\mathcal{C}_\X=(\G,\omega_\G)$, given by 
the canonical $\Pmin$-bundle $\pi_\G\colon\G\to\G/\Pmin=\X$ over $\X$,
together with the Maurer-Cartan one-form $\omega_\G\colon\Fitan{\G}\to\g$ defined by 
$\omega_\G(\tilde{v})\equiv v$ on the left-invariant vector fields of $\G$).
\end{example}

We consider for the rest of the subsection a Cartan geometry $(M,\mathcal{C})=(M,\hat{M},\omega)$ modelled on $\G/\Pmin$.

\subsubsection{Curvature of a Cartan geometry}\label{soussoussoussectioncourbureCartangeometrie}

The following definition replaces the curvature of a Riemannian metric in the 
case of Cartan geometries.
\begin{definition}\label{definitioncourburegeometriesdeCartan}
The \emph{curvature form} of $\mathcal{C}$ is the $\g$-valued
two-form $\Omega$ of $\hat{M}$ defined by the following relation for two vector fields $X$ and $Y$ on $\hat{M}$:
\begin{equation}\label{equationdefinitionformecourbure}
\Omega(X,Y)=d\omega(X,Y)+[\omega(X),\omega(Y)].
\end{equation}
Thanks to the connection $\omega$, the curvature form $\Omega$ is equivalent to a 
\emph{curvature map} $K\colon\hat{M}\to\Lin(\Lambda^2\g,\g)$ on $\hat{M}$ 
(that we will often simply call the \emph{curvature} of $\mathcal{C}$), 
having values in the vector space of $\g$-valued alternated bilinear maps on $\g$, and
defined by the following relation for $\hat{x}\in\hat{M}$ and $v,w\in\g$:
\begin{equation}\label{equationdefinitionfonctioncourbure}
K_{\hat{x}}(v,w)=\Omega(\omega_{\hat{x}}^{-1}(v),\omega_{\hat{x}}^{-1}(w)).
\end{equation}
We will say that the Cartan geometry $\mathcal{C}$ (or the Cartan connection $\omega$) is \emph{torsion-free} if $K_{\hat{x}}(v,w)\in\pmin$
for any $\hat{x}\in\hat{M}$ and $v,w\in\g$.
\end{definition}

 If $v$ or $w$ is tangent to the fiber of the principal bundle $\hat{M}$,
then the curvature form satisfies $\Omega(v,w)=0$ (this is proved in \cite[Chapter 5 Corollary 3.10]{sharpe}).
As $\omega$ maps the tangent space of the fibers to $\pmin$
(because the fundamental vector fields are $\omega$-invariant),
this implies that the curvature $K(v,w)$ vanishes whenever $v$ or $w$ is in $\pmin$.
As a consequence at any point $\hat{x}\in\hat{M}$, $K_{\hat{x}}$ induces a $\g$-valued alternated bilinear map on $\g/\pmin$,
and we will identify in the sequel $K$ with the induced map 
\begin{equation}\label{equationcourbureinduitegsurpmin}
K\colon\hat{M}\to\Lin(\Lambda^2(\g/\pmin),\g).
\end{equation}
The adjoint action of $\Pmin$ induces a linear left action on $\Lin(\Lambda^2(\g/\pmin),\g)$
defined 
for $p\in\Pmin$ and $K\in\Lin(\Lambda^2(\g/\pmin),\g)$ by
\begin{equation}\label{equationdefinitionleftactionPminW}
p\cdot K \colon u\wedge v \mapsto \Ad(p)\cdot (K(\overline{\Ad}(p)^{-1}\cdot u,\overline{\Ad}(p)^{-1}\cdot v)).
\end{equation}
Using the linear right action of $\Pmin$ on $\Lin(\Lambda^2(\g/\pmin),\g)$ defined by $K\cdot p\coloneqq p^{-1}\cdot K$,
$K$ is $\Pmin$-equivariant
(this is proved in \cite[Chapter 5 Lemma 3.23]{sharpe}), and $K$ is moreover
preserved by any local automorphism $f$ of the Cartan 
geometry (i.e 
$K\circ\hat{f}=K$ for any automorphism).

\subsubsection{Lagrangian contact structure induced by a Cartan geometry}\label{soussoussectionlagcontactduneCartan}

At any point $x\in M$ and for any $\hat{x}\in\pi^{-1}(x)$, we denote
by $i_{\hat{x}}\colon\Tan{x}{M}\to\g/\pmin$ the unique isomorphism satisfying
\begin{equation}\label{equationdefinitionig}
i_{\hat{x}}\circ\Diff{\hat{x}}{\pi}=\overline{\omega}_{\hat{x}},
\end{equation}
where $\overline{\omega}$ denotes the projection of $\omega$ on $\g/\pmin$.
As the adjoint action of $\Pmin$ preserves $\pmin$, it induces a representation 
$\overline{\Ad}\colon\Pmin\to\mathrm{GL}(\g/\pmin)$ on the quotient, and
the equivariance of $\omega$ implies the following relation for any $p\in\Pmin$:
\begin{equation}\label{equationrelationig}
 i_{\hat{x}\cdot p}=\overline{\Ad}(p)^{-1}\circ i_{\hat{x}}.
\end{equation}
This relation shows that any $\overline{\Ad}(\Pmin)$-invariant object on $\g/\pmin$ gives rise, through the isomorphisms $i_{\hat{x}}$,
to a well-defined object on the tangent bundle of $M$.
Let us apply this idea to define a Lagrangian contact structure on $M$ associated to the Cartan geometry $\mathcal{C}$.
We introduce 
\begin{equation}\label{equationbaseg-} 
e_{\alpha}=\left(\begin{smallmatrix}
        0 & 0 & 0 \\
        0 & 0 & 0 \\
        0 & 1 & 0
       \end{smallmatrix}\right),  
 e_{\beta}=\left(\begin{smallmatrix}
        0 & 0 & 0 \\
        1 & 0 & 0 \\
        0 & 0 & 0
       \end{smallmatrix}\right), 
        e_{0}=\left(\begin{smallmatrix}
        0 & 0 & 0 \\
        0 & 0 & 0 \\
        1 & 0 & 0
       \end{smallmatrix}\right),
\end{equation}
defining a basis $(\bar{e}_\alpha,\bar{e}_\beta,\bar{e}_0)$ of $\g/\pmin$, in which the matrix of the adjoint action of 
\[
p=
\begin{pmatrix}
a & x & z \\
0 & a^{-1}b^{-1} & y \\
0 & 0 & b
\end{pmatrix}\in \Pmin
\]
is equal to
\begin{equation}\label{actionadjointePmin}
\Mat_{(\bar{e}_\alpha,\bar{e}_\beta,\bar{e}_0)}(\overline{\Ad}(p))=
\begin{pmatrix}
a^{-2}b^{-1} & 0       & a^{-1}y \\
0                   & ab^2 & -b^2x \\
0                   & 0      & a^{-1}b
\end{pmatrix}.
\end{equation}
In particular, the adjoint action of $\Pmin$ individually preserves the lines $\R\bar{e}_\alpha$ and
$\R\bar{e}_\beta$ of $\g/\pmin$. 
Together with the relation \eqref{equationrelationig}, this shows that for $x\in M$, 
the lines $i_{\hat{x}}^{-1}(\R\bar{e}_\alpha)$ and $i_{\hat{x}}^{-1}(\R\bar{e}_\beta)$ of $\Tan{x}{M}$ do not depend on the lift 
$\hat{x}$ of $x$.
The Cartan geometry $\mathcal{C}$ induces thus two one-dimensional distributions 
$E^\alpha_\mathcal{C}(x)=i_{\hat{x}}^{-1}(\R\bar{e}_\alpha)$ and
$E^\beta_\mathcal{C}(x)=i_{\hat{x}}^{-1}(\R\bar{e}_\beta)$ on $M$, and
the curvature of $\mathcal{C}$ will say when do those distributions define a Lagrangian contact structure.

\begin{lemma}\label{lemmalagcontactinduite}
Any torsion-free Cartan geometry $(M,\mathcal{C})$ modelled on $\G/\Pmin$ induces
a Lagrangian contact structure $(E^\alpha_\mathcal{C},E^\beta_\mathcal{C})$ on the three-dimensional base manifold $M$.
\end{lemma}
\begin{proof}[Sketch of proof]
For $x\in M$, considering a local section of the Cartan bundle over $x$, we can push down by $\pi$ the 
$\omega$-constant vector fields $\tilde{e}_\alpha$
and $\tilde{e}_\beta$ of $\hat{M}$ (characterized by $\omega(\tilde{e}_\varepsilon)\equiv e_\varepsilon$)
to local vector fields $X_\alpha$ and $X_\beta$ of $M$ defined on a neighbourhood of $x$,
that respectively generate the distributions $E^\alpha_\mathcal{C}$ and $E^\beta_\mathcal{C}$.
If $K$ has values in $\pmin$, the identity $\omega([\tilde{e}_\alpha,\tilde{e}_\beta])=[e_\alpha,e_\beta]-K(e_\alpha,e_\beta)$ 
(deduced from Cartan's formula
for the differential of a one-form) implies easily that $[X_\alpha,X_\beta]\notin\Vect(X_\alpha,X_\beta)$ 
in the neighbourhood of $x$, finishing the proof.
\end{proof}

\begin{remark}\label{remarkdistributioncartangeometriemodele}
In the case of the Cartan geometry of the model, it is easy to check that 
$(E^\alpha_{\mathcal{C}_\X},E^\beta_{\mathcal{C}_\X})$ is the standard Lagrangian contact structure $\Lx$ of $\X$.
\end{remark}

\subsection{Normal Cartan geometry of a Lagrangian contact structure}
\label{soussoussectionSL3PminCartangeometries}

Any three-dimensional Lagrangian contact structure is actually
induced by a torsion-free Cartan geometry modelled on $\G/\Pmin$.
This equivalence between three-dimensional Lagrangian contact structures 
and Cartan geometries modelled on $\G/\Pmin$ was discovered by Élie Cartan, who developped this notion and after whom
these geometries are named.

\subsubsection{Equivalence problem for Lagrangian contact structures}

A given three-dimensional Lagrangian contact structure is induced by several Cartan connections, and
to obtain an equivalence between both formulations, 
we have to choose a particular one.
This choice will be done through a \emph{normalisation condition} on the curvature.
Using
the basis $(\bar{e}_\alpha\wedge\bar{e}_0,\bar{e}_\beta\wedge\bar{e}_0,\bar{e}_\alpha\wedge\bar{e}_\beta)$ of 
$\Lambda^2(\g/\pmin)$, we consider
the following four-dimensional subspace of $\Lin(\Lambda^2(\g/\pmin),\g)$:
\begin{equation}\label{equationmodulecourburesnormales}
W_K=\left\{ K\colon \bar{e}_\alpha\wedge\bar{e}_0 \mapsto 
                      \left(\begin{smallmatrix} 0 & 0 & K^\alpha \\
 												  0 & 0 & K_\alpha \\
 												  0 & 0 & 0
 			        \end{smallmatrix}\right),
 			        \bar{e}_\beta\wedge\bar{e}_0 \mapsto 
                      \left(\begin{smallmatrix} 0 & K_\beta & K^\beta \\
 												  0 & 0 & 0 \\
 												  0 & 0 & 0
 			        \end{smallmatrix}\right),
 			        \bar{e}_\alpha\wedge\bar{e}_\beta\mapsto 0	\right\}.											
\end{equation}
The linear action of $\Pmin$ preserves $W_K$, that
will be called the \emph{space of normal curvatures}.
Theorem \ref{theoremeprobequivalencestructlagcontact} below is proved in
\cite[Theorem 3 p.14]{doubrovkomrakov}, where the normalisation condition is explicitely calculated
through Cartan's \emph{method of equivalence}
(see also \cite[Theorem 3.1.14 p.271 and Paragraph 4.2.3]{capslovak} that makes the link with general
parabolic Cartan geometries).
\begin{theorem}[E. Cartan, \cite{doubrovkomrakov}, \cite{capslovak}]\label{theoremeprobequivalencestructlagcontact}
For any Lagrangian contact structure $\mathcal{L}$ on a three-dimensional manifold $M$,
there exists a torsion-free Cartan geometry modelled on $\G/\Pmin$
inducing $\mathcal{L}$ on $M$, and whose curvature map has values in the space $W_K$ of normal curvatures.
Such a Cartan geometry 
is unique (up to action of principal bundle automorphisms covering the identity on $M$), and
will be called 
\emph{the normal Cartan geometry of $\mathcal{L}$}. 
\end{theorem}

Furthermore, if $(M_1,\mathcal{L}_1)$ and $(M_2,\mathcal{L}_2)$ are two three-dimensional 
Lagrangian contact structures, 
and $\mathcal{C}_1$, $\mathcal{C}_2$ are the associated normal Cartan geometries,
then the (local) isomorphisms between
$\mathcal{L}_1$ and $\mathcal{L}_2$ and the (local) isomorphism
between $\mathcal{C}_1$ and $\mathcal{C}_2$ are the same.
This a direct consequence of the unicity of the normal Cartan geometry.
The curvature map $K\colon\hat{M}\to W_K$ of the normal Cartan geometry of a 
three-dimensional Lagrangian contact structure $\mathcal{L}$
will simply be called \emph{the curvature of $\mathcal{L}$}.

\subsubsection{Flat Lagrangian contact structures}\label{soussoussectionflatlagcontact}

The homogeneous model space $(\X,\Lx)$ verifies the following analog of Liouville's theorem.
\begin{theorem}\label{theoremLiouvilleXG}
For any connected open subsets $U$ and $V$ of the homogeneous model space $\X$, 
and any diffeomorphism $f$ from $U$ to $V$ that preserves its standard Lagrangian contact structure $\Lx$,
there exists $g\in\G$ such that $f$ is the restriction to $U$ of the translation by $g$.
\end{theorem}
\begin{proof}
The Maurer-Cartan form $\omega_\G$ satisfies for any tangent vectors $v$ and $w$ the structural equation
 $d\omega_\G(v,w)+[\omega_\G(v),\omega_\G(w)]=0$ (see \cite[§3.3 p.108]{sharpe}), implying that the curvature  
 of the Cartan connection $\omega_\G$ is zero.
 Therefore, the curvature satisfies the normalisation condition of Theorem \ref{theoremeprobequivalencestructlagcontact},
 and $\mathcal{C}_\X$ is a normal Cartan geometry on $\X$ modelled on $\G/\Pmin$ and associated to $\Lx$
 (see Remark \ref{remarkdistributioncartangeometriemodele}).
According to Theorem \ref{theoremeprobequivalencestructlagcontact}, 
any local isomorphism of $\Lx$ between two connected open subset $U$ and $V$ of $\X$ lifts therefore
to a local isomorphism of the Cartan geometry $\mathcal{C}_\X$ between $\pi_\G^{-1}(U)$
and $\pi_\G^{-1}(V)$, and such an automorphism is the left translation by an element of $\G$
according to \cite[Chapter 5 Theorem 5.2]{sharpe}.
\end{proof}

A three-dimensional Lagrangian contact structure $(M,\Lm)$ is \emph{flat} 
if the curvature of the normal Cartan geometry of $\Lm$ vanishes identically.
According to the proof of Theorem \ref{theoremLiouvilleXG}, the model space is flat,
and since this property is local, any Lagrangian contact structure locally isomorphic to $(\X,\Lx)$ is flat.
\par The power of Cartan geometries lies in the converse of this statement:
any flat three-dimensional Lagrangian contact structure $\Lm$ is locally isomorphic to the homogeneous model space
(see \cite[Theorem 5.1 and Theorem 5.2 p. 212]{sharpe}).
There exists in this case an atlas of charts from $M$ to $\X$ 
consisting of local isomorphisms of Lagrangian contact structures from $\Lm$ to $\Lx$,
and whose transition maps are restrictions of left translations 
by elements of $\G$ (according to Theorem \ref{theoremLiouvilleXG}).
A maximal atlas satisfying these conditions
is called a \emph{$(\G,\X)$-structure} on $M$.
Any $(\G,\X)$-structure conversely induces on $M$ a Lagrangian contact structure $\Lm$ 
locally isomorphic to $\Lx$, whose charts are local isomorphisms from $\Lm$ to $\Lx$.

\begin{theorem}\label{theoremecourburenulleimpliqueGXstructure}
Any flat three-dimensional Lagrangian contact structure $(M,\Lm)$ is induced by a $(\G,\X)$-structure on $M$.
\end{theorem}

Denoting by $\pi_M\colon\tilde{M}\to M$ the universal cover of $M$, we recall that any $(\G,\X)$-structure on $M$ is described
by a local diffeomorphism $\delta\colon\tilde{M}\to \X$ called the \emph{developping map},
that is equivariant for a morphism $\rho\colon\piun{M}\to\G$ called the \emph{holonomy morphism} 
(see for example \cite[§3.4 p.139-141]{thurston}).
Moreover for any $g\in\G$, the pair $(g\circ\delta,g\rho g^{-1})$ of developping map and holonomy morphism describes
the same $(\G,\X)$-structure.
The Lagrangian contact structure $\Lm$ induced by a $(\G,\X)$-structure is characterized by: 
$\delta^*\Lx=\pi_M^*\Lm$.

\subsubsection{Harmonic curvature}\label{soussoussectioncourbureharmonique}

For $K\in W_K$ an element of the space of normal curvatures 
defined by
\[
K\colon \bar{e}_\alpha\wedge\bar{e}_0 \mapsto 
                      \left(\begin{smallmatrix} 0 & 0 & K^\alpha \\
 												  0 & 0 & K_\alpha \\
 												  0 & 0 & 0
 			        \end{smallmatrix}\right),
 			        \bar{e}_\beta\wedge\bar{e}_0 \mapsto 
                      \left(\begin{smallmatrix} 0 & K_\beta & K^\beta \\
 												  0 & 0 & 0 \\
 												  0 & 0 & 0
 			        \end{smallmatrix}\right),
 			        \bar{e}_\alpha\wedge\bar{e}_\beta\mapsto 0,
\]
and 
\[
p=
\begin{pmatrix}
a & x & z \\
0 & a^{-1}b^{-1} & y \\
0 & 0 & b
\end{pmatrix}\in\Pmin,
\]
the adjoint action \eqref{actionadjointePmin} of $\Pmin$ 
given in Paragraph \ref{soussoussectionSL3PminCartangeometries} enables to compute
the components $\cdot_\alpha$ and $\cdot_\beta$ of $p\cdot K\in W_K$:
\begin{equation}\label{equationactionPminmodulecourbure}
(p\cdot K)_\alpha=a^{-1}b^{-5}K_\alpha, (p\cdot K)_\beta=a^5b K_\beta.
\end{equation}
These expressions show in particular 
that the two-dimensional subspace $W_H=\{K\in W_K \mid K_\alpha=K_\beta=0\}$ of $W_K$ is preserved by the linear
action of $\Pmin$.

\begin{proposition}\label{theoremcourbureharmoniquelagcontact}
If the curvature map of a three-dimensional Lagrangian contact structure $\mathcal{L}$ has values in the subspace $W_H$
(\emph{i.e.} if $K_\alpha$ and $K_\beta$ identically vanish), 
then $\mathcal{L}$ is flat.
\end{proposition}

The following remark will be useful in the proof of this result: 
$\slR{3}$ is a \emph{two-graded} Lie algebra, the graduation being defined by the following subspaces $({\g})_i$ for $i=-2,\cdots,2$:
\begin{equation}\label{equationgraduationsl3}
\slR{3}=
\begin{pmatrix}
({\g})_0 & ({\g})_1 & ({\g})_2 \\
({\g})_{-1} & ({\g})_0 & ({\g})_1 \\
({\g})_{-2} & ({\g})_{-1} & ({\g})_0
\end{pmatrix}.
\end{equation}
The graduation property of $\slR{3}$ simply means that for any $i$ and $j$ we have $[({\g})_i,({\g})_j]\subset ({\g})_{i+j}$,
(where $({\g})_i=\{0\}$ for any $\abs{i}>2$).
This graduation of $\slR{3}$ gives rise to a filtration defined by ${\g}^i=\oplus_{j\geq i} ({\g})_j$, with respect to which
$\slR{3}$ is a filtered Lie algebra, \emph{i.e.} $[{\g}^i,{\g}^j]\subset {\g}^{i+j}$ (with ${\g}^i={\g}$ for $i\leq -2$ and ${\g}^i=\{0\}$ for $i>2$).

\begin{proof}[Proof of Proposition \ref{theoremcourbureharmoniquelagcontact}]
Let $(M,\hat{M},\omega)$ be a normal Cartan geometry modelled on $\G/\Pmin$.
We introduce the following basis of $\slR{3}$:

$e_{0}=\left(\begin{smallmatrix}
        0 & 0 & 0 \\
        0 & 0 & 0 \\
        1 & 0 & 0
       \end{smallmatrix}\right),
e_{\alpha}=\left(\begin{smallmatrix}
        0 & 0 & 0 \\
        0 & 0 & 0 \\
        0 & 1 & 0
       \end{smallmatrix}\right),  
       e_{\beta}=\left(\begin{smallmatrix}
        0 & 0 & 0 \\
        1 & 0 & 0 \\
        0 & 0 & 0
       \end{smallmatrix}\right),
       e_1=\left(\begin{smallmatrix}
        1 & 0 & 0 \\
        0 & -1 & 0 \\
        0 & 0 & 0
       \end{smallmatrix}\right),
e_2=\left(\begin{smallmatrix}
        0 & 0 & 0 \\
        0 & 1 & 0 \\
        0 & 0 & -1
       \end{smallmatrix}\right),
e^{\alpha}=\left(\begin{smallmatrix}
        0 & 0 & 0 \\
        0 & 0 & 1 \\
        0 & 0 & 0
       \end{smallmatrix}\right),
       e^{\beta}=\left(\begin{smallmatrix}
        0 & 1 & 0 \\
        0 & 0 & 0 \\
        0 & 0 & 0
       \end{smallmatrix}\right),
       e^0=\left(\begin{smallmatrix}
        0 & 0 & 1 \\
        0 & 0 & 0 \\
        0 & 0 & 0
       \end{smallmatrix}\right)
       $,
that we denote $\mathcal{B}$.
We denote the coordinate of the Cartan connection $\omega$ with respect to an element $e$ of the basis $\mathcal{B}$
as a real-valued one-form $\omega_e$ on $\hat{M}$, such that $\omega=\sum_{e\in\mathcal{B}}\omega_e e$.
In the same way, the curvature form $\Omega$ of $\omega$ will be denoted as $\Omega=\sum_{e\in\mathcal{B}}\Omega_e e$,
where the $\Omega_e$'s are real-valued two-forms on $\hat{M}$.
According to the form \eqref{equationmodulecourburesnormales} of the curvature map stated in 
Theorem \ref{theoremeprobequivalencestructlagcontact}, if $K_\alpha=K_\beta=0$ identically, then
the only non-zero two-form $\Omega_e$ is 
$\Omega^0=K^\alpha \omega_\alpha\wedge\omega_0+K^\beta \omega_\beta\wedge\omega_0$.
The Bianchi identity proved in \cite[Chapter 5 Lemma 3.30]{sharpe} gives
$d\Omega=[\Omega,\omega]$,
where $[\Omega,\omega]=L\circ (\Omega\wedge\omega)$ with $L\colon v\otimes w\in\g\otimes\g \mapsto [v,w]\in\g$
(see \cite[Chapter 1.5 p.61]{sharpe} for this definition).
The graduation property of $\slR{3}$ exposed in the beginning of the paragraph implies $[e^0,\g^+]=\{0\}$,
and we have the following Lie brackets relations between the elements of $\mathcal{B}$:
$[e^0,e_0]=e_1+e_2$, $[e^0,e_\alpha]=e^\beta$, $[e^0,e_\beta]=-e^\alpha$, $[e^0,e_1]=[e^0,e_2]=-e^0$.
We finally obtain the following equalities by projecting the Bianchi identity to $\R e^\beta$ and $\R e^\alpha$:
\[
0=-K^\alpha \omega_\alpha\wedge\omega_0\wedge\omega_\beta, 0=K^\beta \omega_\beta\wedge\omega_0\wedge\omega_\alpha.
\]
As $(\omega_\alpha,\omega_0,\omega_\beta)$ is at each point $\hat{x}\in\hat{M}$ a basis of the dual space
$(\omega^{-1}_{\hat{x}}(({\g})_{-2}\oplus({\g})_{-1})))^*$,
the three-form $\omega_\alpha\wedge\omega_0\wedge\omega_\beta$ does not vanish,
and the above equalities imply therefore $K^\alpha=K^\beta=0$ identically, \emph{i.e.} $K=0$ as announced.
\end{proof}

\begin{remark}
 The components $K_\alpha$ and $K_\beta$ of the curvature actually encode the \emph{harmonic curvature}
 of a normal Cartan geometry modelled on $\G/\Pmin$, that is known to be the only obstruction to the flatness for parabolic Cartan geometries.
 With this point of view, the above Proposition \ref{theoremcourbureharmoniquelagcontact} is the manifestation
 in the specific case of Lagrangian contact structures of a general phenomena arising for any parabolic geometry
 (see for example \cite[Theorem 3.1.12]{capslovak}).
\end{remark}

\subsubsection{Normal generalized Cartan geometry of an enhanced Lagrangian contact structure}
\label{soussoussectiongeomcartangeneralisee}

Let $\mathcal{S}=(E^\alpha,E^\beta,E^c)$ be an enhanced Lagrangian contact structure
on a three-dimensional manifold $M$, and $\mathcal{C}=(\hat{M},\omega)$ be
the normal Cartan geometry 
of the underlying Lagrangian contact structure $(E^\alpha,E^\beta)$.
Using the isomorphisms $i_{\hat{x}}$ defined in \eqref{equationdefinitionig},
the transverse distribution $E^c$ is encoded by the map
\[
\varphi\colon \hat{x}\in\hat{M} \mapsto i_{\hat{x}}(E^c_{\pi(\hat{x})})\in\mathbb{V},
\]
having values in the open subset 
\[
\mathbb{V}=\enstq{L\in\mathbf{P}(\g/\pmin)}{L\not\subset\Vect(\bar{e}_\alpha,\bar{e}_\beta)}
\]
of $\mathbf{P}(\g/\pmin)$.
Endowing $\mathbb{V}$ with the right $\Pmin$-action
defined by $L\cdot p=\overline{\Ad}(p)^{-1}(L)$, $\varphi$ is $\Pmin$-equivariant.
Conversely, 
any $\Pmin$-equivariant application $\varphi\colon\hat{M}\to \mathbb{V}$
defines a 
transverse distribution $E^c_{\pi(\hat{x})}=i_{\hat{x}}^{-1}(\varphi(\hat{x}))$
compatible with the Lagrangian contact structure $(\Ealpha,\Ebeta)$.

\begin{definition}\label{definitiongeomCartangeneralisee}
$(\mathcal{C},\varphi)=(\hat{M},\omega,\varphi)$ will be called 
\emph{the normal generalized Cartan geometry} 
of the enhanced Lagrangian contact structure $\mathcal{S}$. 
\end{definition}

\subsection{Killing fields of Lagrangian contact structures}

\subsubsection{Some classical properties of Killing fields}\label{soussoussectionproprietesstructlagcontactgeneralisees}

A \emph{(local) Killing field} of a Lagrangian contact structure $(M,\Lm)$ is a (local) vector field $X$ of $M$ whose flow 
preserves $\Lm$. The Killing fields of an enhanced Lagrangian contact structure $\Sm$ are defined in the same way.
We will denote by $\mathfrak{Kill}(U,\Lm)$ 
the subalgebra of Killing fields of $\Lm$ defined on an open subset $U\subset M$,
and by $\mathfrak{kill}^{loc}_\Lm(x)$ 
the Lie algebra of germs of Killing fields of $\Lm$ defined on a neighbourhood of $x$.
\par The following statement summarizes important properties of Killing fields, coming from their description 
through Cartan geometries and well-known in this context.
The results are stated for Lagrangian contact structures, but are true as well for enhanced Lagrangian contact structures.

\begin{lemma}\label{lemmaautlocauxactionlibresurMhat}
Let $M$ be a three-dimensional connected manifold endowed with a Lagrangian contact structure $\Lm$, and
$\mathcal{C}=(\hat{M},\omega)$ be a normal Cartan geometry on $M$ associated to $\Lm$.
\begin{enumerate}
\item If $\hat{f}$ is a $\Pmin$-equivariant diffeomorphism of $\hat{M}$ that covers $\id_M$ and preserves $\omega$,
then $\hat{f}=\id_{\hat{M}}$.
If $\hat{X}$ is a $\Pmin$-invariant vector field on $\hat{M}$ whose flow preserves $\omega$ and whose projection on $M$ vanishes,
then $\hat{X}=0$.
As a consequence,
the lift of a local automorphism $f$ (respectively Killing field $X$) of $\Lm$ to a $\Pmin$-equivariant diffeomorphism $\hat{f}$ of $\hat{M}$
that preserves $\omega$ (resp. to a $\Pmin$-invariant vector field $\hat{X}$ on $\hat{M}$ whose flow preserves $\omega$),
is \emph{unique}.
\item If the lift $\hat{X}$ of a Killing field $X$ of $\Lm$ 
vanishes at some point $\hat{x}$, then $X=0$.
In other words, the linear map
$X\in\mathfrak{Kill}(M,\Lm)\mapsto\omega_{\hat{x}}(\hat{X}_{\hat{x}})\in\g$
is injective.
\item The Lie algebra morphism $X\in\mathfrak{Kill}(M,\Lm)\mapsto [X]_x\in\mathfrak{kill}^{loc}_\Lm(x)$
sending a Killing field of $\Lm$ to its germ at $x$ is injective.
\end{enumerate}
\end{lemma}
\begin{proof}[Sketch of proof]
1. The first assertion is a direct consequence of 
\cite[Proposition 1.5.3]{capslovak} for Cartan geometries modelled on $\G/\Pmin$, and implies the second one. \\
2. Let us assume that a local automorphism $\hat{f}$ of $\mathcal{C}$ fixes a point $\hat{x}\in\hat{M}$.
Then as $\hat{f}$ preserves the parallelism defined by $\omega$, a classical argument implies that $\hat{f}$ is trivial 
on the connected component of $\hat{x}$. This remark easily implies the assertion about Killing fields. \\
3. According to \cite[Lemma 7.1]{bader}, a local automorphism that is 
trivial in the neighbourhood of $x$ is trivial on the connected component of its domain of definition that contains $x$.
This result easily implies the statement concerning Killing fields.
\end{proof}

\begin{remark}\label{remarkouvertminimalexistencedeschampsdeKillinglocaux}
The second statement of the previous lemma shows in particular that for any connected open neighbourhood $U$ of $x\in M$,
the dimension of $\mathfrak{Kill}(U,\Lm)$ is bounded from above by $\dim\g=8$.
Therefore, if we consider a decreasing sequence of connected open neighbourhoods $U_i$ of $x$
such that $\cap_i U_i=\{x\}$, then $\dim \mathfrak{Kill}(U_i,\Lm)$
is constant for $i$ large enough.
This proves the existence of a connected
open neighbourhood $U$ of $x$ such that
\[
X\in\mathfrak{Kill}(U,\Lm)\mapsto [X]_x\in\mathfrak{kill}^{loc}_\Lm(x)
\]
is a Lie algebra isomorphism.
\end{remark}

The following Lemma is the translation of Theorem \ref{theoremLiouvilleXG} for Killing fields of $(\X,\Lx)$.
\begin{lemma}\label{lemmadescriptionkillinggeomCartanplatemodel}
\begin{enumerate}
\item At any point $x\in\X$, the Lie algebra of local Killing fields of $\Lx$ at $x$ is identified with $\g$ through 
the fundamental vector fields of the action.
In other words, the application
$v\in\g\mapsto [v^\dag]_x\in \mathfrak{kill}^{loc}_{\Lx}(x)$ 
sending $v\in\g$ to the germ of $v^\dag$ at $x$,
is an anti-isomorphism of Lie algebras.
\item Any local Killing field of $(\X,\Lx)$ defined on a connected neighbourhood of a point $x\in \X$ 
is the restriction of a global Killing field defined on $\X$.
In other words,
$X\in\mathfrak{Kill}(\X,\Lx)\mapsto [X]_x\in\mathfrak{kill}^{loc}_{\Lx}(x)$
is a Lie algebra isomorphism.
\end{enumerate}
\end{lemma}
\begin{proof}
1. If $v^\dag$ is trivial in the neighbourhood of $x$, then for any $t\in\R$, $\e^{tv}$ acts trivially on an open neighbourhood of $x$.
But the action of $\G$ on $\X$
is \emph{analytic}:
if $g$ and $h$ in $\G$ have the same action on some non-empty open subset of $\X$, then $g=h$ 
(because the
linear subspace generated by the pre-image in $\R^3$ of a non-empty open subset of $\RP{2}$ is equal to $\R^3$).
Therefore, $\e^{tv}=\id$ for any $t\in\R$ and $v=0$.
The application $v\mapsto[v^\dag]_x$ is thus injective, 
and as $\dim\mathfrak{kill}^{loc}_{\Lx}(x)\leq \dim\g$
according to the third assertion of Lemma \ref{lemmaautlocauxactionlibresurMhat}, it is an isomorphism.
Finally, $v\mapsto v^\dag$ is known to be an anti-morphism of Lie algebras. \\
2. Any local Killing field at $x$ is the restriction of $v^\dag$ for some $v\in\g$ according to the first assertion,
and extends therefore to a Killing field defined on $\X$.
\end{proof}

\subsubsection{Total curvature map of an enhanced Lagrangian contact structure}\label{soussoussectiontotalcurvature}

Let $(\mathcal{C},\varphi)=(\hat{M},\omega,\varphi)$ be
the normal Cartan geometry 
of a three-dimensional enhanced Lagrangian contact structure $(M,\Sm)$.
With $K\colon\hat{M}\to W_K$ the curvature map of $\mathcal{C}$, we define
 the curvature map 
\[
\K\coloneqq (K,\varphi) \colon\hat{M}\to W_{\K}\coloneqq W_K\times \mathbb{V},
\]
of the enhanced Lagrangian contact structure $(M,\Sm)$,
which is $\Pmin$-equivariant for the right diagonal action of $\Pmin$ on $W_\K$.

\par If $W$ is any manifold endowed with a right action of $\Pmin$, 
we define $B(W)\coloneqq\{(w,l) \mid w\in W,l\in\Lin(\g,\Tan{w}{W})\}$
(this is a vector bundle over $W$), that we endow with the right $\Pmin$-action 
$(w,l)\cdot p=(w\cdot p,\Diff{w}{R_p}\circ l\circ \Ad(p))$.
For any smooth $\Pmin$-equivariant map $\psi\colon\hat{M}\to W$,
we define a $\Pmin$-equivariant map $D^1\psi\colon\hat{M}\to B(W)$ encoding the differential of $\psi$ as follows:
$D^1\psi(\hat{x})=(\psi(\hat{x}),\Diff{\hat{x}}{\psi}\circ\omega_{\hat{x}}^{-1})$.
We also define inductively $B^{k+1}(W)=B(B^k(W))$ and $D^{k+1}\psi=D(D^k\psi)\colon\hat{M}\to B^{k+1}(W)$ 
for any $k\in\N$ (with $B^0(W)=W$ and $D^0\phi=\phi$). 

\par Denoting $m=\dim\g=8$,
we define
$W_{\Ktot}\coloneqq B^m(W_{\K})$, and the \emph{total curvature}
\[
\Ktot\coloneqq D^m\K\colon\hat{M}\to W_{\Ktot}
\]
of the enhanced Lagrangian contact structure $\Sm$.
The total curvature $\Ktot$ is $\Pmin$-equivariant and preserved by local automorphisms of $\Sm$
(\emph{i.e.} for any such local automorphism $f$ we have $\Ktot\circ\hat{f}=\Ktot$).
We also define for $k\in\N^*$
the space of Killing generators of order $k$ by
$\Kill^k(\hat{x})=\omega_{\hat{x}}(\Ker(\Diff{\hat{x}}{D^{k-1}\K}))\subset\g$, and the space of
Killing generators of total order
by $\Kill^{tot}(\hat{x})=\Kill^{m+1}(\hat{x})=\omega_{\hat{x}}(\Ker(\Diff{\hat{x}}{\Ktot}))\subset\g$. 

\subsubsection{Gromov's theory}\label{soussoussectiontheoriedeGromov}

The \emph{integrability locus} of $\hat{M}$ is defined as the set $\hat{M}^{int}$ of those points $\hat{x}\in\hat{M}$ such that 
for any $v\in\Kill^{tot}(\hat{x})$, there exists a local Killing field $X$ of $\Sm$ defined around $\pi(\hat{x})$
and such that $\omega_{\hat{x}}(\hat{X}_{\hat{x}})=v$.
It is easy to check that $\hat{M}^{int}$ is a $\Pmin$-equivariant set, and we define the 
\emph{integrability locus} of $M$ as $M^{int}=\pi(\hat{M}^{int})$.

\begin{theorem}[Integrability theorem]\label{theoremintegrabilite}
Let $(M,\Sm)$ be a three-dimensional  enhanced Lagrangian contact structure of total curvature $\Ktot$,
and $\hat{M}$ be its normal Cartan bundle.
Then the integrability locus $\hat{M}^{int}$ of $\hat{M}$ is equal to the set of points $\hat{x}\in\hat{M}$ 
where the rank of $\Diff{\hat{x}}{\Ktot}$ is locally constant.
In particular, $\hat{M}^{int}$ is open and dense, and so is the integrability locus $M^{int}$ of $M$.
\end{theorem}

Gromov investigates in \cite{gromov} the integration of ``jets'' of Killing fields for very general 
rigid geometric structures, and proves results related to the above Theorem.
In the case of three-dimensional enhanced Lagrangian contact structures, 
the equivalence with normal generalized Cartan geometries allows to avoid 
the notion of jets of Killing fields, replaced by the one of Killing generators of total order.
In this setting, Theorem \ref{theoremintegrabilite} is a consequence of \cite[Theorem 4.19]{pecastaing}.
We use here a modification of the statement of Pecastaing proved by Frances in \cite[Theorem 2.2]{frances}.
The proof of the statement of Frances
for \emph{generalized} Cartan geometries is straightforward by following the lines of the proof he does for Cartan geometries,
and using \cite[Lemma 4.20 and Lemma 4.9]{pecastaing}.
%

\section{Quasi homogeneity and flatness of the structure}\label{sectionSlochomsurOmega}

From now on and until Paragraph \ref{soussectionpreuvetheoremeA}, 
we are under the hypotheses of Theorem \ref{theoremeoptimal}
and we adopt its notations. 
$M$ is thus a three-dimensional compact connected and orientable manifold,
$\mathcal{S}=(E^\alpha,E^\beta,E^c)$ is an enhanced Lagrangian contact structure on $M$,
and we denote by $\Lm=(\Ealpha,\Ebeta)$ its underlying Lagrangian contact structure.
Finally, $f$ is an orientation-preserving automorphism of $(M,\Sm)$ such that:
\begin{itemize}
 \item each of the distributions $E^\alpha$ and $E^\beta$ is weakly contracted by $f$ (see Definition \ref{definitionweakly}),
 \item and $f$ has a dense orbit.
\end{itemize}
In particular, the non-wandering set $NW(f)=NW(f^{-1})$ equals $M$.
We recall that in this case, the set $\Rec(f)$ (respectively $\Rec(f^{-1})$)
of recurrent points of $f$ (respectively $f^{-1}$) is a dense $G_\delta$-subset of $M$.
Therefore, $\Rec(f)\cap\Rec(f^{-1})$ is dense in $M$ as well.

\subsection{Quasi homogeneity of the enhanced Lagrangian contact structure}

At a point $x\in M$, we introduce the subalgebra
\begin{equation}\label{equationalgebreisotropie}
\mathfrak{is}^{loc}_\Sm(x)=\enstq{X\in\mathfrak{kill}^{loc}_\Sm(x)}{X(x)=0}
\end{equation}
of local Killing fields vanishing at $x$, that we call the \emph{isotropy subalgebra} of $\Sm$.

\begin{definition}\label{definitionautlocorbitetlocalementhom}
The $\Kill^{loc}$-orbit (for $\Sm$, respectively $\Lm$)
of a point $x\in M$ is the set of points that can be reached from $x$
by flowing along finitely many local Killing fields of $\Sm$ (respectively $\Lm$).
An enhanced Lagrangian contact structure $(M,\Sm)$ (resp. a Lagrangian contact structure $(M,\Lm)$) is 
\emph{locally homogeneous} if any connected component of $M$ is a $\Kill^{loc}$-orbit.
\end{definition}

The first claim of the following Proposition is a consequence of Gromov's ``open-dense orbit theorem'', 
and the second one is a reformulation
in the context of enhanced Lagrangian contact structures 
of a work done by Frances in \cite[Proposition 5.1]{frances} for pseudo-Riemannian structures.
\begin{proposition}\label{corollairepremiereetapepreuve}
There exists an open and dense subset $\Omega$ of $M$,
such that the enhanced Lagrangian contact structure $\Sm$ is locally homogeneous
in restriction to $\Omega$.
Moreover for any $x\in\Omega$, the isotropy subalgebra $\mathfrak{is}^{loc}_\Sm(x)$ is non-trivial.
\end{proposition}
\begin{proof}
Since $\Sm$ has an automorphism $f$ with a dense orbit, 
Gromov's dense orbit theorem directly implies the first claim (see \cite[Corollary 3.3.A]{gromov},
and \cite[Theorem 4.13]{pecastaing} for a proof in the case of generalized Cartan geometries).
Since the integrability locus $M^{int}$ is open and dense (see Theorem \ref{theoremintegrabilite}),
and $\Rec(f)\cap\Rec(f^{-1})$ is dense in $M$,
there finally exists a point $x\in\Omega\cap M^{int}\cap \Rec(f)\cap \Rec(f^{-1})$. We show now that 
$\mathfrak{is}^{loc}_\Sm(x)$ is non-zero. 
\par Let us denote by $(\hat{M},\omega,\varphi)$ the normal 
generalized Cartan geometry of $\Sm$ (see Definition \ref{definitiongeomCartangeneralisee}),
and choose a lift $\hat{x}\in\pi^{-1}(x)$ in the Cartan bundle.
Possibly replacing $f$ by $f^{-1}$, we have $\underset{n\to+\infty}{\lim}\norme{\Diff{x}{f^n}\restreinta_{E^\alpha}}=0$, and
by hypothesis on $x$, there exists a strictly increasing sequence $n_k$ of integers such that $f^{n_k}(x)$ converges to $x$,
implying the existence of a sequence $p_k\in\Pmin$ such that $\hat{f}^{n_k}(\hat{x})\cdot p_k^{-1}$ converges to $\hat{x}$.
We claim that the sequence $\hat{f}^{n_k}(\hat{x})$ has to leave every compact subset of $\hat{M}$,
implying that $p_k$ also leaves every compact subset of $\Pmin$.
In fact, if not, some subsequence $(\hat{f}^{n'_k}(\hat{x}))$ converges in $\hat{M}$, 
implying that $(\hat{f}^{n'_k})$ converges to some diffeomorphism of $\hat{M}$ for the $\cc^\infty$-topology, 
because $\hat{f}$ preserves the parallelism defined by $\omega$
(see \cite[Theorem \MakeUppercase{\romannumeral1}.3.2]{kobayashi}). 
Therefore, $(f^{n'_k})$ also converges for the $\cc^\infty$-topology to some diffeomorphism of $M$, 
contradicting 
$\underset{k\to+\infty}{\lim}\norme{\Diff{x}{f^{n'_k}}\restreinta_{E^\alpha}}=0$. 
\par The sequel of the proof of \cite[Proposition 5.1]{frances} will enable us to conclude, using the total curvature 
$\Ktot\colon\hat{M}\to W_{\Ktot}$
of the generalized Cartan geometry associated to $\Sm$ (see Paragraph \ref{soussoussectiontheoriedeGromov}).
By $\Pmin$-equivariance of the total curvature and its invariance by automorphisms,
$p_k\cdot \Ktot(\hat{x})=\Ktot(\hat{f}^{n_k}(\hat{x})\cdot p_k^{-1})$
converges to $\Ktot(\hat{x})$.
The manifold $W_{\Ktot}$ has a canonical structure of algebraic variety for which the action of $\Pmin$ is algebraic
(because its action on the space $W_K$ of normal curvatures and on the algebraic variety 
$\mathbb{V}\subset\mathbf{P}(\g/\pmin)$ are algebraic, see \cite[Remark 4.16]{pecastaing} for more details).
Therefore, the orbits of the action of $\Pmin$ on $W_{\Ktot}$ are locally closed, and are thus imbedded submanifolds. 
In particular, there exists a sequence $\varepsilon_k\in \Pmin$ converging to the identity and such that
$p_k\cdot\Ktot(\hat{x})=\varepsilon_k\cdot\Ktot(\hat{x})$,
\emph{i.e.} such that $\varepsilon_k^{-1}p_k\in\Stab_{\Pmin}(\Ktot(\hat{x}))$.
As $\varepsilon_k^{-1}p_k$ leaves every compact subset of $\Pmin$,
$\Stab_{\Pmin}(\Ktot(\hat{x}))<\Pmin$ is non-compact.
But $\Stab_{\Pmin}(\Ktot(\hat{x}))$ is an algebraic subgroup of $\Pmin$
and has thus a finite number of connected components, finally implying that
its identity component is also non-compact.
\par There exists thus a non-zero vector $v\in\pmin$ in the Lie algebra of $\Stab_{\Pmin}(\Ktot(\hat{x}))$. 
For any $t\in\R$ we have by hypothesis
$\Ktot(\hat{x}\cdot\exp(tv))=\Ktot(\hat{x})\cdot \exp(tv)=\Ktot(\hat{x})$, and
deriving this equality at $t=0$ we obtain 
$\Diff{\hat{x}}{\Ktot}(\omega_{\hat{x}}^{-1}(v))=0$, \emph{i.e.} $v\in\omega_{\hat{x}}(\Ker(\Diff{\hat{x}}{\Ktot}))=\Kill^{tot}(\hat{x})$.
As $\hat{x}$ is in the integrability locus $\hat{M}^{int}$ of $\hat{M}$, 
there exists a local Killing field $X\in\mathfrak{kill}_\Sm^{loc}(x)$
such that $\omega_{\hat{x}}(\hat{X}_{\hat{x}})=v\neq 0$, implying in particular that $X\neq 0$ and
$X(x)=0$, \emph{i.e.} that $X\in\mathfrak{is}_\Sm^{loc}(x)\setminus\{0\}$.
\par The isotropy subalgebra at any point $y\in\Omega$ being linearly isomorphic to the one at $x$
because $\Omega$ is an $\Aut^{loc}$-orbit,
$\mathfrak{is}^{loc}_\Sm(y)$ is finally non-zero at any point $y\in\Omega$, which finishes the proof of the corollary.
\end{proof}

%

\subsection{Flatness of the Lagrangian contact structure}\label{section3dhomlagcontact}

In particular, the underlying Lagrangian contact structure $\Lm=(\Ealpha,\Ebeta)$ is also locally homogeneous 
with non-zero isotropy in restriction to the open and dense subset $\Omega$. 
The following result due to Tresse in \cite{tresse} (see also \cite[\S 4.5.2]{gapphenomenon})
implies that $\Lm\restreinta_\Omega$ is flat.

\begin{theorem}[Tresse \cite{tresse}]\label{theorem3dlochomlagcontact}
 Any three-dimensional locally homogeneous connected Lagrangian contact structure with non-zero isotropy is flat.
\end{theorem}

By density of $\Omega$ and continuity of the curvature, the Lagrangian contact structure $(M,\Lm)$ is therefore flat, and
according to Paragraph \ref{soussoussectionflatlagcontact}, we obtain the following.

\begin{corollary}\label{corollaryLGXstructure}
The Lagrangian contact structure $\Lm$ is
described by a $(\G,\X)$-structure on $M$.
\end{corollary}

The rest of this paragraph is devoted to give a self-contained proof of Tresse's Theorem \ref{theorem3dlochomlagcontact}. 
We consider a locally homogeneous Lagrangian contact structure $\Lm$ with non-zero isotropy defined on a three-dimensional
connected manifold $M$.
We denote by $(M,\mathcal{C})=(M,\hat{M},\omega)$ the normal Cartan geometry 
of $\mathcal{L}$, and by $K\colon\hat{M}\to W_K$ its curvature map.
Choosing $x\in M$ and $\hat{x}\in\hat{M}$, it suffices to prove that $K(\hat{x})=0$ by local homogeneity of $\mathcal{C}$. 
We will denote by
\[
\mathfrak{h}=\mathfrak{kill}^{loc}_\mathcal{L}(x) \text{~and~} \mathfrak{i}=\mathfrak{is}^{loc}_\mathcal{L}(x)
\]
the algebra of local Killing fields of $\Lm$ at $x$ and its isotropy subalgebra.
As $\Lm$ is locally homogeneous, $\ev_x(\mathfrak{h})\coloneqq\enstq{X(x)}{X\in\mathfrak{h}}=\Tan{x}{M}$,
and in particular $\dim\mathfrak{h}-\dim\mathfrak{i}=3$.
The following result gives us a sufficient condition for the vanishing of the curvature.
\begin{lemma}\label{lemmaconditionisotropieannulationcourbure}
Let $f$ be a local automorphism of a locally homogeneous three-dimensional Lagrangian contact structure $(M,\Lm)$
fixing a point $x\in M$, let $\hat{x}\in\pi^{-1}(x)$ be a lift of $x$ in the normal Cartan bundle of $\Lm$,
and let $p\in\Pmin$ be the 
\emph{holonomy of $\hat{f}$ at $\hat{x}$}, characterized by
$
\hat{f}(\hat{x})=\hat{x}\cdot p^{-1}.
$
If $p=\exp(v)$ with
\begin{equation}\label{equationformeholonomiesannulantK}
v=
\begin{pmatrix}
a & * & * \\
0 & -a-b & * \\
0 & 0 & b
\end{pmatrix}\in\pmin
\text{such that~}
b\neq -5a \text{~and~} a\neq -5b,
\end{equation}
then $\Lm$ is flat.
\end{lemma}
\begin{proof}
Since the curvature $K$ is preserved by $\hat{f}$
and $\Pmin$-equivariant (see paragraph \ref{soussoussectionSL3PminCartangeometries}), 
we obtain $p\cdot K(\hat{x})=K(\hat{x}\cdot p^{-1})=K(\hat{f}(\hat{x}))=K(\hat{x})$,
where the holonomy $p$ is of the form
\[
p=\begin{pmatrix}
\lambda & * & * & \\
0 & \lambda^{-1}\mu^{-1} & * \\
0 & 0 & \mu
\end{pmatrix}
\]
with 
$\mu\neq\lambda^{-5}$ and $\lambda\neq\mu^{-5}$ by hypothesis.
According to the expression of the components $(p.K)_\alpha$ 
and $(p.K)_\beta$ of the curvature given in \eqref{equationactionPminmodulecourbure}, we have
$\lambda\mu^5K(\hat{x})_\alpha=K(\hat{x})_\alpha$ and $\lambda^5\mu K(\hat{x})_\beta=K(\hat{x})_\beta$,
implying $K(\hat{x})_\alpha=K(\hat{x})_\beta=0$. 
The  structure being locally homogeneous and the subspace $W_H=\enstq{K\in W_K}{K_\alpha=K_\beta=0}$ being $\Pmin$-invariant,
$K$ has values in $W_H$ on a neighbourhood of $\hat{x}$,
and therefore $K=0$ on this neighbourhood according to Proposition \ref{theoremcourbureharmoniquelagcontact}.
By local homogeneity, $\Lm$ is flat.
\end{proof}

We introduce the Cartan subalgebra $\mathfrak{a}\simeq\R^2$ of diagonal matrices of $\pmin$,
and the projection  $p\colon\pmin\to\mathfrak{a}$ on $\mathfrak{a}$ parallel to $\heis{3}$,
which is a Lie algebra morphism.
The following linear map will play an important role in the proof:
\begin{equation*}\label{equationdefinitionphipreuvetresse}
\phi\colon X\in\mathfrak{i}\mapsto p(\omega_{\hat{x}}(\hat{X}_{\hat{x}}))\in\mathfrak{a}.
\end{equation*}
\begin{fact}\label{factconditionsuffisanteannulationcourbure}
If there exists $X\in\mathfrak{i}$ such that $\phi(X)$ satisfies the hypotheses \eqref{equationformeholonomiesannulantK}
of Lemma \ref{lemmaconditionisotropieannulationcourbure}, then $\Lm$ is flat.
\end{fact}
\begin{proof}
We have the following relation for any $t\in\R$
\begin{equation}\label{equationrelationholonomieduflotdevdag}
\varphi^t_{\hat{X}}(\hat{x})=\hat{x}\cdot \exp(t\omega_{\hat{x}}(\hat{X}_{\hat{x}})).
\end{equation}
Denoting by $p(t)$ the element of $\Pmin$ such that $\varphi^t_{\hat{X}}(\hat{x})=\hat{x}\cdot p(t)$, 
$\{p(t)\}_{t\in\R}$ is a one-parameter subgroup.
There exists thus $w\in\pmin$ such that $p(t)=\exp(tw)$, and deriving the relation $\varphi^t_{\hat{X}}(\hat{x})=\hat{x}\cdot \exp(tw)$
at $t=0$ we obtain $\omega_{\hat{x}}(\hat{X}_{\hat{x}})=w$ (because the Cartan connection $\omega$ reproduces the 
fundamental vector fields of the action of $\Pmin$).
There exists thus an automorphism $\varphi$ of $(M,\Lm)$ fixing $x$ and such that 
$\hat{\varphi}(\hat{x})=\hat{x}\cdot \exp(\omega_{\hat{x}}(\hat{X}_{\hat{x}}))^{-1}$.
As $\phi(X)=p(\omega_{\hat{x}}(\hat{X}_{\hat{x}}))$ satisfies
the conditions \eqref{equationformeholonomiesannulantK}, $\omega_{\hat{x}}(\hat{X}_{\hat{x}})$ also does, and  
Lemma \ref{lemmaconditionisotropieannulationcourbure} implies that $\Lm$ is flat.
\end{proof}

\begin{factnonnumerote}
If $\Ker(\phi)\neq\{0\}$ then $\Lm$ is flat.
\end{factnonnumerote}
\begin{proof}
There exists then $X\in\mathfrak{i}$ such that $v\coloneqq\omega_{\hat{x}}(\hat{X}_{\hat{x}})\in\heis{3}=(\g)^1$, \emph{i.e.}  
\[
v=
\left(\begin{smallmatrix}
0 & a & c \\
0 & 0 & b \\
0 & 0 & 0
\end{smallmatrix}\right)\neq 0.
\]
\par We first assume that $(a,b)\neq(0,0)$. 
For an element of the form
$w=\left(\begin{smallmatrix}
0 & 0 & 0 &\\
a' & 0 & 0 \\
0 & b' & 0
\end{smallmatrix}\right)
$ in $\g$, we have 
$[v,w]=
\left(
\begin{smallmatrix}
aa' & * & 0 \\
0 & bb'-aa' & * \\
0 & 0 & -bb'
\end{smallmatrix}
\right)
$, and as $a\neq0$ or $b\neq0$, there exists such an element $w\in\g$ satisfying
$[v,w]=
\left(
\begin{smallmatrix}
1 & * & 0 \\
0 & -1 & * \\
0  & 0 & 0 
\end{smallmatrix}
\right)$ 
or 
$[v,w]=
\left(
\begin{smallmatrix}
0 & * & 0 \\
0 & 1 & * \\
0  & 0 & -1 
\end{smallmatrix}
\right)$.
As $\Lm$ is locally homogeneous, there exists a Killing field $Y\in\mathfrak{h}$ such that $Y_o=\Diff{\hat{x}}{\pi}(\omega_{\hat{x}}^{-1}(w))$,
implying $\omega_{\hat{x}}(\hat{Y}_{\hat{x}})=w+w_0$ with $w_0\in\pmin=(\g)^0$.
We now use the relation 
\begin{equation}\label{equationchampsdeKilling}
\omega([\hat{X},\hat{Y}])=-[\omega(\hat{X}),\omega(\hat{Y})]+K(\omega(\hat{X}),\omega(\hat{Y})).
\end{equation}
verified for any Killing fields of the Cartan geometry $\mathcal{C}$, that will be proved at the end of this demonstration.
We obtain $\omega_{\hat{x}}([\hat{X},\hat{Y}]_{\hat{x}})=-[v,w]+[v,w_0]+K(v,w+w_0)\in\pmin$, 
where $[v,w_0]\in(\g)^1$ according to the filtration property of 
$\g$, and $K(v,w+w_0)=0$ because $v\in\pmin$ (see Paragraph \ref{soussoussectionSL3PminCartangeometries}).
In particular $[X,Y]\in\mathfrak{i}$, and
$\phi([X,Y])$ is equal to one of the diagonal matrices $[1,-1,0]$ or $[0,1,-1]$, that both satisfy
the condition \eqref{equationformeholonomiesannulantK}.
Therefore $\Lm$ is flat according to Fact 
\ref{factconditionsuffisanteannulationcourbure}. 
\par If $a=b=0$, we can find an element $w\in\g$ such that 
$
[v,w]=
\left(
\begin{smallmatrix}
1 & 0 & 0 \\
0 & 0 & 0 \\
0  & 0 & -1 
\end{smallmatrix}
\right)
$, and by the same argument as above we find $Y\in\mathfrak{h}$ such that $[X,Y]\in\mathfrak{i}$ and 
$\phi([X,Y])
=
\left(
\begin{smallmatrix}
1 & 0 & 0 \\
0 & 0 & 0 \\
0  & 0 & -1 
\end{smallmatrix}
\right)$. As this element of $\mathfrak{a}$ satisfies the conditions \eqref{equationformeholonomiesannulantK}, 
$\Lm$ is flat by Fact \ref{factconditionsuffisanteannulationcourbure}. 
\par We now prove the relation \eqref{equationchampsdeKilling} for two Killing fields $X$ and $Y$ of the Cartan geometry $\mathcal{C}$.
Since the flow of $X$ preserves $\omega$, the Lie derivative $L_X\omega$ vanishes identically, and
applying Cartan's formula 
$L_X=d\circ\iota_X+\iota_X\circ d$ to $Y$, we obtain 
$Y\cdot\omega(X)+d\omega(X,Y)=0$.
Cartan's formula 
$d\omega(X,Y)=X\cdot\omega(Y)-Y\cdot\omega(X)-\omega([X,Y])$ implies
$X\cdot\omega(Y)=\omega([X,Y])$, and
as $L_Y\omega=0$ as well, we also have $-Y\cdot\omega(X)=\omega([X,Y])$.
Equation \eqref{equationchampsdeKilling} then follows from
the definition of the curvature.
\end{proof}
\begin{factnonnumerote}
If $\phi(\mathfrak{i})=\mathfrak{a}$ then $\Lm$ is flat.
\end{factnonnumerote}
\begin{proof}
There exists in this case a Killing field $X\in\mathfrak{i}$ such that 
$\phi(X)=
\left(
\begin{smallmatrix}
1 &  0 & 0 \\
0 & 0 & 0 \\
0 & 0 & -1 
\end{smallmatrix}
\right)
$, which satisfies the hypotheses \eqref{equationformeholonomiesannulantK}, implying that $\Lm$ is flat according 
to Fact \ref{factconditionsuffisanteannulationcourbure}.
\end{proof}

It remains to handle the case when $\phi$ is injective, and $\phi(\mathfrak{i})$ is one-dimensional. 
There exists then $V\in\mathfrak{i}$ such that $\mathfrak{i}=\R V$, 
and we can moreover assume without lost of generality that $v\coloneqq \omega_{\hat{x}}(\hat{V}_{\hat{x}})\in\pmin$
does not verify the condition \eqref{equationformeholonomiesannulantK} (if it does, then $\Lm$ is flat
according to Fact \ref{factconditionsuffisanteannulationcourbure}).
In other words, denoting the components of $v$ in $\mathfrak{a}$ by
\begin{equation*}
\phi(V)=p(v)=
\begin{pmatrix}
a & 0 & 0 \\
0 & -a-b & 0 \\
0 & 0 & b
\end{pmatrix}
\in\mathfrak{a},
\end{equation*}
with $(a,b)\in\R^2$, we assume that
\begin{equation}\label{equationhypotheseab}
\text{either~} a=-5b\neq0, \text{or~} b=-5a\neq0.
\end{equation}
\par Since $v\in\pmin$, the curvature part of the relation \eqref{equationchampsdeKilling} vanishes,
and for any $X\in\mathfrak{h}$ we have:
\begin{equation}\label{equationconjugaisonadV}
\omega_{\hat{x}}(\widehat{[V,X]}_{\hat{x}})=-[v,\omega_{\hat{x}}(\hat{X}_{\hat{x}})].
\end{equation}
The linear map 
\[
\varphi\colon X\in\mathfrak{h}\mapsto \omega_{\hat{x}}(\hat{X}_{\hat{x}})\in\g
\]
is injective according to Lemma 
\ref{lemmaautlocauxactionlibresurMhat}, and as $\ev_x(\mathfrak{h})=\Tan{x}{M}$ by local homogeneity, 
$\varphi$ induces an isomorphism $\bar{\varphi}$ between $\mathfrak{h}/\mathfrak{i}$ and $\g/\pmin$.
Using the notations \eqref{equationbaseg-} in Paragraph \ref{soussoussectionlagcontactduneCartan} for the basis 
$(\bar{e}_\alpha,\bar{e}_\beta,\bar{e}_0)$ of $\g/\pmin$, there exists 
$X$, $Y$, and $Z$ in $\mathfrak{h}$
such that $\varphi(X)\in e_\alpha+\pmin$, $\varphi(Y)\in e_\beta+\pmin$, and $\varphi(Z)\in e_0+\pmin$.
According to \eqref{equationconjugaisonadV}, $\bar{\varphi}$ intertwines the adjoint action of $V$
on $\mathfrak{h}/\mathfrak{i}$ and the adjoint action of $-v$ on $\g/\pmin$, implying
\begin{equation}\label{equationadpmin}
\Mat_{(\bar{X},\bar{Y},\bar{Z})}(\overline{\ad}(V))=
\Mat_{(\bar{e}_\alpha,\bar{e}_\beta,\bar{e}_0)}(\overline{\ad}(-v))=
\begin{pmatrix}
-a-2b & 0 & * \\
0 & 2a+b & * \\
0 & 0 & a-b
\end{pmatrix}.
\end{equation}
\par We will denote by $A=-a-2b$ and $B=2a+b$ the eigenvalues of $\overline{\ad}(V)$ with respect to $\bar{X}$ and $\bar{Y}$.
Our hypotheses \eqref{equationhypotheseab} on $a$ and $b$ imply $A\neq0$ and $B\neq0$,
allowing us to
choose $X$ and $Y$ in $\mathfrak{h}$ satisfying 
\[
[V,X]=AX \text{~and~} [V,Y]=BY.
\]
In fact, if $X\in\mathfrak{h}$ satisfies $\bar{\varphi}(X+\mathfrak{i})=\bar{e}_\alpha$, 
there exists 
$\lambda\in\R$ such that $[V,X]=aX+\lambda V$ according to \eqref{equationadpmin}, and
$X'=X+\frac{\lambda}{A}V$ satisfies then $[V,X']=AX'$. We deal with the case of $Y$ by the same computations.
\par The Jacobi identity yields $[V,[X,Y]]=(A+B)[X,Y]$, implying in particular that $[X,Y]\notin\Vect(X,Y,V)$ since  
$A+B$ is distinct from $A$, $B$ and $0$.
A second application of the same identity gives 
$[V,[X,[X,Y]]]=(2A+B)[X,[X,Y]]$ and $[V,[Y,[X,Y]]]=(A+2B)[Y,[X,Y]]$.
Furthermore, if $[X,[X,Y]]\neq0$, then $2A+B$ is an eigenvalue of $\ad(V)\in\Lin(\mathfrak{h})$,
and is thus equal to one of the eigenvalues $A$, $B$, or $A+B$ (since $\dim\mathfrak{h}=4$ and $2A+B\neq0$).
But the equalities $2A+B=A+B$ or $2A+B=B$ would contradict $A\neq0$, and the equality
$2A+B=A$ would likewise contradict our hypotheses on $a$ and $b$. 
Consequently, $[X,[X,Y]]=0$, and for the same reasons $[Y,[X,Y]]=0$.
\par As a consequence, $\mathcal{E}\coloneqq\Vect(X,Y,[X,Y])$ is a subalgebra of $\h$ isomorphic to $\heis{3}$.
There is a connected open neighbourhood $U$ of $x$ such that the injective linear map
$X\in\mathfrak{Kill}(U,\mathcal{L})\mapsto [X]_x\in\mathfrak{kill}^{loc}_\mathcal{L}(x)$ is an isomorphism
(see Remark \ref{remarkouvertminimalexistencedeschampsdeKillinglocaux}),
and there is thus an injective Lie algebra morphism $\iota\colon\heis{3}\to\mathfrak{Kill}(U,\mathcal{L})$
 of image $\mathcal{E}$.
According to the work of Palais in \cite{palais}, chapter \MakeUppercase{\romannumeral 2} Theorem \MakeUppercase{\romannumeral 11} 
and its corollary, there exists a (unique) local action of $\Heis{3}$ on $U$ that
integrates this infinitesimal action, 
\emph{i.e.} such that $X^\dag=\iota(X)\restreinta_U$ for any $X\in\heis{3}$.
In particular, the local action of $\Heis{3}$ on $U$ preserves $\Lm$, and
as $\iota(\heis{3})\cap\mathfrak{i}=\{0\}$, 
the orbital map at $x$ is a $\Heis{3}$-equivariant embedding.
The Lagrangian contact structure $\Lm$ is thus
locally isomorphic to a left-invariant Lagrangian contact structure on $\Heis{3}$.
The following lemma implies then that $\Lm$ is flat, finishing the proof of Theorem \ref{theorem3dlochomlagcontact}.

\begin{lemma}\label{lemmastructHeis3invarplates}
Any left-invariant Lagrangian contact structure on $\Heis{3}$ is flat.
\end{lemma}
\begin{proof}
The left-invariant Lagrangian contact structure $\mathcal{M}_0=(\R\tilde{X},\R\tilde{Y})$ of $\Heis{3}$ generated by
$X=\left(\begin{smallmatrix}
0 & 1 & 0 \\
0 & 0 & 0 \\
0 & 0 & 0
\end{smallmatrix}\right)$ and
$Y=\left(\begin{smallmatrix}
0 & 0 & 0 \\
0 & 0 & 1 \\
0 & 0 & 0
\end{smallmatrix}\right)$ is flat.
In fact, we will see in Paragraph \ref{soussoussectionYaffine}
that $(\Heis{3},\mathcal{M}_0)$ is isomorphic to an open subset of the homogeneous model space $(\X,\Lx)$.
Considering a left-invariant Lagrangian contact structure $\mathcal{M}$ on $\Heis{3}$, 
it suffices thus to find an isomorphism
of Lagrangian contact structures from $\mathcal{M}_0$ to $\mathcal{M}$ to prove our claim.
\par There exists $v,w\in\heis{3}$ such that $\mathcal{M}=(\R\tilde{v},\R\tilde{w})$,
and as $\R\tilde{v}\oplus\R\tilde{w}$ is a contact distribution, $[v,w]\notin\Vect(v,w)$.
Denoting $Z=\left(\begin{smallmatrix}
0 & 0 & 1 \\
0 & 0 & 0 \\
0 & 0 & 0
\end{smallmatrix}\right)$, $v=aX+bY+cZ$, and $w=a'X+b'Y+c'Z$, we have
$[v,w]=(ab'-ba')Z$, which implies $ab'-ba'\neq 0$.
The Lie algebra automorphism $\varphi$ of $\heis{3}$ whose matrix in the basis $(X,Y,Z)$ is
$
\left(\begin{smallmatrix}
a & a' & 0 \\
b & b' & 0 \\
c & c' & ab'-ba'
\end{smallmatrix}\right)$ sends $(X,Y)$ to $(v,w)$, 
and as $\Heis{3}$ is simply-connected,
there exists a Lie group automorphism $\phi$ of $\Heis{3}$ whose differential at identity is $\varphi$.
Since $\phi$ is an automorphism, $\Diff{e}{\phi}(X,Y)=(v,w)$ implies $\phi^*\tilde{v}=\tilde{X}$ and $\phi^*\tilde{w}=\tilde{Y}$,
\emph{i.e.} $\phi$ is an isomorphism of Lagrangian contact structures from $\mathcal{M}_0$ to $\mathcal{M}$.
\end{proof}

\section{Local model of the enhanced Lagrangian contact structure}\label{sectionlocalmodeleSm}

In the previous section, we proved that the Lagrangian contact structure $\Lm$ is locally isomorphic to the homogeneous
model space $(\X,\Lx)$, and thus described by a $(\G,\X)$-structure on $M$.
The classical strategy is then to reduce the possibilities for the images of the developping map $\delta\colon\tilde{M}\to M$
and of the holonomy morphism $\rho\colon\piun{M}\to \G$ of this structure.
\par In the case studied by Ghys in \cite{ghys} of an Anosov flow preserving the structure,
the holonomy group $\rho(\piun{M})\subset\G$ is centralized by a one-parameter subgroup of $\G$, which
reduce dramatically the possibilities for $\rho(\piun{M})$.
But in the case of a discrete-time dynamics, we do not have any relevant algebraic restriction
of this kind on $\rho(\piun{M})$.
\par For this reason, we have to look not only at the local homogeneity of $\Lm$ on $\Omega$, 
but at the local homogeneity of the whole \emph{enhanced} Lagrangian contact structure 
$\Sm=(E^\alpha,E^\beta,E^c)$ on this open dense subset.
In this section, we will show that in restriction to $\Omega$, 
$\Sm$ is locally isomorphic to an \emph{infinitesimal homogeneous model},
that preserves a distribution transverse to the contact plane.

\subsection{Two algebraic models}\label{soussectionexemplealgebriquesgeometrie}

We begin by describing this models 
in an algebraic way.

\subsubsection{Left-invariant structure on $\SL{2}$}\label{sousoussectionstructSL2}

We will use the following basis for the Lie algebra $\slR{2}$ of $\SL{2}$:
\begin{equation}\label{equationbasecanoniquesl2}
E=
\left(\begin{smallmatrix}
         0 & 1 \\
         0 & 0 
        \end{smallmatrix}\right),
         F=\left(\begin{smallmatrix}
         0 & 0 \\
         1 & 0 
        \end{smallmatrix}\right), \text{~and~}
         H=\left(\begin{smallmatrix}
         1 & 0 \\
         0 & -1 
        \end{smallmatrix}\right).
\end{equation}
The Lie bracket relation $[E,F]=H$ between these three vectors shows that they define
a left-invariant enhanced Lagrangian contact structure
$\mathcal{S}_{\SL{2}}=(\R\tilde{E},\R\tilde{F},\R\tilde{H})$ on $\SL{2}$.
Moreover, the right action of the one-parameter subgroup $A$
generated by $H$ preserves $\mathcal{S}_{\SL{2}}$.
We endow the universal cover $\SLtilde{2}$ of $\SL{2}$
with the pullback of $\Sm_{\SL{2}}$, 
so that 
the right action of the one-parameter subgroup $\tilde{A}$ of $\SLtilde{2}$ generated by $H$ preserves $\Sm_{\SLtilde{2}}$. 
\par Let $\Gamma_0$ be a cocompact lattice of $\SLtilde{2}$, and $u\colon\Gamma_0\to\tilde{A}$ be a morphism
whose graph-group $\Gamma=\enstq{(\gamma,u(\gamma))}{\gamma\in\Gamma_0}$
acts freely, properly, and cocompactly on $\SLtilde{2}$, \emph{via} the action $(g,a)\cdot x=gxa$
(these morphisms are called \emph{admissible} by Salein and studied in detail 
in his thesis \cite{salein}).
Then the standard structure of $\SLtilde{2}$ is preserved by $\Gamma$, 
and $\Gamma\backslash\SLtilde{2}$
is endowed with the induced enhanced Lagrangian contact structure $\Sm$, whose distributions are exactly
the invariant distributions of the algebraic contact-Anosov flow
$(R_{a^t})$ on $\Gamma\backslash\SLtilde{2}$.

\subsubsection{Left-invariant structure on $\Heis{3}$}\label{soussoussectionstructHeis3}

We will use the following basis for the Lie algebra $\heis{3}$ of $\Heis{3}$:
\[
X=\left(\begin{smallmatrix}
0 & 1 & 0 \\
0 & 0 & 0 \\
0 & 0 & 0
\end{smallmatrix}\right),
Y=\left(\begin{smallmatrix}
0 & 0 & 0 \\
0 & 0 & 1 \\
0 & 0 & 0
\end{smallmatrix}\right),
Z=\left(\begin{smallmatrix}
0 & 0 & 1 \\
0 & 0 & 0 \\
0 & 0 & 0
\end{smallmatrix}\right).
\]
According to the Lie bracket relation $[X,Y]=Z$, $\mathcal{S}_{\Heis{3}}=(\R\tilde{X}, \R\tilde{Y}, \R\tilde{Z})$
is a left-invariant enhanced Lagrangian contact structure on $\Heis{3}$.
The subgroup 
\[
\mathcal{A}=\enstq{\varphi_{\lambda,\mu}}{(\lambda,\mu)\in{\R^*}^2} 
\]
of automorphisms introduced in the introduction (see \eqref{equationdefinitionautomorphismHeis3})
is exactly the subgroup of $\Aut(\Heis{3})$ preserving $\mathcal{S}_{\Heis{3}}$. 
\par Any cocompact lattice $\Gamma$ of $\Heis{3}$ preserves $\mathcal{S}_{\Heis{3}}$, 
and the quotient $\Gamma\backslash\Heis{3}$ will always be endowed with the induced enhanced Lagrangian contact structure $\Sm$.
The invariant distributions of  
a partially hyperbolic affine automorphism $L_g\circ \varphi$ of $\Gamma\backslash\Heis{3}$, 
with $g\in\Heis{3}$ and $\varphi\in\mathcal{A}$, 
are precisely given by $\Sm$.

\subsection{Two homogeneous open subsets of the model space}\label{soussectionouvertshomogenes}


The left-invariant structures of $\SL{2}$ and $\Heis{3}$ can be geometrically imbedded in $\X$ as 
homogeneous open subsets, that will be the
local models of the enhanced Lagrangian contact structure $\Sm$ in restriction to $\Omega$.

\subsubsection{Some specific surfaces of $\X$, and one affine chart}\label{soussoussectionsurfacescarteaffine}

For $D$ a projective line of $\RP{2}$, we define the \emph{$\beta-\alpha$ surface} 
\[
\Sbetaalpha(D)=\pi_\alpha^{-1}(D)=\cup_{y\in\Cbeta(D)}\Calpha(y),
\]
and for $m\in\RP{2}$, the analog \emph{$\alpha-\beta$ surface} 
\[
\Salphabeta(m)=\pi_\beta^{-1}(\enstq{L\in\RP{2}_*}{m\in L})=\cup_{y\in \Calpha(m)} \Cbeta(y).
\] 
The open subset
\[
\Omega_a\coloneqq \X\setminus\Sbetaalpha([e_1,e_2])
\]
of $\X$, composed by pointed projective lines $(m,D)$ for which $m\notin[e_1,e_2]$,
will be identified with the set $\X_a$ of pointed affine lines 
of $\R^2$ as follows:
\begin{equation}\label{equationidentifiactionYaXa}
\phi_a\colon(m,D)\in \Omega_a \mapsto (m\cap P, D\cap P)\in\Xa,
\end{equation} 
where $\Vect(e_1,e_2)+(0,0,1)$ is identified with $\R^2$ by translation.
The diffeomorphism $\phi_a$ is moreover equivariant for the canonical identification
\begin{equation}\label{equationphiaequivariant}
\begin{bmatrix}
A & X \\
0 & 1
\end{bmatrix}\in \Stab_\G(\Omega_a)
\mapsto
A+X\in\Aff{2}
\end{equation}
of $\Stab_\G(\Omega_a)$ with the group of affine transformations of $\R^2$.

\subsubsection{The open subset $Y_\torus$}\label{soussoussectionYtorus}

We will embed $\SL{2}$ in $\G$ as follows:
\[
\iota\colon
g\in\SL{2}\mapsto
\begin{bmatrix}
g & 0 \\
0 & 1
\end{bmatrix}\in\G.
\]
The resulting copy $S_0$ of $\SL{2}$ acts simply transitively at 
$o_\torus=([1,0,1],[(1,0,1),e_2])=\phi_a^{-1}(e_1+\R e_2)\in\Omega_a$, and
its orbit $Y_\torus=S_0\cdot o_\torus$ 
can be described as
\[
Y_\torus=\Omega_a\setminus S_{\alpha,\beta}[e_3]=
\phi_a^{-1}\left(
\enstq{m+L}{m\in\R^2\setminus\{(0,0)\},L\in\RP{1}\setminus\{\R m\}}
\right).
\]
The left-invariant structure of $\SL{2}$ induces on $Y_\torus$ 
a $S_0$-invariant enhanced Lagrangian contact structure
\begin{equation}\label{equationdefinitionSt}
\mathcal{S}_\torus=(\theta_{o_\torus}\circ\iota)_*\mathcal{S}_{\SL{2}},
\end{equation}
which is \emph{compatible with $\Lx$}
in the sense that its $\alpha$ and $\beta$-distributions coincide with the ones of $\Lx$, and
whose central distribution is entirely described by its value at $o_\torus$:
\begin{equation}\label{equationdefinitionEcStorus}
\Exc_\torus(o_\torus)=\R H_0^\dag(o_\torus),
\text{where~}
H_0=
\left(
\begin{smallmatrix}
1 & 0 & 0 \\
0 & -1 & 0 \\
0 & 0 & 0 
\end{smallmatrix}
\right).
\end{equation}
\par We denote by $A^{\pm}$ the subgroup of $\SL{2}$ composed by diagonal matrices.
The right action of $A^{\pm}$ preserves $\mathcal{S}_{\SL{2}}$, and the 
direct product $\SL{2}\times A^{\pm}$ acts on $\SL{2}$ by $(g,a)\cdot h=gha$.
The isomorphism from $\SL{2}$ to $(Y_\torus,\Sm_\torus)$
given by the orbital map at $o_\torus$
is equivariant for the identification
\begin{equation}\label{equationrelationequivarianceAutYt}
\left(g,
\left(
\begin{smallmatrix}
\lambda & 0 \\
0 & {\lambda^{-1}}
\end{smallmatrix}
\right)
\right)\in\SL{2}\times A^{\pm}\mapsto 
\lambda g
\in\GL{2}.
\end{equation}
In particular,
\[
H_\torus\coloneqq
\begin{bmatrix}
\GL{2} & 0 \\
0 & 1 
\end{bmatrix}
\]
is contained in the automorphism group of $(Y_\torus,\Sm_\torus)$.

\subsubsection{The open subset $Y_\affine$}\label{soussoussectionYaffine}

The action of $\Heis{3}$ at $o_\affine=([e_3],[e_3,e_2])=\phi_a^{-1}((0,0)+\R e_2)\in\Omega_a$
is simply transitive, 
and its orbit $Y_\affine=\Heis{3}\cdot o_\affine$ can be described as
\[
Y_\affine=\Omega_a\setminus S_{\alpha,\beta}[e_1]=
\phi_a^{-1}\left(
\enstq{m+L}{m\in\R^2,L\in\RP{1}\setminus\{\R e_1\}}
\right).
\] 
We endow $Y_\affine$ with the $\Heis{3}$-invariant enhanced Lagrangian contact 
structure
\begin{equation}\label{equationdefinitionSa}
\mathcal{S}_\affine=(\theta_{o_\affine}\restreinta_{\Heis{3}})_*\mathcal{S}_{\Heis{3}}
\end{equation}
which is compatible with 
$\Lx$, and whose central distribution is entirely determined by
\begin{equation}\label{equationdefinitionEcSaffine}
\Exc_\affine(o_\affine)=\R Z^\dag(o_\affine).
\end{equation}
\par Let us recall that $\mathcal{A}$ is the subgroup of automorphisms of $\Heis{3}$ that moreover preserve $\Sm_{\Heis{3}}$
(see Paragraph \ref{soussoussectionstructHeis3}).
The group of affine automorphisms $L_g\circ\varphi$ of $\Heis{3}$,
where $g\in\Heis{3}$ and $\varphi\in\mathcal{A}$, will be seen as a semi-direct 
subgroup $\Heis{3}\rtimes\mathcal{A}$.
With this notation, the isomorphism
from $(Y_\affine,\Sm_\affine)$ to $\Heis{3}$ given by the orbital map at $o_\affine$
is equivariant for the identification
\begin{equation}\label{equationequivarianceaffine}
\left[\begin{smallmatrix}
\lambda & x & z \\
0 & \lambda^{-1}\mu^{-1} & y \\
0 & 0 & \mu
\end{smallmatrix}\right]
\in\Pmin
\mapsto 
\left(
\left(\begin{smallmatrix}
1 & \lambda\mu x & \mu^{-1}z \\
0 & 1 & \mu^{-1}y \\
0 & 0 & 1
\end{smallmatrix}\right),
\varphi_{\lambda^2\mu,\lambda^{-1}\mu^{-2}}
\right)\in\Heis{3}\rtimes\mathcal{A},
\end{equation}
and in particular,
$H_\affine\coloneqq\Pmin$ is contained in the automorphism group of $(Y_\affine,\Sm_\affine)$. 

\subsection{From the infinitesimal model to the local model}

We take back the notations of Theorem \ref{theoremeoptimal}.
We recall that $\pi_M\colon\tilde{M}\to M$ denotes the universal cover of $M$
and that $\Omega$ is a dense and open subset of $M$ where $\Sm$ is 
locally homogeneous (see Proposition \ref{corollairepremiereetapepreuve}).
We will denote 
$\tilde{\mathcal{S}}=\pi_M^*\mathcal{S}=(\Etildealpha,\Etildebeta,\Etildec)$,
$\tilde{\Omega}=\pi_M^{-1}(\Omega)$,
and $\delta\colon\tilde{M}\to \X$ a developping map of the $(\G,\X)$-structure of $M$ describing the Lagrangian contact structure 
$\Lm$ (see Corollary \ref{corollaryLGXstructure} and Paragraph \ref{soussoussectionflatlagcontact}).
We finally choose for this whole section a connected component $O$ of $\tilde{\Omega}$,
\emph{i.e.} an open $\Kill^{loc}$-orbit of $\Stilde$.
\par Our goal in this section is to describe the local model of $\Stilde$ in restriction to $O$. 

\subsubsection{Infinitesimal model}

At any point of $\X$, we will identify the Lie algebra of local Killing fields of $\Lx$ with $\g$ through the 
fundamental vector fields of the action of $\G$ (see Lemma \ref{lemmadescriptionkillinggeomCartanplatemodel}).
Since the developping map $\delta$ is a local isomorphism from $\Ltilde$ to $\Lx$, it induces at each point $x\in\tilde{M}$
an isomorphism 
\begin{equation}
\delta^*\colon v\in\g=\mathfrak{kill}^{loc}_{\Lx}(\delta(x)) \mapsto \delta^*v\in\mathfrak{kill}^{loc}_{\Ltilde}(x),
\end{equation}
of Lie algebras,
whose inverse will be denoted by $\delta_*\colon\mathfrak{kill}^{loc}_{\Ltilde}(x)\to \g$.
For $X\in\mathfrak{kill}^{loc}_{\Ltilde}(x)$ and $t\in\R$ for which $\varphi^t_X(x)$ exists,
denoting $v=\delta_*[X]_x\in\g$, we have 
\begin{equation}\label{equationrelationequivariancechampsKilling}
\delta(\varphi^t_X(x))=\e^{tv}\cdot\delta(x).
\end{equation}

\begin{lemma}\label{lemmaalgebrelocaleKillingdeO}
There exists a subalgebra $\mathfrak{h}$ of $\g$
such that
\[
\mathfrak{Kill}(O,\Stilde\restreinta_O)=(\delta^*\mathfrak{h})\restreinta_O=\enstq{(\delta^*v)\restreinta_O}{v\in\mathfrak{h}}.
\]
Moreover, any local Killing field of $\Stilde$ on $O$ extends to the whole $\Kill^{loc}$-orbit $O$.
\end{lemma}
\begin{proof}
It suffices to show that the subalgebra $\mathfrak{h}(x)=\delta_*\mathfrak{kill}^{loc}_{\Stilde}(x)$
is locally constant on $O$.
This will in fact imply by connexity of $O$ that $\mathfrak{h}(x)$ is constant equal to some Lie subalgebra $\mathfrak{h}$ on $O$, and
then $(\delta^*\mathfrak{h})\restreinta_O\subset\mathfrak{Kill}(O,\Stilde\restreinta_O)$.
But for $x\in O$, $\dim\mathfrak{h}=\dim\mathfrak{kill}^{loc}_{\Stilde}(x)\geq\mathfrak{Kill}(O,\Stilde\restreinta_O)$
(see Lemma  \ref{lemmaautlocauxactionlibresurMhat}), and this inclusion is thus an equality.
\par For any $x\in O$ 
there exists an open connected neighbourhood $U$ of $x$ such that any local Killing field 
of $\Stilde$ at $x$ extends to a Killing field defined on $U$ 
(see Remark \ref{remarkouvertminimalexistencedeschampsdeKillinglocaux}),
and for any $y\in U$ we have thus $\mathfrak{h}(x)\subset \mathfrak{h}(y)$.
But $\mathfrak{h}(x)$ and $\mathfrak{h}(y)$ have the same dimension
since $x$ and $y$ are in the same Kill\textsuperscript{loc}-orbit of $\Stilde$, and this inclusion is thus an equality.
This shows that
$\mathfrak{h}(x)$ is locally constant and finishes the proof.
\end{proof}

We denote from now on by $H$ the connected Lie subgroup of $\G$ of subalgebra $\mathfrak{h}$.
It is not necessarily closed in $\G$, but the action of $H$ on $\X$ is smooth for the structure of immersed submanifold of $H$.
\begin{lemma}\label{lemmadeltaOcontenudansY}
All the points of $\delta(O)$ are in the same orbit $Y$ under the action of $H$.
In particular, $Y$ is open.
\end{lemma}
\begin{proof}
We consider $x$ and $y$ in $O$, and we want to find
$h\in H$ such that $\delta(y)=h\cdot \delta(x)$.
By hypothesis, as $x$ and $y$ are in the same Kill\textsuperscript{loc}-orbit of $\Stilde$, there exists a finite number of points 
$x_1=x,\dots,x_n=y$ such that for any $i\leq n-1$ there exists a local Killing field $X_i$ of $\Stilde$ satisfying
$x_{i+1}=\varphi_{X_i}^1(x_i)$.
According to Lemma \ref{lemmaalgebrelocaleKillingdeO}, there exists for each $i$ an element $v_i\in\mathfrak{h}$ such that 
$X_i=\delta^*v_i$, and we have $\delta(x_{i+1})=\e^{v_i}\delta(x_i)$ according to the equation 
\eqref{equationrelationequivariancechampsKilling},
implying $\delta(y)=\e^{v_{n-1}}\dots\e^{v_1}x_0\in H\cdot \delta(x)$.
\end{proof}

We choose from now on a point $x\in O$, we denote $x_0=\delta(x)\in Y$,
and we consider the \emph{isotropy subalgebra} 
\begin{equation}\label{equationdefinitioni}
\mathfrak{i}=\mathfrak{stab}_{\mathfrak{h}}(x_0)\coloneqq\enstq{v\in\mathfrak{h}}{v(x_0)=0} 
\end{equation}
of $\mathfrak{h}$ at $x_0$,
characterized by $\delta^*\mathfrak{i}=\mathfrak{is}^{loc}_{\Stilde}(x)$.
Since the orbit $Y$ of $x_0$ under $H$ 
is open, $\dim\mathfrak{h}-\dim\mathfrak{i}=3$,
and $\mathfrak{i}$ is non-trivial according to  
Proposition \ref{corollairepremiereetapepreuve}.
We also denote $\Exc(x_0)=\Diff{x}{\delta}(\Etildec(x))$,
and $\mathfrak{h}/\mathfrak{i}=D^\alpha\oplus D^\beta\oplus D^c$
the splitting sent to $\Tan{x_0}{Y}=(\Exalpha\oplus\Exbeta\oplus\Exc)(x_0)$
by the isomorphism $\overline{\Diff{e}{\theta_{x_0}}}$ induced by the orbital map at $x_0$.

\begin{lemma}\label{lemmastructH0invarianteYdeltaisolocal}
\begin{enumerate}
\item The adjoint representation $\overline{\ad}\colon\mathfrak{i}\to\Lin(\mathfrak{h}/\mathfrak{i})$ preserves the line $D^c$
in $\mathfrak{h}/\mathfrak{i}$, \emph{i.e.} for any $v\in\mathfrak{i}$ we have $\overline{\ad}(v)(D^c)\subset D^c$.
\item There exists in the neighbourhood of $x_0$ 
an unique $H$-invariant germ of a smooth one-dimensional distribution $\Exc$ that extends $\Exc(x_0)$ on a neighbourhood of $x_0$,
and this distribution is everywhere transverse to $\Exalpha\oplus\Exbeta$.
\item The developping map $\delta$ is an isomorphism between the enhanced Lagrangian contact structures
$\Stilde$ and $\Sx\coloneqq(\Exalpha,\Exbeta,\Exc)$, from a neighbourhood of $x$ to a neighbourhood of $x_0$.
\item $\mathfrak{h}=\mathfrak{kill}^{loc}_{\Sx}(x_0)$ and $\mathfrak{i}=\mathfrak{is}^{loc}_{\Sx}(x_0)$.
\item If $I=\Stab_H(x_0)$ is a connected subgroup of $H$, then there exists an unique $H$-invariant
smooth one-dimensional distribution $\Exc$ that extends $\Exc(x_0)$ on the whole open orbit $Y$, and
$\Exc$ is transverse to $\Exalpha\oplus\Exbeta$.
Furthermore, $\delta\restreinta_O$ is a local isomorphism from $(O,\Stilde\restreinta_O)$ to $(Y,\Sx)$.
\end{enumerate}
\end{lemma}
\begin{proof}
1. For $v\in\mathfrak{i}$, denoting $X=\delta^*v\in\mathfrak{is}^{loc}_{\Stilde}(x)$, 
equation \eqref{equationrelationequivariancechampsKilling}
implies $\Exc(x_0)=\Diff{x_0}{\e^{tv}}(\Exc(x_0))$
for any $t\in\R$, and thus $D^c=\overline{\Ad}(\e^{tv})\cdot D^c=\exp(t\overline{\ad}(v))\cdot D^c$.
Derivating this last equality at $t=0$, we obtain 
$\overline{\ad}(v)\cdot D^c \subset D^c$. \\
2. The group $I=\Stab_{H}(x_0)$ and its identity component $I^0$
are closed in $H$ for its topology of immersed submanifold, and
the orbital map at $x_0$ induces a local diffeomorphism $\bar{\theta}_{x_0}\colon H/I^0\to Y$,
equivariant for the action of $H$.
We saw previously that $\overline{\Ad}(\exp(\mathfrak{i}))$
preserves $D^c$, implying that the subgroup $\enstq{i\in I^0}{\overline{\Ad}(i)\cdot D^c=D^c}$  
is equal to $I^0$ by connexity, \emph{i.e.} that $I^0$ preserves $D^c$.
Therefore, $H/I^0$ supports an unique $H$-invariant smooth one-dimensional distribution extending
$D^c$, that can be pushed by $\bar{\theta}_{x_0}\colon H/I^0\to Y$, 
to a $H$-invariant distribution extending $\Exc(x_0)$ on a neighbourhood of $x_0$.
Conversely, the pullback of any $H$-invariant distribution extending $\Exc(x_0)$ on a neighbourhood of $x_0$ 
is $H$-invariant on $H/I^0$, which proves the unicity of the germ of $\Exc$.
As it is preserved by $H$, it must remain transverse to $\Exalpha\oplus\Exbeta$.  \\
3. For $y$ sufficiently close to $x$, there exists $X\in\mathfrak{Kill}(O,\Stilde\restreinta_O)$ such that $y=\varphi^1_X(x)$.
Denoting $y_0=\delta(y)$ and $v\in\mathfrak{h}$ such that $\delta^*v=X$,
we have
$\Diff{y_0}{\e^{-v}}\circ\Diff{y}{\delta}(\Etildec(y))=\Diff{x}{\delta}\circ\Diff{y}{\varphi^{-1}_X}(\Etildec(y))
=\Exc(x_0)$, implying $\Diff{y}{\delta}(\Etildec(y))=\Exc(y_0)$ by $H$-invariance of $\Exc$. \\
4. This is a direct consequence of $\delta^*\mathfrak{h}=\mathfrak{kill}^{loc}_{\Stilde}(x)$,
$\delta^*\mathfrak{i}=\mathfrak{is}^{loc}_{\Stilde}(x)$, and of the fact that $\delta$ is a local isomorphism from $\Stilde$ 
to $\Sx$ at $x$. \\
5. Concerning the first assertion, the orbital map at $x_0$ induces a $H$-equivariant diffeomorphism from 
$H/I$ to $Y$, and we saw in the proof of the second assertion
that $H/I^0=H/I$ supports an unique $H$-invariant distribution extending $D^c$ on $H/I^0$,
which stays transverse to the contact plane.
\par The set $\mathcal{E}$ of points $y\in O$ such that 
$\delta$ is a local isomorphism in the neighbourhood of $y$ is open and non-empty,
and we only have to prove that $\mathcal{E}$ is closed to conclude by connexity of $O$.
Let $z \in O$ be an adherent point of $\mathcal{E}$, and let us denote $z_0=\delta(z)$.
There exists a point $y\in\mathcal{E}$ sufficiently close to $z$ such that, for some Killing field $X$ of $\Stilde$,
$z=\varphi^1_X(y)$. 
Denoting $v\in\mathfrak{h}$ such that $X=\delta^*v$, we have 
$\Diff{z_0}{\e^{-v}}\circ\Diff{z}{\delta}(\Etildec(z))=\Diff{y}{\delta}\circ\Diff{z}{\varphi^{-1}_X}(\Etildec(z))=\Exc(y_0)$,
implying $\Diff{z}{\delta}(\Etildec(z))=\Exc(z_0)$ by $H$-invariance of $\Exc$.
By local homogeneity of $\Sm\restreinta_O$, we can reach all the points of some neighbourhood $U$ of $z$ in $O$
by a Killing field, and the same computation as before shows that $\delta\restreinta_U$ is a local isomorphism, \emph{i.e.} that
$z\in\mathcal{E}$.
\end{proof}

\subsubsection{Local model of an open $\Kill^{loc}$-orbit}

We will call
\begin{equation}\label{equationdefinitionflipdiffeo}
\kappa\colon (m,D)\in\X \mapsto (D^\bot,m^\bot)\in\X
\end{equation}
the \emph{flip diffeomorphism} of the homogeneous model space.
This involution switches the distributions
$\Exalpha$ and $\Exbeta$ of the standard Lagrangian contact structure, 
and is moreover equivariant for the Lie group morphism $\Theta\colon g\mapsto\transp{g}^{-1}$ of $\G$.
\par Consequently, interverting the distributions $\Ealpha$ and $\Ebeta$ of the Lagrangian contact structure of $M$
is equivalent to composing the developping map $\delta$
with $\kappa$. 
At the level of the subalgebra $\mathfrak{h}$ introduced in the previous paragraph, it is equivalent to apply the Lie algebra morphism 
$\Diff{e}{\Theta}=\theta\colon A\mapsto -\transp{A}$.
\par Denoting
\begin{subnumcases}{\label{equationdefinitionhtorushaffine} \eqref{equationdefinitionhtorushaffine}}
\mathfrak{h}_\torus=
\enstq{\begin{pmatrix}
A & 0 \\
0 & -\tr(A)
\end{pmatrix}
}{A\in\mathfrak{gl}_2},
& \nonumber \\
\mathfrak{h}_\affine=\pmin,
& \nonumber
\end{subnumcases}
we will prove in the next section that:
\begin{proposition}\label{propositionclassificationalgebriquemodelelocalS}
Up to conjugacy in $\G$ or image by $\theta=-\transp{\cdot}$,
$\mathfrak{h}$ is equal to $\mathfrak{h}_\torus$ or $\mathfrak{h}_\affine$.
\end{proposition}

To deduce a local information about $\Stilde\restreinta_O$ from this infinitesimal classification, 
it only remains to look at the action of the connected Lie subgroups 
$H^0_\torus\coloneqq\GLplus{2}$ and $H^0_\affine=\Pmin^+$ of $\G$, of respective Lie algebras 
$\mathfrak{h}_\torus$ and $\mathfrak{h}_\affine$.
\begin{proposition}\label{propositioninformationsYtYa}
\begin{enumerate}
\item $Y_\torus$ (respectively $Y_\affine$) is the only open orbit of $H^0_\torus$
(resp. of $H^0_\affine$) on $\X$.
\item $\Sm_\torus$ (respectively $\Sm_\affine$) is the only $H^0_\torus$-invariant (resp. $H^0_\affine$-invariant)
enhanced Lagrangian contact structure of $Y_\torus$ (resp. $Y_\affine$) that is compatible with $\Lx$.
\end{enumerate}
\end{proposition}
\begin{proof}
1. Both of these groups are contained in 
$\Stab_{\G}[e_1,e_2]=
\{
\left[\begin{smallmatrix}
A & X \\
0 & 1
\end{smallmatrix}\right]
\mid A\in\GL{2},X\in\R^2\}
$, that preserves the surface $S_{\beta,\alpha}[e_1,e_2]$,
and whose only open orbit is thus $\Omega_a=\X\setminus S_{\beta,\alpha}[e_1,e_2]$.
Any open orbit  of one these groups is therefore contained in $\Omega_a$.
Since $H^0_\torus$ preserves the surface $S_{\alpha,\beta}[e_3]$, any open orbit of $H^0_\torus$
is contained in $Y_\torus=\X\setminus(S_{\beta,\alpha}[e_1,e_2]\cup S_{\alpha,\beta}[e_3])=H^0_\torus\cdot o_\torus$. 
In the same way, since $H^0_\affine$ preserves $S_{\alpha,\beta}[e_1]$, any open orbit of $H^0_\affine$
is contained in $Y_\affine=\X\setminus(S_{\beta,\alpha}[e_1,e_2]\cup S_{\alpha,\beta}[e_1])=H^0_\affine\cdot o_\affine$. \\
2. We start with $Y_\torus$, and we denote 
\[
\mathfrak{i}_\torus=\Lie(\Stab_{H^0_\torus}(o_\torus))=\enstq{\left(\begin{smallmatrix}
a & 0 & 0 \\
0 & -2a & 0 \\
0 & 0 & a
\end{smallmatrix}\right)}{a\in\R},
\]
and
\[
E=
\left(\begin{smallmatrix}
         0 & 1 & 0 \\
         0 & 0 & 0 \\
         0 & 0 & 0
        \end{smallmatrix}\right),
         F=\left(\begin{smallmatrix}
         0 & 0 & 0 \\
         1 & 0 & 0 \\
         0 & 0 & 0
        \end{smallmatrix}\right),
         H=\left(\begin{smallmatrix}
         1 & 0 & 0 \\
         0 & -1 & 0 \\
         0 & 0 & 0
        \end{smallmatrix}\right).
\]
The standard Lagrangian contact structure of $\X$ 
satisfies $\R E^\dag(o_\torus)=\Exalpha(o_\torus)$ and $\R F^\dag(o_\torus)=\Exbeta(o_\torus)$, and
for $a\in\R$, the adjoint action of the diagonal element $[a,-2a,a]$ of $\mathfrak{i}_\torus$  
has the following diagonal matrix
in the basis 
$(\bar{E},\bar{F},\bar{H})$ of $\mathfrak{h}/\mathfrak{i}$:
\[ 
 \Mat_{(\bar{E},\bar{F},\bar{H})}(\overline{\ad}([a,-2a,a]))=[3a,-3a,0].
\]
Any line $D^c$ of $\mathfrak{h}_\torus/\mathfrak{i}_\torus$ that is transverse to $\Vect(\bar{E},\bar{F})$ has projective coordinates
$[x,y,1]$ in the basis $(\bar{E},\bar{F},\bar{H})$ for some $(x,y)\in\R^2$, and $\overline{\ad}([a,-2a,a])(D^c)$ is therefore
generated by the vector of coordinates $(3ax,-3ay,0)$.
The only transverse line stabilized by $\overline{\ad}(\mathfrak{i}_\torus)$ is therefore $\R\bar{H}$,
and $\Exc_\torus$ is the only $H^0_\torus$-invariant distribution of $Y_\torus$ transverse to $\Lx$.
\par Let us denote
\[
\mathfrak{i}_\affine=\Lie(\Stab_{H^0_\affine}(o_\affine))=
\enstq{\left(\begin{smallmatrix}
a & 0 & 0 \\
0 & -a-b & 0 \\
0 & 0 & b
\end{smallmatrix}\right)}{(a,b)\in\R^2}.
\]
and
\begin{equation}\label{equationnotationhaffine}
X=
\left(\begin{smallmatrix}
0 & 1 & 0 \\
0 & 0 & 0 \\
0 & 0 & 0 
\end{smallmatrix}\right),
Y=
\left(\begin{smallmatrix}
0 & 0 & 0 \\
0 & 0 & 1 \\
0 & 0 & 0 
\end{smallmatrix}\right),
Z=
\left(\begin{smallmatrix}
0 & 0 & 1 \\
0 & 0 & 0 \\
0 & 0 & 0 
\end{smallmatrix}\right).
\end{equation}
We have $\R X^\dag(o_\affine)=\Exalpha(o_\affine)$, $\R Y^\dag(o_\affine)=\Exbeta(o_\affine)$,
and for $(a,b)\in\R^2$, the adjoint action of the diagonal element $[a,-a-b,b]$ of $\mathfrak{i}_\affine$ 
has the following diagonal matrix
in the basis $(\bar{X},\bar{Y},\bar{Z})$ of $\pmin/\mathfrak{i}_\affine$ 
\begin{equation}\label{equationmatriceactionadjointeiaffine}
\Mat_{(\bar{X},\bar{Y},\bar{Z})}(\overline{\ad}([a,-a-b,b]))=
[2a+b,-a-2b,a-b].
\end{equation}
Any line $D^c$ of $\pmin/\mathfrak{i}_\affine$ that is transverse to $\Vect(\bar{X},\bar{Y})$ has 
projective coordinates of the form $[x,y,1]$ in the basis $(\bar{X},\bar{Y},\bar{Z})$ for some $(x,y)\in\R^2$,
and $\overline{\ad}([a,-a-b,b])(D^c)$ is therefore generated by the vector of coordinates $((2a+b)x,(-a-2b)y,a-b)$.
The only transverse line stabilized by $\overline{\ad}(\mathfrak{i}_\affine)$ is therefore
$\R\bar{Z}$, and  $\Exc_\affine$ is
the only $H^0_\affine$-invariant distribution of $Y_\affine$ transverse to $\Lx$.
\end{proof}

We can finally describe the local geometry of $O$, which is a connected component of $\tilde{\Omega}=\pi_M^{-1}(\Omega)$.
\begin{corollary}\label{propositionmodelelocaldeO}
Up to inversion of the distributions $\Ealpha$ and $\Ebeta$,
the restriction $\delta\restreinta_O$ of the developping map to $O$
is a local isomorphism from $(O,\Stilde\restreinta_O)$ to  $(Y_\torus,\mathcal{S}_\torus)$, 
or to $(Y_\affine,\mathcal{S}_\affine)$.
\end{corollary}
\begin{proof}
The inversion of the distributions $\Ealpha$ and $\Ebeta$ is equivalent to apply $\theta$ to $\mathfrak{h}$,
and the conjugation of $\mathfrak{h}$ by $g\in\G$ is equivalent to replace the developping map $\delta$
by $g\circ\delta$ (that describes the same $(\G,\X)$-structure on $M$).
According to Proposition \ref{propositionclassificationalgebriquemodelelocalS}, we can thus assume that
$\mathfrak{h}$ is equal to $\mathfrak{h}_\torus$ or $\mathfrak{h}_\affine$,
and the open orbit $Y$ is therefore equal to $Y_\torus$
(respectively $Y_\affine$) according to Proposition \ref{propositioninformationsYtYa}.
Since the isotropy subgroups $\Stab_{H^0_\torus}(o_\torus)$ and $\Stab_{H^0_\affine}(o_\affine)$ are connected,
there exists a $H^0_\torus$-invariant (resp. $H^0_\affine$-invariant)
enhanced Lagrangian contact structure $\Sx$ on $Y$ that is compatible with $\Lx$ and such that 
$\delta\restreinta_O$ is a local isomorphism from $(O,\Stilde\restreinta_O)$ to $(Y,\Sx)$
(see Lemma \ref{lemmastructH0invarianteYdeltaisolocal}).
According to Proposition \ref{propositioninformationsYtYa}, $\Sx$ is equal to $\Sm_\torus$ (resp. $\Sm_\affine$).
\end{proof}

\section{Classification of the infinitesimal model}\label{sectionclassificationmodeleinfinitesimal}

The goal of this section is to prove Proposition \ref{propositionclassificationalgebriquemodelelocalS}.
Let us recall that the Lie subalgebras $\mathfrak{i}\subset\mathfrak{h}$ of $\g$ are characterized by
$(\delta^*\mathfrak{h})\restreinta_O=\mathfrak{Kill}(O,\Stilde\restreinta_O)$ 
and $[\delta^*\mathfrak{i}]_x=\mathfrak{is}^{loc}_{\Stilde}(x)$
(see Lemma \ref{lemmaalgebrelocaleKillingdeO} and \eqref{equationdefinitioni}).

\subsection{Algebraic reduction}

We first prove some purely algebraic restrictions on $\mathfrak{h}$.
\begin{lemma}\label{lemmareductiondimensionh}
The dimension of $\mathfrak{h}$ is either 4 or 5.
\end{lemma}
\begin{proof}
Possibly translating the developping map by an element of $\G$, we can assume that $x_0=o=([e_1],[e_1,e_2])\in\X$, and
since the adjoint action of $\Pmin$ on the lines of $\g/\pmin$ transverse to $\Vect(\bar{e}_\alpha,\bar{e}_\beta)$
is transitive (see Paragraph \ref{soussectiongeomcartanassociee}), we can moreover assume that
$D^c=\overline{\Diff{e}{\theta_o}}(\R \bar{e}_0)$
with
$
e_0=
\left(
\begin{smallmatrix}
0 & 0 & 0 \\
0 & 0 & 0 \\
1 & 0 & 0
\end{smallmatrix}
\right)
$.
As a consequence, $\mathfrak{i}=\mathfrak{h}\cap\pmin$ is contained in
\begin{equation}\label{equationformedeicontenustabE0}
 \mathfrak{o}=
 \enstq{v\in\pmin}{\overline{\ad}(v)(\R \bar{e}_0)\subset \R \bar{e}_0}=
 \left\{
 \left(\begin{smallmatrix}
  a & 0 & z \\
  0 & -a-b & 0 \\
  0 & 0 & b
 \end{smallmatrix}\right)
 \mid (a,b,z)\in\R^3
 \right\}.
\end{equation}
Denoting 
$
e^0=
\left(
\begin{smallmatrix}
 0 & 0 & 1 \\
 0 & 0 & 0 \\
 0 & 0 & 0
\end{smallmatrix}
\right)\in\mathfrak{o}$, 
we now prove that $\mathfrak{i}\cap\R e^0=\{0\}$, implying $\dim\mathfrak{i}\leq 2$
and finishing the proof of the Lemma, since $\mathfrak{i}$ is non-zero and $\dim\h-\dim\mathfrak{i}=3$. 
\par Let us assume by contradiction that $e^0\in\mathfrak{i}$.
As $\mathfrak{h}+\pmin=\g$ (because the orbit of $o$ under $H$ is open), 
there exists $v\in\mathfrak{h}$ and $w\in\pmin$ such that 
$
e_\beta=
\left(
\begin{smallmatrix}
 0 & 0 & 0 \\
 1 & 0 & 0 \\
 0 & 0 & 0
\end{smallmatrix}
\right)
=v+w
$.
But $[v,e^0] \in\mathfrak{h}$,
and $[w,e^0]\in\R e^0 \subset\mathfrak{i}$ since $w\in\pmin$,
finally implying that
$
\left(
\begin{smallmatrix}
 0 & 0 & 0 \\
 0 & 0 & 1 \\
 0 & 0 & 0
\end{smallmatrix}
\right)
=[e_\beta,e^0]
=[v,e^0]+[w,e^0]\in\mathfrak{h}\cap\pmin=\mathfrak{i}\subset\mathfrak{o}$,
which contradicts the description of $\mathfrak{o}$ in
\eqref{equationformedeicontenustabE0}.
\end{proof}

Let
\begin{equation}\label{equationdecompositionLevih}
 \h=\mathfrak{s}\ltimes_\phi\mathfrak{r},
\end{equation}
be the Levi decomposition of $\h$,
where $\mathfrak{s}$ is a semi-simple subalgebra of $\h$ (or is trivial if $\mathfrak{h}$ is solvable), 
$\mathfrak{r}$ is the solvable radical of $\h$ (it is an ideal of $\mathfrak{h}$),
and $\phi$ is the restriction of the adjoint representation $\ad\colon\mathfrak{h}\to\Der\mathfrak{h}$
($\phi\colon\mathfrak{s}\to \Der \mathfrak{r}$ 
describes the bracket in $\h$ by 
$[v,w]=\phi(v)(w)$ for $v\in\mathfrak{s}$ and $w\in\mathfrak{r}$).
\par A proper semi-simple subalgebra of $\slR{3}$ of dimension less than 5 is three-dimensional, and is thus isomorphic to 
$\slR{2}$ or to $\so{3}$.
Moreover, up to conjugacy in $\SL{3}$, the only embedding of $\so{3}$ in $\slR{3}$ is the inclusion,
and the only embeddings of $\slR{2}$ in $\slR{3}$ are
\begin{equation}\label{equationplongementsdesl2}
\mathfrak{s}_0\coloneqq\enstq{
 \begin{pmatrix}
  A & 0 \\
  0 & 0
 \end{pmatrix}
 }{A\in\slR{2}}
 \text{~and~}
 \biso{1}{2}.
\end{equation}
If $\mathfrak{h}$ is not solvable, $\mathfrak{s}$ is thus equal to
$\mathfrak{s}_0$, $\biso{1}{2}$ or $\so{3}$ up to conjugacy in $\SL{3}$.
The centralizers of these 
subalgebras in $\slR{3}$ are
\begin{equation}\label{equationcalculcentralisateurs}
\begin{cases}
\Cent_{\slR{3}}(\biso{1}{2})=\Cent_{\slR{3}}(\so{3})=\{0\}, \\
\Cent_{\slR{3}}(\mathfrak{s}_0)=\enstq{
\left(
\begin{smallmatrix}
x & 0 & 0 \\
0 & x & 0 \\
0 & 0 & -2x
\end{smallmatrix}
\right)
}{x\in\R}.
\end{cases}
\end{equation}

\begin{lemma}\label{lemmaclassificationsousalgebresdesl3}
Up to conjugacy in $\SL{3}$ or image by
$\theta=-\transp{\cdot}$, we have the following results.
 \begin{enumerate}
  \item If $\mathfrak{h}$ is not solvable, then
\begin{enumerate}
\item $\mathfrak{s}$ is equal to $\mathfrak{s}_0$,
\item and $\mathfrak{h}$ is equal to $\mathfrak{h}_\torus$ or to 
\begin{equation}\label{equationh1}
\mathfrak{h}_1=
\R^2\rtimes\slR{2}=
\enstq{
\left(
\begin{smallmatrix}
A & X \\
0 & 0
\end{smallmatrix}
\right)
}{A\in\slR{2},X\in\R^2}.
\end{equation}
\end{enumerate}  
\item If $\mathfrak{h}$ is solvable,
then either $\mathfrak{h}$ is contained in $\mathfrak{h}_\affine=\pmin$, or equal to 
\begin{equation}\label{equationh1}
\mathfrak{h}_2=
\R^2\rtimes\similitude{2}=
\enstq{
\left(
\begin{smallmatrix}
A & X \\
0 & -\tr A
\end{smallmatrix}
\right)
}{A\in\similitude{2},X\in\R^2},
\end{equation}
where 
$\similitude{2}=
\enstq{
\left(
\begin{smallmatrix}
a & b \\
-b & a
\end{smallmatrix}
\right)
}{(a,b)\in\R^2}
$.
 \end{enumerate}
\end{lemma}
\begin{proof}
1.a) Let us assume by contradiction that $\mathfrak{s}$ is conjugated to $\biso{1}{2}$ or 
$\so{3}$, implying that $\Cent_{\slR{3}}\mathfrak{s}=\{0\}$ according to \eqref{equationcalculcentralisateurs}. 
Since $\mathfrak{s}$ is simple, if the Lie algebra morphism $\phi$ is not injective then it is trivial, implying
$\mathfrak{r}\subset \Cent_{\slR{3}}\mathfrak{s}=\{0\}$ and thus $\dim\mathfrak{h}=\dim\mathfrak{s}=3$
which contradicts Lemma \ref{lemmareductiondimensionh}.
Our hypothesis on $\mathfrak{s}$ implies therefore that $\phi$ is injective, and in particular
that $\dim\Der\mathfrak{r}\geq\dim\mathfrak{s}=3$.
\par Since $\dim\mathfrak{s}=3$, the solvable radical $\mathfrak{r}$ is of dimension
1 or 2 according to Lemma \ref{lemmareductiondimensionh}, 
and is thus isomorphic to $\R$, $\mathfrak{aff}(\R)$, or $\R^2$.
But if $\mathfrak{r}$ is isomorphic to $\R$ or $\mathfrak{aff}(\R)$,
then $\Der \mathfrak{r}$ is of dimension 1 or 2 which contradicts 
the injectivity of $\phi$, and
$\mathfrak{r}$ is thus isomorphic to $\R^2$.
Since $\so{3}$ has no non-zero two-dimensional representation, this implies that
$\mathfrak{s}$ is conjugated to $\biso{1}{2}$.
The connected Lie subgroup $H$ of $\SL{3}$ of Lie algebra $\mathfrak{h}$ contains then
$\mathrm{SO}^0(1,2)$,
and its adjoint action induces thus by restriction a two-dimensional representation $\phi$
of $\mathrm{SO}^0(1,2)$ on $\mathfrak{r}$ (because $\mathfrak{r}$ is an ideal of $\mathfrak{h}$).
Since $\mathrm{SO}^0(1,2)$ is isomorphic to $\PSL{2}$,
$\phi$ is trivial, implying that
$\phi$ is trivial as well, which
contradicts the injectivity of $\phi$.
Finally, $\mathfrak{s}$ is conjugated to $\mathfrak{s}_0$. \\
1.b) Let us assume by contradiction that $\mathfrak{r}$ is isomorphic to $\mathfrak{aff}(\R)$.
Then $\Der\mathfrak{r}$ is two-dimensional and $\phi$ is thus non-injective, \emph{i.e.} trivial by simplicity of $\mathfrak{s}_0$. 
But $\mathfrak{r}$ is then contained in the centralizer of $\mathfrak{s}_0$ 
which is one-dimensional according to \eqref{equationcalculcentralisateurs},
contradicting the original hypothesis.
Therefore, $\mathfrak{r}$ is isomorphic to $\R^2$ or $\R$.
\par We first assume that $\mathfrak{r}$ is isomorphic to $\R^2$, implying that
$\phi$ is injective (otherwise
$\mathfrak{r}\subset\Cent_{\slR{3}}\mathfrak{s}_0$ which is one-dimensional).
We use the linear mapping 
$\ev_{e_3}\restreinta_{\mathfrak{r}}\colon M\in\mathfrak{r}\mapsto M(e_3)\in \R^3$
and discuss according to the dimension of its image $\mathfrak{r}(e_3)$.
Let us emphasize that $\mathfrak{r}$ is normalized by the
connected Lie subgroup $S_0$ of $\SL{3}$ of Lie algebra $\mathfrak{s}_0$,
and that $\mathfrak{r}(e_3)$ is thus preserved by $S_0$. 
If $\mathfrak{r}(e_3)$ is a plane then $\mathfrak{r}(e_3)=\Vect(e_1,e_2)$ since it is preserved by $S_0$,
and $\ev_{e_3}\restreinta_{\mathfrak{r}}$ is moreover injective. 
There exists $v\in\mathfrak{r}$ such that $\ev_{e_3}(v)=e_1$, and with
$A=
\left(\begin{smallmatrix}
1 & 0 \\
0 & -1
\end{smallmatrix}\right)\in\slR{2}$ and
$
u=
\left(\begin{smallmatrix}
A & 0 \\
0 & 0
\end{smallmatrix}\right)\in\mathfrak{s}_0
$
we have $\ev_{e_3}([u,v])=e_1=\ev_{e_3}(v)$.
This implies $[u,v]=v$
by injectivity of $\ev_{e_3}\restreinta_{\mathfrak{r}}$,
and finally
$
v=
\left(\begin{smallmatrix}
0 & 0 & 1 \\
0 & 0 & 0 \\
0 & x & 0
\end{smallmatrix}\right)
$ for some $x\in\R$.
The same reasoning with $w\in\mathfrak{r}$ such that $\ev_{e_3}(w)=e_2$ and
$A=
\left(\begin{smallmatrix}
-1 & 0 \\
0 & 1
\end{smallmatrix}\right)\in\slR{2}$,
implies that 
$
w=
\left(\begin{smallmatrix}
0 & 0 & 0 \\
0 & 0 & 1 \\
y & 0 & 0
\end{smallmatrix}\right)
$ for some $y\in\R$.
Since $\mathfrak{r}$ is abelian we have $[v,w]=0$, which
implies $x=y=0$ and proves that
$
\mathfrak{r}=
\left(
\begin{smallmatrix}
0 & \R^2 \\
0 & 0
\end{smallmatrix}
\right)
$, \emph{i.e.} that $\mathfrak{h}=\R^2\rtimes\slR{2}$.
If $\mathfrak{r}(e_3)=\{0\}$, then $p\restreinta_{\mathfrak{r}}$ is injective,
implying $p(\mathfrak{r})=\R^2$. Therefore $\dim(\theta(\mathfrak{r}))(e_3)=2$ which brings us back to the first case,
and $\theta(\mathfrak{h})=\R^2\rtimes\slR{2}$.
Finally, $\dim\mathfrak{r}(e_3)=1$ is impossible.
Otherwise, $\mathfrak{r}'=\ker\ev_{e_3}\cap\mathfrak{r}$ is one-dimensional, and since
$p\colon 
\left(\begin{smallmatrix}
B & 0 \\
X & 0
\end{smallmatrix}\right)
\in\mathfrak{r}'
\mapsto X\in\R^2
$ is injective (because $\mathfrak{r}\cap\mathfrak{s}_0=\{0\}$), $p(\mathfrak{r}')$ is a line of $\R^2$.
But for $w\in\mathfrak{r}'$ and
$
v=\left(\begin{smallmatrix}
A & 0 \\
0 & 0
\end{smallmatrix}\right)\in\mathfrak{s}_0
$ we have $p([v,w])=-p(w)A$, \emph{i.e.} $p(\mathfrak{r}')$ is preserved by $\slR{2}$ and cannot be a line.
\par We now assume that $\mathfrak{r}$ is isomorphic to $\R$.
Then $\phi$ is non-injective and thus trivial, 
implying $\mathfrak{r}\subset\Cent_{\slR{3}}\mathfrak{s}_0$.
This inclusion is an equality by equality of dimensions, proving $\mathfrak{h}=\mathfrak{h}_\torus$. \\
2. As $\mathfrak{h}$ is solvable, it preserves a complex line in $\C^3$ according to Levi's theorem.
More precisely, either $\mathfrak{h}$ preserves a real line, or it preserves a plane on which it acts by similarities.
The second case implies
$\mathfrak{h}\subset\R^2\rtimes\similitude{2}=\mathfrak{h}_2$ up to conjugacy in $\SL{3}$.
In the first case we can assume that $\mathfrak{h}$ preserves $\R e_1$, and
if the representation 
$
\left(\begin{smallmatrix}
* & * \\
0 & A
\end{smallmatrix}\right)\in\mathfrak{h}\mapsto A\in\gl{2}
$
also preserves a real line, then $\mathfrak{h}\subset\pmin=\mathfrak{h}_\affine$ up to conjugacy.
If not, then $\theta(\mathfrak{h})\subset
\enstq{
\left(
\begin{smallmatrix}
-\tr A & 0 \\
X & A
\end{smallmatrix}
\right)
}{A\in\similitude{2},X\in\R^2}$, according to the same remark than before.
This last subalgebra being conjugated to 
$\R^2\rtimes\similitude{2}=\mathfrak{h}_2$, this concludes the proof of the lemma.
\end{proof}

\subsection{Two further properties of the infinitesimal model}

We now prove two further properties of the infinitesimal model $(\mathfrak{h},\mathfrak{i})$,
in order to eliminate the ``exotic'' cases $\mathfrak{h}_1$ and $\mathfrak{h}_2$ that appeared in 
the algebraic clasification of Lemma \ref{lemmaclassificationsousalgebresdesl3}.
\begin{lemma}\label{lemmehestmaximale}
Let $\mathfrak{l}$ be a subalgebra of $\g$ containing $\mathfrak{h}$,
$\mathfrak{j}=\mathfrak{stab}_{\mathfrak{l}}(x_0)$ be the isotropy at $x_0$, and
$D^c$ be the line of $\mathfrak{l}/\mathfrak{j}$ sent to $\Exc(x_0)$ by 
the orbital map at $x_0$. 
If $\overline{\ad}(\mathfrak{j})(D^c)\subset D^c$, 
then $\mathfrak{l}=\mathfrak{h}$.
\end{lemma}
\begin{proof}
Let us denote by $L$ the connected subgroup of $\G$ of Lie algebra $\mathfrak{l}$, and by $J^0$ the identity component of 
$J=\Stab_L(x_0)$.
As $\overline{\ad}(\mathfrak{j})$ preserves $D^c$, $\overline{\Ad}(\exp(\mathfrak{j}))$ preserves $D^c$, and
the subgroup of elements $j\in J^0$ such that $\overline{\Ad}(j)$ preserves $D^c$ is thus equal to $J^0$ by connexity.
The construction made in the second assertion of Lemma \ref{lemmastructH0invarianteYdeltaisolocal} is thus valid for $L/J^0$, 
and proves 
the existence of an unique $L$-invariant enhanced Lagrangian contact structure $\Sx'$ extending 
$(\Exalpha(x_0),\Exbeta(x_0),\Exc(x_0))$ in the neighbourhood of $x_0$.
As $\mathfrak{h}\subset\mathfrak{l}$, $H\subset L$, 
and $\Sx'$ is thus $H$-invariant, implying $\Sx'=\Sx$
by the unicity of such a tructure.
Therefore $\mathfrak{l}\subset\mathfrak{kill}^{loc}_{\Sx}(x_0)=\mathfrak{h}$, which concludes the proof.
\end{proof}

\begin{lemma}\label{lemmaactionisotropideiPH}
Let us assume that $\mathfrak{i}$ is one-dimensional, and let $v$ be a non-zero element of $\mathfrak{i}$.
Then the eigenvalues of $\overline{\ad}(v)\in\Lin(\mathfrak{h}/\mathfrak{i})$ with respect to the eigenlines $D^\alpha$ and $D^\beta$ 
are non-zero.
\end{lemma}
\begin{proof}
We already know that $\overline{\ad}(\mathfrak{i})$ is diagonalizable with eigenlines $D^\alpha$, $D^\beta$, 
and $D^c$ (see Lemma \ref{lemmastructH0invarianteYdeltaisolocal}).
The proof is the same for the eigenvalues of both eigenlines $D^\alpha$ and $D^\beta$, and we only do it for $D^\alpha$.
By density of $\Rec(f)\cap\Rec(f^{-1})$ in $M$ (see the introduction of Section \ref{sectionSlochomsurOmega}), 
there exists $x\in O$ such that $\bar{x}=\pi_M(x)\in\Rec(f)\cap\Rec(f^{-1})$, and
possibly replacing $f$ by $f^{-1}$, 
we have $\underset{n\to+\infty}{\lim}\norme{\Diff{\bar{x}}{f^n}\restreinta_{\Ealpha(\bar{x})}}_M=0$
for a given Riemannian metric that we fix on $M$.
\par By hypothesis on $\bar{x}$, there exists a sequence $(\gamma_k)$ in $\piun{M}$ and a strictly increasing sequence $(n_k)$ of integers
such that $\gamma_k\tilde{f}^{n_k}(x)$ converges to $x$, and we
can moreover assume up to extraction that $x_k\in O$ for any $k$, implying that $\gamma_k \tilde{f}^{n_k}$ 
preserves $O$.
Endowing $\tilde{M}$ with the pullback $\tilde{\mu}_M$ of the Riemannian metric of $M$, we have
$\underset{k\to+\infty}{\lim}\norme{\Diff{x}{(\gamma_k\tilde{f}^{n_k})}\restreinta_{\Etildealpha(x)}}_{\tilde{\mu}_M}=0$
(since $\piun{M}$ acts by isometries).
\par Liouville's theorem \ref{theoremLiouvilleXG} for the homogeneous model space $(\X,\Lx)$
implies the existence of a unique sequence $(g_k)$ in $\G$
satisfying 
\begin{equation}\label{equationdefinitionsuiteholonomiegk}
\delta\circ \gamma_k\tilde{f}^{n_k}=g_k\circ\delta \text{~on a neighbourhood of~} x.
\end{equation}
Denoting $x_0=\delta(x)$,  $g_k\cdot x_0=\delta\circ\gamma_k\tilde{f}^{n_k}(x)\in Y=H\cdot x_0$
converges to $x_0$, and there exists thus a sequence $h_k\in H$ converging to the identity in $\G$ and such that  
$h_k\cdot x_0=g_k\cdot x_0$. 
Since $\delta$ is a local isomorphism
from $\Stilde\restreinta_O$ to $\Sx$ on a neighbourhood of $x$, the equation \eqref{equationdefinitionsuiteholonomiegk} 
defining $g_k$ shows that
$g_k$ preserves $\Sx$ on a neighbourhood of $x_0$.
By $H$-invariance of $\Sx$, $i_k=h_k^{-1}g_k$ also preserves $\Sx$, 
and $i_k$ is thus contained in the closed subgroup 
\[
I\coloneqq \enstq{i\in\Stab_\G(x_0)}{i\text{~preserves~}\Sx\text{~on a neighbourhood of~}x_0}
\]
of $\G$.
The Lie algebra of $I$ is equal to $\mathfrak{i}$ because $\mathfrak{is}^{loc}_{\Sx}(x_0)=\mathfrak{i}$ 
(see Lemma \ref{lemmastructH0invarianteYdeltaisolocal}).

\begin{factnonnumerote}
$I=\enstq{i\in\Stab_\G(x_0)}{\Ad(i)\cdot\mathfrak{h}=\mathfrak{h}\text{~and~}\overline{\Ad}(i)\cdot D^c=D^c}$.
In particular $I$ is algebraic and has a finite number of connected components.
\end{factnonnumerote}
\begin{proof}
For $i\in I$ and $v\in\mathfrak{h}$, the relation $\Diff{x_0}{i}\circ\Diff{e}{\theta_{x_0}}=\Diff{e}{\theta_{x_0}}\circ \Ad(i)$ implies
${i^{-1}}^*v^\dag=(\Ad(i)\cdot v)^\dag$.
Since $i$ is a local automorphism of $\Sx$ and $v^\dag$ a Killing field of $\Sx$, $(\Ad(i)\cdot v)^\dag$ is also a Killing field of $\Sx$,
implying $\Ad(i)\cdot v\in\mathfrak{h}$ since $\mathfrak{kill}^{loc}_{\Sx}(x_0)=\mathfrak{h}^\dag$ 
(see Lemma \ref{lemmastructH0invarianteYdeltaisolocal}).
Moreover, $\Diff{x_0}{i}(\Exc_{x_0})=\Exc_{x_0}$ implies $\overline{\Ad}(i)\cdot D^c=D^c$.
\par Let us conversely assume that $i\in\Stab_\G(x_0)$ satisfies 
$\Ad(i)\cdot \mathfrak{h}=\mathfrak{h}$ and $\overline{\Ad}(i)\cdot D^c=D^c$.
We consider $v\in\mathfrak{h}$ sufficiently close to $0$, such that with
$h=\e^v\in H$ and $y=h\cdot x_0\in Y$, $\Sx$ is defined at $y$.
Since $\overline{\Ad}(i)\cdot D^c=D^c$, $\Diff{x_0}{L_i}(\Exc(x_0))=\Exc(x_0)$, 
and $h'\coloneqq ihi^{-1}=\e^{\Ad(i)\cdot v}\in H$ because $\Ad(i)\cdot \mathfrak{h}=\mathfrak{h}$.
By $H$-invariance of $\Exc$, we obtain
$\Diff{y}{i}(\Exc(y))=\Diff{x_0}{h'}\circ\Diff{x_0}{i}(\Exc(x_0))=\Exc(i\cdot y)$,
proving that $i\in I$.
\end{proof}

We can thus assume up to extraction that $(i_k)$ lies in a given connected component of $I$,
and there exists then $g\in I$ such that $j_k=g i_k$ is contained in the identity component
$I^0$.
We endow $\X$ with a Riemannian metric $\mu_\X$, and denote by $\tilde{\mu}_\X=\delta^*\mu_\X$ its pullback 
on $\tilde{M}$.
Since $(\gamma_k\tilde{f}^{n_k}(x))$ is relatively compact in $\tilde{M}$, 
the metrics $\tilde{\mu}_M$ and $\tilde{\mu}_\X$ are equivalent in restriction to $(\gamma_k\tilde{f}^{n_k}(x))$,
and the limit stated above for $\tilde{\mu}_M$ is thus valid for $\tilde{\mu}_\X$,
implying $\lim\norme{\Diff{x_0}{g_k}\restreinta_{\Exalpha(x_0)}}_{\mu_\X}=0$.
As $j_k=gh_k^{-1}g_k$ with $(gh_k^{-1})$ relatively compact in $\G$, we also have
$\lim \norme{\Diff{x_0}{j_k}\restreinta_{\Exalpha(x_0)}}_{\mu_{\X}}=0$.
\par $I^0$ being connected and one-dimensional, there exists a non-zero $v\in\mathfrak{i}$ and a sequence $t_k\in\R$
such that $i_k=\exp(t_k v)$, implying that
$\Diff{x_0}{j_k}$
is conjugated by the orbital map to $\exp(t_k\overline{\ad}(v))$,
and thus $\lim\norme{\exp(t_k\overline{\ad}(v))\restreinta_{D^\alpha}}=0$.
Denoting by
$\lambda_\alpha$ the eigenvalue of $\overline{\ad}(v)$ with respect to $D^\alpha$, 
$\exp(t_k\overline{\ad}(v))\restreinta_{D^\alpha}=\exp(\lambda_\alpha t_k)\id_{D^\alpha}$
implies then $\lambda_\alpha\neq0$.
\end{proof}

\subsection{End of the classification}

We now put into our analysis the geometrical and dynamical properties of $\mathfrak{h}$ 
proved above.

\begin{lemma}\label{lemmah1respectepashypothesesgeometriques}
$\mathfrak{h}_1=\R^2\rtimes\slR{2}$ does not satisfy the
geometrical conditions of Lemma \ref{lemmastructH0invarianteYdeltaisolocal}.
\end{lemma}
\begin{proof}
The only open orbit 
of the connected Lie subgroup $H_1$ of $\G$ of Lie algebra $\mathfrak{h}_1$
is the open subset $\Omega_a$
defined in Paragraph \ref{soussoussectionsurfacescarteaffine}.
If $H_1\cdot x_0$ is open for some point $x_0\in \X$,
we can thus assume that $x_0=([e_3],[e_3,e_1])\in\Omega_a$ up to conjugacy in $H_1$,
implying that
$
\mathfrak{i}_1=\Lie(\Stab_{H_1}(x_0))=
\enstq{
\left(
\begin{smallmatrix}
a & b \\
0 & -a 
\end{smallmatrix}
\right)
}{a,b\in\R^2}
$. 
Denoting
$v_\alpha=
\left(\begin{smallmatrix}
0 & 0 \\
1 & 0
\end{smallmatrix}\right)
$
and 
$v_\beta=
\left(\begin{smallmatrix}
1 \\
0
\end{smallmatrix}\right)\in\mathfrak{h}_1$,
we have $\R v_\alpha^\dag(x_0)=\Exalpha(x_0)$ and $\R v_\beta^\dag(x_0)=\Exbeta(x_0)$, and
defining
$v_c=
\left(\begin{smallmatrix}
0 \\
1
\end{smallmatrix}\right)
$
and
$
i=
\left(\begin{smallmatrix}
0 & 1 \\
0 & 0
\end{smallmatrix}\right)\in\mathfrak{i}_1
$,
the matrix of $\overline{\ad}(i)$ in the basis $(\bar{v}_\alpha,\bar{v}_\beta,\bar{v}_c)$ of $\mathfrak{h}_1/\mathfrak{i}_1$ is
\[
\Mat_{(\bar{v}_\alpha,\bar{v}_\beta,\bar{v}_c)}\overline{\ad}(i)=
\left(
\begin{smallmatrix}
0 & 0 & 0 \\
0 & 0 & 1 \\
0 & 0 & 0
\end{smallmatrix}
\right).
\]
Any line of $\mathfrak{h}_1/\mathfrak{i}_1$ that is transverse to $\Vect(\bar{v}_\alpha,\bar{v}_\beta)$ has projective coordinates
$[a,b,1]$ in the basis $(\bar{v}_\alpha,\bar{v}_\beta,\bar{v}_c)$ for some $(a,b)\in\R^2$,
and $\overline{\ad}(i)(D^c)$ has thus coordinates $[0,1,0]$.
This proves that $\overline{\ad}(i)(D^c)\not\subset D^c$, \emph{i.e.} that $\mathfrak{h}_1$ does not satisfy
the geometrical conditions of Lemma \ref{lemmastructH0invarianteYdeltaisolocal}.
\end{proof}

\begin{lemma}\label{lemmah2etsousalgebreproprepminrespectentpasproprdynamiques}
If $\mathfrak{h}$ is a four-dimensional subalgebra of $\mathfrak{h}_\affine=\pmin$, or is equal to 
$\mathfrak{h}_2=\R^2\rtimes\similitude{2}$, 
then $\mathfrak{h}$ does not respect both the geometrical conditions of Lemma \ref{lemmastructH0invarianteYdeltaisolocal} and the
dynamical condition of Lemma \ref{lemmaactionisotropideiPH}.
\end{lemma}
\begin{proof}
We first assume that $\mathfrak{h}$ is a four-dimensional subalgebra of $\pmin$. 
Therefore $H\subset\Pmin$,  
and if $H\cdot x_0$ is open then $x_0\in Y_\affine$ according to Proposition \ref{propositioninformationsYtYa}.
We can thus assume up to conjugacy in $H$ that $x_0=o_\affine=([e_3],[e_3,e_2])\in Y_\affine$, implying that
\[
\mathfrak{i}=\mathfrak{stab}_{\mathfrak{h}}(o_\affine)\subset
\mathfrak{i}_\affine=\mathfrak{stab}_{\pmin}(o_\affine)=\enstq{
\left(
\begin{smallmatrix}
a & 0 & 0 \\
0 & -a-b & 0 \\
0 & 0 & b
\end{smallmatrix}
\right)
}{(a,b)\in\R^2}.
\]
Let $D^c\subset\mathfrak{h}/\mathfrak{i}$ be a line preserved by $\overline{\ad}(\mathfrak{i})$,
and such that $\Diff{e}{\overline{\theta}_{o_\affine}}(D^c)$ is transverse to $(\Exalpha\oplus\Exbeta)(o_\affine)$. 
Since $\mathfrak{h}$ is a proper subalgebra of $\pmin$, Lemma \ref{lemmehestmaximale} implies
that $\overline{\ad}(\mathfrak{i}_\affine)(D^c)\not\subset D^c$, and thus that
$\mathfrak{stab}_{\mathfrak{i}_\affine}(D^c)=
\enstq{v\in\mathfrak{i}_\affine}{\overline{\ad}(v)(D^c)\subset D^c}=
\mathfrak{i}$ is one-dimensional.
Any line $D^c$ of $\pmin/\mathfrak{i}_\affine$ which is transverse to 
the contact plane
has projective coordinates $[x,y,1]$ in the basis $(\bar{X},\bar{Y},\bar{Z})$ of $\pmin/\mathfrak{i}_\affine$,
for some $(x,y)\in\R^2$ 
(see Proposition \ref{propositionmodelelocaldeO}), and
according to \eqref{equationmatriceactionadjointeiaffine}, we have:
\begin{itemize}
\item if $x=y=0$, \emph{i.e.} $D^c=\R \bar{Z}$, then $\mathfrak{stab}_{\mathfrak{i}_\affine}(\R \bar{Z})=\mathfrak{i}_\affine$;
\item if $x=0$ and $y\neq 0$, \emph{i.e.} $D^c=D^c_Y(t)\coloneqq\R (\bar{Z}+t\bar{Y})$ for some $t\in\R$,
then $\mathfrak{stab}_{\mathfrak{i}_\affine}(D^c_Y(t))$ is equal to the line $\mathfrak{i}_Y$
generated by the diagonal matrix
$[1,1,-2]=
\left(
\begin{smallmatrix}
 1 & 0 & 0 \\
 0 & 1 & 0 \\
 0 & 0 & -2
\end{smallmatrix}
\right)
$;
\item if $x\neq0$ and $y=0$, \emph{i.e.} $D^c=D^c_X(t)\coloneqq\R (\bar{Z}+t\bar{X})$ for some $t\in\R$,
then $\mathfrak{stab}_{\mathfrak{i}_\affine}(D^c_X(t))$ is equal to the line $\mathfrak{i}_X$
generated by the diagonal matrix
$[-2,1,1]=
\left(
\begin{smallmatrix}
 -2 & 0 & 0 \\
 0 & 1 & 0 \\
 0 & 0 & 1
\end{smallmatrix}
\right)
$;
\item if $x\neq0$ and $y\neq0$, then $\mathfrak{stab}_{\mathfrak{i}_\affine}(D^c)=\{0\}$.
\end{itemize}
The only transverse lines with a one-dimensional stabilizer being
$D^c_X(t)$ and $D^c_Y(t)$, $\mathfrak{i}$ is equal to $\mathfrak{i}_X$ or $\mathfrak{i}_Y$.
But  
$\Mat_{(\bar{X},\bar{Y},\bar{Z})}\overline{\ad}([1,1,-2])=[0,3,3]$
and  $\Mat_{(\bar{X},\bar{Y},\bar{Z})}\overline{\ad}([-2,1,1])=[-3,0,-3]$ according to \eqref{equationmatriceactionadjointeiaffine}, \emph{i.e.}
the elements of $\mathfrak{i}_X$
and $\mathfrak{i}_Y$ have zero eigenvalue with respect to either the $\alpha$ or the $\beta$-direction, proving
that $\mathfrak{h}$ does not satisfy the dynamical condition of Lemma \ref{lemmaactionisotropideiPH}.
\par In the same way, if $\mathfrak{h}=\mathfrak{h}_2$, then we can assume that $x_0=o_\affine\in \Omega_a$
up to conjugacy in 
$H_2=\R^2\rtimes\Sim{2}$,
implying $\mathfrak{i}_2=\mathfrak{stab}_{\mathfrak{h}_2}(o_\affine)=\mathfrak{i}_Y$ defined above.
We saw that the elements of $\mathfrak{i}_Y$ have zero eigenvalue 
with respect to the $\alpha$-direction, proving that $\mathfrak{h}_2$ does not satisfy the dynamical condition of Lemma 
\ref{lemmaactionisotropideiPH}. 
\end{proof}

Proposition \ref{propositionclassificationalgebriquemodelelocalS} now directly follows from 
Lemmas \ref{lemmaclassificationsousalgebresdesl3}, \ref{lemmah1respectepashypothesesgeometriques} and
\ref{lemmah2etsousalgebreproprepminrespectentpasproprdynamiques}.




%







\section{Global structure}\label{sectionHYstructsurM}

From the local model that we determined for the enhanced Lagrangian contact structure $\Sm$,
we will now deduce a global information.

\subsection{Local homogeneity of the enhanced Lagrangian contact structure}\label{soussectionSlocalementhom}

So far, we only have informations on a dense and open subset $\Omega$ of $M$
(see Propositions \ref{corollairepremiereetapepreuve} and \ref{propositionmodelelocaldeO}), and
the first step to obtain a global information is to prove the following.

\begin{proposition}\label{propositionomega=M}
The open dense subset $\Omega$ equals $M$,
\emph{i.e.} $\Sm$ is locally homogeneous on $M$.
\end{proposition}

We will denote in this paragraph by $(\mathcal{C},\varphi)=(\hat{M},\omega,\varphi)$
the normal generalized Cartan geometry of the enhanced Lagrangian contact structure 
$\Stilde=\pi_M^*\Sm$ of $\tilde{M}$, and by $\Ktot\colon\hat{M}\to W_{\Ktot}$ its total curvature
(see Paragraphs \ref{soussoussectiongeomcartangeneralisee} and \ref{soussoussectiontotalcurvature}).
We recall that $\tilde{\Omega}=\pi_M^{-1}(\Omega)\subset\tilde{M}$, and that the projection of the Cartan bundle
is denoted by $\pi\colon\hat{M}\to \tilde{M}$.
\par We also recall that the local homogeneity of $\Stilde\restreinta_{\tilde{\Omega}}$ means that
the connected components of $\tilde{\Omega}$ are exactly its  $\Kill^{loc}$-orbits 
(see Definition \ref{definitionautlocorbitetlocalementhom}).
Since the rank of $\Diff{}{\Ktot}$ is invariant by the right action of $\Pmin$
and by the flow of Killing fields, this shows that
$\rg(\Diff{}{\Ktot})$ is constant over any connected component of $\tilde{\Omega}$.
\par We choose for this whole paragraph a connected component $O$ of $\tilde{\Omega}$ 
(\emph{i.e.} a $\Kill^{loc}$-orbit of $\Stilde$)
such that $\rg(\Diff{\hat{x}}{\Ktot})$ for $\hat{x}\in\pi^{-1}(O)$ is maximal among $\rg(\Diff{\hat{x}}{\Ktot})$
for $\hat{x}\in\pi^{-1}(\tilde{\Omega})$. 
We will denote by $(Y,\Sx)$ 
the local model of $\Stilde\restreinta_{O}$, 
equal to $(Y_\torus,\mathcal{S}_\torus)$ or $(Y_\affine,\mathcal{S}_\affine)$
and such that $\delta\restreinta_O\colon (O,\Stilde\restreinta_O)\to (Y,\Sx)$ is a local isomorphism (see
Corollary \ref{propositionmodelelocaldeO}).
We still denote by $\mathfrak{h}$ the subalgebra of Killing fields of $\Sx$,
respectively equal to $\mathfrak{h}_\torus$ or $\mathfrak{h}_\affine$ (see Proposition \ref{propositionclassificationalgebriquemodelelocalS}),
and by $H$ the corresponding Lie connected subgroup 
\[
H_\torus^0=
\begin{bmatrix}
\GLplus{2} & 0 \\
0 & 1
\end{bmatrix}
\text{~or~} H_\affine^0=\Pmin^+, 
\]
of $\G$ of Lie algebra $\mathfrak{h}$, preserving $\Sx$.

\par We recall that $\delta\colon\tilde{M}\to\X$ denotes the developping map of the $(\G,\X)$-structure of $M$
describing the flat Lagrangian contact structure $\Lm$ (see Proposition \ref{corollaryLGXstructure}).
\begin{lemma}\label{lemmaimageborddeO}
The boundary of $O$ is mapped to $\X\setminus Y$ by the developping map:
$\delta(\partial O)\subset \X\setminus Y$.
\end{lemma}
\begin{proof}
Let us assume by contradiction that there exists $x\in \partial O$ such that $x_0=\delta(x)\in Y$.
The pullback $\tilde{\h}\coloneqq\delta^*\h=\enstq{\delta^*v}{v\in\h}$
is a subalgebra of vector fields on $\tilde{M}$, such that 
$\mathfrak{Kill}(O,\Stilde\restreinta_O)=\tilde{\h}\restreinta_O$
according to Lemma \ref{lemmaalgebrelocaleKillingdeO}.
As $x_0\in Y$,
there exists an open and convex neighbourhood $W_0$ of $0$ in $\h$ such that
$V=\exp(W_0)\cdot x_0\subset Y$ is an open neighbourhood of $x_0$.
Denoting $W=\delta^*W_0\subset\tilde{\h}$, $U=\enstq{\varphi^1_X(x)}{X\in W}$ is thus an open neighbourhood of $x$,  
and possibly shrinking $W_0$, 
we can moreover assume that $\delta\restreinta_U$ is a diffeomorphism from $U$ to $V$.
As $x\in\partial O$, there exists $y\in U\cap O$, and $X\in W$ such that $x=\varphi_X^1(y)$, 
implying that $\varphi^t_X(y)\in U$ for any $t\in\intervalleff{0}{1}$,
and thus $\delta(\varphi^t_X(y))\in V\subset Y$.
Denoting $t_0=\inf\enstq{t\in\intervalleff{0}{1}}{\varphi^t_X(y)\in\partial O}$, $t_0>0$ 
because $O$ is open, 
and $\varphi^{t_0}_X(y)\in\partial O$
because $\partial O$ is closed. 
Replacing $x$ by $\varphi^{t_0}_X(y)$ 
and $X$ by $\frac{X}{t_0}\in W$, we finally have $y\in O$, $x=\varphi^1_X(y)\in\partial O$,
and for any $t\in\intervallefo{0}{1}$, $\varphi^t_X(y)\in O$, with $X\restreinta_O\in\Kill(O,\Stilde\restreinta_O)$. 
\par Choosing $\hat{y}\in\pi^{-1}(y)$, the invariance of $D^1\Ktot$ by local automorphisms 
and the fact that $\varphi^t_X$ is a local automorphism of $(\mathcal{C},\varphi)$ on the neighbourhood of $y$ 
for any $t\in\intervallefo{0}{1}$ 
implies $D^1\Ktot(\hat{\varphi}_X^t(\hat{y}))=D^1\Ktot(\hat{y})$ for any $t\in\intervallefo{0}{1}$.
Denoting $\hat{x}=\hat{\varphi}^1_X(\hat{y})$, we obtain
$D^1\Ktot(\hat{x})=D^1\Ktot(\hat{y})$ by continuity, 
\emph{i.e.} $\Ktot(\hat{x})=\Ktot(\hat{y})$ 
and $\Diff{\hat{x}}{\Ktot}\circ\omega_{\hat{x}}^{-1}=\Diff{\hat{y}}{\Ktot}\circ\omega_{\hat{y}}^{-1}$.
\par This implies $\hat{x}\in\hat{M}^{int}$. 
In fact as the rank of $\Diff{}{\Ktot}$ can only increase locally, there is an
open neighbourhood $\mathcal{U}$ of $\hat{x}$ where the rank of $\Diff{}{\Ktot}$ is greater than $\rg(\Diff{\hat{x}}{\Ktot})$.
Let us assume by contradiction that
the open subset of $\mathcal{U}$ where $\rg(\Diff{\hat{x}'}{\Ktot})>\rg(\Diff{\hat{x}}{\Ktot})$ is non-empty.
Then by density of $\pi^{-1}(\tilde{\Omega})$, there exists
$\hat{z}\in\pi^{-1}(\tilde{\Omega})$ such that $\rg(\Diff{\hat{z}}{\Ktot})>\rg(\Diff{\hat{x}}{\Ktot})$.
But $\rg(\Diff{\hat{x}}{\Ktot})=\rg(\Diff{\hat{y}}{\Ktot})$ because $D^1\Ktot(\hat{x})=D^1\Ktot(\hat{y})$, and thus 
$\rg(\Diff{\hat{z}}{\Ktot})>\rg(\Diff{\hat{y}}{\Ktot})$ with $\hat{y}\in\pi^{-1}(O)$,
wich contradicts our hypothesis of maximality of $\rg(\Diff{}{\Ktot})$ on $O$.
Therefore $\rg(\Diff{}{\Ktot})$ is constant on the open neighbourhood $\mathcal{U}$ of $\hat{x}$, 
proving that $\hat{x}\in\hat{M}^{int}$ according to Integrability theorem \ref{theoremintegrabilite}.
\par As the $\Kill^{loc}$-orbit $O$ of $y$ is open, 
$\omega_{\hat{y}}^{-1}(\pmin)+\Ker(\Diff{\hat{y}}{\Ktot})=\Tan{\hat{y}}{\hat{M}}$,
and therefore $\omega_{\hat{x}}^{-1}(\pmin)+\Ker(\Diff{\hat{x}}{\Ktot})=\Tan{\hat{x}}{\hat{M}}$ because 
$\Diff{\hat{x}}{\Ktot}\circ\omega_{\hat{x}}^{-1}=\Diff{\hat{y}}{\Ktot}\circ\omega_{\hat{y}}^{-1}$.
Since $\hat{x}\in\hat{M}^{int}$, for any $v\in\Ker(\Diff{\hat{x}}{\Ktot})$
there is a local Killing field $X$ of $\Stilde$ defined in the neighbourhood of $x$
such that $\hat{X}_{\hat{x}}=v$. 
But $\omega_{\hat{x}}^{-1}(\pmin)+\Ker(\Diff{\hat{x}}{\Ktot})=\Tan{\hat{x}}{\hat{M}}$,
and we thus have $\enstq{X_x}{X\in\mathfrak{kill}^{loc}_{\Stilde}(x)}=\Tan{x}{\tilde{M}}$,
implying that the $\Kill^{loc}$-orbit of $x$ is open.
Since $x\in\partial O$, the $\Kill^{loc}$-orbit of $x$ intersects thus the $\Kill^{loc}$-orbit $O$, \emph{i.e.} $x\in O$, which
contradicts our initial hypothesis.
This contradiction concludes the proof of the lemma.
\end{proof}

Lemma \ref{lemmaimageborddeO} allows us to reduce the study of the central direction $\Etildec$ on the boundary of $O$, to the
study of the central direction $\Exc$ on the boundary of $Y$.
We first do some geometrical remarks about the open subsets $Y_\affine$ and $Y_\torus$ of $\X$, defined in 
Paragraphs \ref{soussoussectionYtorus} and \ref{soussoussectionYaffine}.
\par Let us recall that, denoting $D_\infty=[e_1,e_2]$, $m_\torus=[e_3]$ and $m_\affine=[e_1]$, we have
\[
Y_\torus=\X\setminus (S_{\beta,\alpha}(D_\infty)\cup S_{\alpha,\beta}(m_\torus))
\text{~and~} Y_\affine=\X\setminus (S_{\beta,\alpha}(D_\infty)\cup S_{\beta,\alpha}(m_\affine)).
\]
In particular, for $\varepsilon=\affine$ and $\torus$:
$\X\setminus Y_\varepsilon=\partial Y_\varepsilon=S_{\beta,\alpha}(D_\infty)\cup S_{\alpha,\beta}(m_\varepsilon)$.
\par We define in both cases
\[
\mathcal{G}\coloneqq\enstq{x\in\partial Y}{\Calpha(x)\nsubseteq \partial Y \text{~or~} \Cbeta(x)\nsubseteq \partial Y}.
\]
It is easy to check that for $\varepsilon=\affine$ and $\torus$, we have
\[
\mathcal{G}_\varepsilon=\partial Y_\varepsilon\setminus 
\{\Cbeta(D_\infty)\cup \Calpha(m_\varepsilon)\cup (S_{\beta,\alpha}(D_\infty)\cap S_{\alpha,\beta}(m_\varepsilon))\},
\]
and that for any $x\in\mathcal{G}$, if $\mathcal{C}^\varepsilon(x)\nsubseteq\partial Y$ for $\varepsilon=\alpha$ or $\beta$, 
then $\mathcal{C}^\varepsilon(x)\setminus\{x\}\subset Y$.
\par We have $S_{\beta,\alpha}(D_\infty)\cap S_{\alpha,\beta}(m_\affine)=\Cbeta(D_\infty)\cup \Calpha(m_\affine)$, and
$S_{\beta,\alpha}(D_\infty)\cap S_{\alpha,\beta}(m_\torus)$ is equal to the \emph{chain defined by $(m_\torus,D_\infty)$},
denoted by $\mathcal{C}(m_\torus,D_\infty)$ and defined as follows:
\[
\mathcal{C}(m_\torus,D_\infty)\coloneqq \enstq{(m,[m,m_\torus])}{m\in D_\infty}.
\]
Finally, we will use the following description of the respective orbits of $H$ on $\mathcal{G}$:
\begin{enumerate}
\item the orbits of $H_\torus^0$ on $\mathcal{G}_\torus$ are 
$\mathcal{G}_\torus^1=S_{\alpha,\beta}(m_\torus)\setminus(\Calpha(m_\torus)\cup\mathcal{C}(m_\torus,D_\infty))$
where $\Calpha(x)\setminus\{x\}\subset Y_\torus$,
and $\mathcal{G}_\torus^2=S_{\beta,\alpha}(D_\infty)\setminus(\Cbeta(D_\infty)\cup\mathcal{C}(m_\torus,D_\infty))$
where $\Cbeta(x)\setminus\{x\}\subset Y_\torus$;
\item the orbits of $H_\affine^0$ on $\mathcal{G}_\affine$ are
$\mathcal{G}_\affine^1=S_{\alpha,\beta}(m_\affine)\setminus(\Calpha(m_\affine)\cup \Cbeta(D_\infty))$ 
where $\Calpha(x)\setminus\{x\}\subset Y_\affine$
and
$\mathcal{G}_\affine^2=S_{\beta,\alpha}(D_\infty)\setminus(\Calpha(m_\affine)\cup \Cbeta(D_\infty))$
where $\Cbeta(x)\setminus\{x\}\subset Y_\affine$.
\end{enumerate}

We will now prove that the central direction $\Exc$
degenerates along the $\alpha$ and $\beta$-circles when converging to a point of $\mathcal{G}$.
\begin{lemma}\label{lemmadegenerescence}
Let $\gamma\colon\intervalleff{0}{1}\to \X$ be a smooth path such that $\gamma(\intervalleof{0}{1})\subset Y$,
$x=\gamma(0)\in\mathcal{G}$, and $\gamma(\intervalleff{0}{1})$ is entirely contained in $\Calpha(x)$, 
or entirely contained in $\Cbeta(x)$.
Then $\Exc(\gamma(t))$ converges at $t=0$ to a line contained in $(\Exalpha\oplus\Exbeta)(x)$.
\end{lemma}
\begin{proof}
As the action of $H$ on $Y$ preserves $\Exc$, 
it will be sufficient to prove this result for one point of each of the two orbits of $H$ on $\mathcal{G}$ described above,
in each of the two cases $Y_\torus$ or $Y_\affine$. 
We thus have only four cases to handle, and
we saw that in each case, either $\Calpha(x)\setminus\{x\}\subset Y$ and $\Cbeta(x)\subset \partial Y$,
or the contrary. 
We thus have only one possibility to consider for $\gamma$ in each of these four cases, either that  
$\gamma(\intervalleff{0}{1})\subset\Calpha(x)$, or that $\gamma(\intervalleff{0}{1})\subset\Cbeta(x)$.
To clarify our strategy,
let $x$ be a point of $\mathcal{G}_\mu^i$
for $\mu=\torus$ or $\affine$ and $i=1$ or 2, and let us consider 
the following data:
\begin{itemize}
\item a one-parameter subgroup $\{g^t\}_{t\in\R}$ of $\G$
such that, denoting $x(t)=g^t\cdot x$, we have $\enstq{x(t)}{t\in\R}=\mathcal{C}^\varepsilon(x)\setminus\{y\}$, with 
$y\in\mathcal{C}^\varepsilon(x)\cap Y$, and $\varepsilon=\alpha$ or $\beta$ according to the case considered,
\item a one-parameter subgroup $\{h^t\}_{t\in\R}$ of $H$ such that $g^t\cdot x=x(t)=h^{t^{-1}}\cdot y$ for any $t\in\R^*$,
\item $A$ in $\g$ such that $\Diff{e}{\theta_y}(\R A)=\Exc(y)$,
where $\theta_y\colon \G\to \X$ is the orbital map at $y$,
\item and $g_0\in\G$ such that $g_0\cdot x=o$ 
where $o=([e_1],[e_1,e_2])$ is the usual base-point of $\X$.
\end{itemize}
Then for any $t\in\R^*$ we have
$\Diff{x(t)}{(g_0g^{-t})}(\Exc(x(t)))=
\Diff{e}{\theta_o}(\R\Ad(g_0g^{-t}h^{t^{-1}})\cdot A))$.
Denoting by $p\colon\g\to\g/\pmin$ the canonical projection, let us assume that
$p(\R\Ad(g_0g^{-t}h^{t^{-1}})\cdot A)$ converges at $t=0$ to a line contained in 
$\Vect(\bar{e}_\alpha,\bar{e}_\beta)$.
Then $\Diff{e}{\theta_o}(\R\Ad(g_0g^{-t}h^{t^{-1}})\cdot A)\subset\Tan{o}{\X}$ converges to a line 
$L\subset(\Exalpha\oplus\Exbeta)(o)$, and as $g^tg_0^{-1}$
converges to $g_0^{-1}$ at $t=0$, we deduce by continuity that $\Exc(x(t))$ converges at $t=0$ to $\Diff{o}{g_0^{-1}}(L)$,
contained in $(\Exalpha\oplus\Exbeta)(x)$, because $g_0^{-1}$ preserves $\Exalpha\oplus\Exbeta$.
\par In conclusion, we only have to find, in each of the four cases $\mu=\torus$ or $\affine$ and $i=1$ or 2,
a point $x\in\mathcal{G}^i_\mu$, together with $g^t$, $h^t$, $A$, and $g_0$ satisfying the above conditions, and to prove that 
$p(\R\Ad(g_0g^{-t}h^{t^{-1}})\cdot A)$ converges at $t=0$ to a line contained in 
$\Vect(\bar{e}_\alpha,\bar{e}_\beta)$.
\par We begin with $Y_\torus$, for which we choose for both orbits $\mathcal{G}^1_\torus$ and $\mathcal{G}^2_\torus$ 
the point $y\coloneqq o_\torus=([1,0,1],[(1,0,1),e_2])\in Y_\torus$.
Let us recall that in this case, 
$
A=
\left(
\begin{smallmatrix}
1 & 0 & 0 \\
0 & -1 & 0 \\
0 & 0 & 0
\end{smallmatrix}
\right)
$
satisfies $\Exc(o_\torus)=\Diff{e}{\theta_{o_\torus}}(\R A)$ (see Paragraph \ref{soussoussectionYtorus}).
\begin{itemize}
\item For $\mathcal{G}^1_\torus$, choosing $x=([1,0,1],[(1,0,1),e_1])=([1,0,1],[e_1,e_3])$, 
$g_0=
\left(
\begin{smallmatrix}
1 & 0 & 0 \\
1 & 0 & -1 \\
0 & 1 & 0 
\end{smallmatrix}
\right)
$,
and the one-parameter subgroups
$g^t=
\left(
\begin{smallmatrix}
1 & 0 & 0 \\
t & 1 & -t \\
0 & 0 & 1
\end{smallmatrix}
\right)
$ of $\G$ and 
$h^t=
\left(
\begin{smallmatrix}
1 & t & 0 \\
0 & 1 & 0 \\
0 & 0 & 1
\end{smallmatrix}
\right)
$ of $H_\torus^0$ such that $g^t\cdot x=h^{t^{-1}}\cdot o_\torus\in\Calpha(x)$,
we obtain: 
\[
\Ad(g_0g^{-t}h^{t^{-1}})\cdot A
=
\left(
\begin{smallmatrix}
1 & -2 & -2t^{-1} \\
1 & -2 & -2t^{-1} \\
-t & t & 1
\end{smallmatrix}
\right),
\]
and thus $p(\R\Ad(g_0g^{-t}h^{t^{-1}})\cdot A)$ converges at $t=0$ to $\R\bar{e}_\beta$.
\item For $\mathcal{G}_\torus^2$, choosing $x=([e_2],[e_2,(1,0,1)])$,
$
g_0=
\left(
\begin{smallmatrix}
0 & 1 & 0 \\
1 & 0 & 0 \\
1 & 0 & -1 
\end{smallmatrix}
\right)\
$, and the one-parameter subgroups 
$
g^t=
\left(
\begin{smallmatrix}
1 & t & 0 \\
0 & 1 & 0 \\
0 & t & 1
\end{smallmatrix}
\right)
$ of $\G$
and 
$
h^t=
\left(
\begin{smallmatrix}
1 & 0 & 0 \\
t & 1 & 0 \\
0 & 0 & 1
\end{smallmatrix}
\right)
$ of $H_\torus^0$ such that $g^t\cdot x=h^{t^{-1}}\cdot o_\torus\in\Cbeta(x)$, 
we obtain 
\[
\Ad(g_0g^{-t}h^{t^{-1}})\cdot A
=
\left(
\begin{smallmatrix}
1 & 2t^{-1} & 0 \\
0 & -1 & 0\\
t & 1 & 0
\end{smallmatrix}
\right),
\] 
and thus $p(\R\Ad(g_0g^{-t}h^{t^{-1}})\cdot A)$ converges at $t=0$ to $\R\bar{e}_\alpha$.
\end{itemize}

We now consider the case of $Y_\affine$, for which we choose for both orbits $\mathcal{G}^1_\affine$ and $\mathcal{G}^2_\affine$ 
the point $y\coloneqq o_\affine=([e_3],[e_3,e_2])\in Y_\affine$,
and we recall that in this case 
$
A=
\left(
\begin{smallmatrix}
0 & 0 & 1 \\
0 & 0 & 0 \\
0 & 0 & 0
\end{smallmatrix}
\right)
$
satisfies the above condition $\Exc(o_\affine)=\Diff{e}{\theta_{o_\affine}}(\R A)$ (see Paragraph \ref{soussoussectionYaffine}).
\begin{itemize}
\item For $\mathcal{G}^1_\affine$, choosing $x=([e_3],[e_3,e_1])$, 
$g_0=
\left(
\begin{smallmatrix}
0 & 0 & 1 \\
1 & 0 & 0 \\
0 & 1 & 0 
\end{smallmatrix}
\right)
$,
and the one-parameter subgroups
$g^t=
\left(
\begin{smallmatrix}
1 & 0 & 0 \\
t & 1 & 0 \\
0 & 0 & 1
\end{smallmatrix}
\right)
$ of $\G$ and 
$h^t=
\left(
\begin{smallmatrix}
1 & t & 0 \\
0 & 1 & 0 \\
0 & 0 & 1
\end{smallmatrix}
\right)
$ of $H_\affine^0$ such that $g^t\cdot x=h^{t^{-1}}\cdot o_\affine\in\Calpha(x)$, 
we obtain: 
\[
\Ad(g_0g^{-t}h^{t^{-1}})\cdot A
=
\left(
\begin{smallmatrix}
0 & 0 & 0 \\
1 & 0 & 0 \\
-t & 0 & 0
\end{smallmatrix}
\right),
\]
and thus $p(\R\Ad(g_0g^{-t}h^{t^{-1}})\cdot A)$ converges at $t=0$ to $\R\bar{e}_\beta$.
\item For $\mathcal{G}_\affine^2$, choosing $x=([e_2],[e_2,e_3])$,
$
g_0=
\left(
\begin{smallmatrix}
0 & 1 & 0 \\
0 & 0 & 1 \\
1 & 0 & 0 
\end{smallmatrix}
\right)\
$, and the one-parameter subgroups 
$
g^t=
\left(
\begin{smallmatrix}
1 & 0 & 0 \\
0 & 1 & 0 \\
0 & t & 1
\end{smallmatrix}
\right)
$ of $\G$
and 
$
h^t=
\left(
\begin{smallmatrix}
1 & 0 & 0 \\
0 & 1 & t \\
0 & 0 & 1
\end{smallmatrix}
\right)
$ of $H_\affine^0$ such that $g^t\cdot x=h^{t^{-1}}\cdot o_\affine\in\Cbeta(x)$, 
we obtain 
\[
\Ad(g_0g^{-t}h^{t^{-1}})\cdot A
=
\left(
\begin{smallmatrix}
0 & 0 & 0 \\
0 & 0 & 0\\
t & 1 & 0
\end{smallmatrix}
\right),
\] 
and thus $p(\R\Ad(g_0g^{-t}h^{t^{-1}})\cdot A)$ converges at $t=0$ to $\R\bar{e}_\alpha$.
\end{itemize}
According to the discussion above, this concludes the proof of the lemma.
\end{proof}

We are now able to prove the proposition \ref{propositionomega=M}.
\begin{proof}[Proof of the proposition \ref{propositionomega=M}]
Let us assume by contradiction that $\Omega\neq M$.
We choose a connected component $O$ of $\tilde{\Omega}$ such that
the rank of $\Diff{\hat{x}}{\Ktot}$ for $\hat{x}\in\pi^{-1}(O)$ is maximal among the rank of $\Diff{\hat{x}}{\Ktot}$
for $\hat{x}\in\pi^{-1}(\tilde{\Omega})$. 
As $\varnothing\neq O\neq \tilde{M}$ by hypothesis, there exists $x\in\partial O$, and as $\Etildealpha\oplus\Etildebeta$ is contact, 
\cite[Theorem 4.1]{sussmann} implies the 
existence of a piecewise smooth path $\gamma\colon\intervalleff{0}{1}\to\tilde{M}$ constituted of a finite concatenation
of segments of $\alpha$ and $\beta$-leaves, joining $x=\gamma(1)$ to a point $y=\gamma(0)\in O$.
Denoting $t_0=\inf\enstq{t\in\intervalleff{0}{1}}{\gamma(t)\in\partial O}$, $t_0>0$ 
and $\gamma(t_0)\in\partial O$. 
Replacing $x$ by $\gamma(t_0)$, keeping only the last smooth arc of $\gamma$, replacing $y$ by the origin of this arc,
and choosing a parametrization of this arc by $\intervalleff{0}{1}$,
we finally end with a smooth path $\gamma\colon\intervalleff{0}{1}\to\tilde{M}$ such that
$\gamma(\intervallefo{0}{1})\subset O$, $x=\gamma(1)\in\partial O$, and $\gamma(\intervalleff{0}{1})$
is entirely contained in a same $\alpha$ or $\beta$-leaf.
The proof being the same in the two cases, we assume that 
$\gamma(\intervalleff{0}{1})\subset\widetilde{\mathcal{F}}^\alpha(x)$
to fix the ideas.
Denoting $x_0=\delta(x)$, $x_0\in\X\setminus Y$ according to Lemma \ref{lemmaimageborddeO},
and $\delta(\gamma(\intervallefo{0}{1}))\subset Y$ because $\delta(O)\subset Y$ (see Lemma \ref{lemmadeltaOcontenudansY}). 
Finally $\delta(\gamma(\intervallefo{0}{1}))$ is an open interval of $C^\alpha(x_0)$ contained in $Y$, and
$x_0\in\X\setminus Y$, \emph{i.e.} $x_0\in\mathcal{G}$.
Denoting $\gamma_0(t)=\delta(\gamma(t))$,
Lemma \ref{lemmadegenerescence} implies therefore that $\Exc(\gamma_0(t))$ converges to a line 
$D^c_0\subset(\Exalpha\oplus\Exbeta)(x_0)$ at $t=1$. As $\delta\restreinta_O$ is a local isomorphism
between $\Stilde\restreinta_O$ and $\Sx$, we have 
$\Etildec(\gamma(t))=(\Diff{\gamma(t)}{\delta})^{-1}(\Exc(\gamma_0(t)))$ for any $t\in\intervallefo{0}{1}$, 
implying $\Etildec(x)=(\Diff{x}{\delta})^{-1}(D^c_0)$ by continuity.
Since $\delta$ is a local isomorphism between the Lagrangian contact structures $\Ltilde$ and $\Lx$, 
this implies that $\Etildec(x)\subset(\Etildealpha\oplus\Etildebeta)(x)$, which contradicts the definition 
of the transverse distribution $\Etildec$.
This contradiction concludes the proof of the proposition.
\end{proof}

\subsection{Reduction of the holonomy group}

We first describe the global and local automorphisms of 
the models $(Y_\torus,\mathcal{S}_\torus)$ and $(Y_\affine,\mathcal{S}_\affine)$.
\begin{proposition}\label{propositiongroupeautomorphismeetLiouvillequatremodeles}
\begin{enumerate}
\item 
$\Aut(Y_\torus,\Sm_\torus)=H_\torus=
\begin{bmatrix}
\GL{2} & 0 \\
0 & 1
\end{bmatrix}$
and
$\Aut(Y_\affine,\Sm_\affine)=H_\affine=\Pmin$. 
\item Let $(Y,\Sx)$ be one of the two models $(Y_\torus,\mathcal{S}_\torus)$ or $(Y_\affine,\mathcal{S}_\affine)$.
Then any local isomorphism of $\Sx$ between two connected open subsets of $Y$
is the restriction of the action of a global automorphism of $\Aut(Y,\Sx)$.
\end{enumerate}
\end{proposition}
\begin{proof}
1. The inclusions $H_\torus\subset\Aut(Y_\torus,\Sm_\torus)$ and $H_\affine\subset\Aut(Y_\affine,\Sm_\affine)$
were explained in Paragraphs \ref{soussoussectionYtorus} and \ref{soussoussectionYaffine}.
Since the automorphism groups are contained in the stabilizers of the open subsets,
the equalities follow because 
$H_\torus=\Stab_\G(Y_\torus)$ and $H_\affine=\Stab_\G(Y_\affine)$. \\
2. Let us emphasize that
$\Aut(Y,\Sx)$ is precisely the normalizer of $\mathfrak{h}$ in $\G$.
Let $\varphi$ be a local automorphism of $\Sx$ between two connected open subsets $U$ and $V$ of $Y$.
For any $v\in\mathfrak{h}$, since $v\restreinta_V$ is a Killing field of $\Sx$,
$\varphi^*(v\restreinta_V)$ is a Killing field of $\Sx$, and therefore
$\varphi^*(v\restreinta_V)=w\restreinta_U$ for some $w\in\mathfrak{h}$.
But $\varphi$ is in particular a local automorphism of the Lagrangian contact structure $\Lx$ of $\X$,
and is thus the restriction to an open subset $U\subset Y$ of the left translation by an element $g\in \G$, 
according to Theorem \ref{theoremLiouvilleXG}.
Therefore $w\restreinta_U=\varphi_*(v\restreinta_V)=(\Ad(g)\cdot v)\restreinta_U$, implying
that $\Ad(g)\cdot v=w\in\mathfrak{h}$ since the action of $\G$ is analytic (see Lemma \ref{lemmadescriptionkillinggeomCartanplatemodel}).
Consequently, $g\in\Nor_\G(\mathfrak{h})=\Aut(Y,\Sx)$.
\end{proof}

Let us recall that $\rho\colon\piun{M}\to \G$ denotes the holonomy morphism associated to the developping map
$\delta\colon\tilde{M}\to\X$ of the $(\G,\X)$-structure of $M$ 
(see Corollary \ref{corollaryLGXstructure} and Paragraph \ref{soussoussectionflatlagcontact}).

\begin{proposition}\label{propositionSHYstructuresurMdeuxmodeles}
The holonomy group $\rho(\piun{M})$ is contained in $\Aut(Y,\Sx)$.
Consequently, $M$ has either a $(H_\torus,Y_\torus)$-structure or a $(H_\affine,Y_\affine)$-structure,
and its developping map is a local isomorphism
of enhanced Lagrangian contact structures from $\Stilde$ to $\Sm_\torus$ (respectively $\Sm_\affine$).
\end{proposition}
\begin{proof}
 According to Proposition \ref{propositionomega=M}, $\Sm$ is locally homogeneous, 
 and we thus deduce from Proposition \ref{propositionmodelelocaldeO} that, up to interversion of the distributions 
 $\Ealpha$ and $\Ebeta$, the developping map $\delta$ of the $(\G,\X)$-structure of $M$ can be chosen to be a local 
 isomorphism from $(\tilde{M},\Stilde)$ to one of the two models $(Y_\torus,\mathcal{S}_\torus)$ 
 or $(Y_\affine,\mathcal{S}_\affine)$.
 According to Liouville's theorem \ref{propositiongroupeautomorphismeetLiouvillequatremodeles} proved 
 for these two models, the holonomy morphism has moreover values in the corresponding automorphism group
 $H_\torus$ (respectively $H_\affine$) described in the same result, and
 $M$ is finally endowed with a $(H_\torus,Y_\torus)$-structure (resp. $(H_\affine,Y_\affine)$-structure).
 Concerning the interversion of $\Ealpha$ and $\Ebeta$, 
 it is easy to construct for both models $(Y_\torus,\mathcal{S}_\torus)$ 
 and $(Y_\affine,\mathcal{S}_\affine)$,
 a diffeomorphism of $Y$ interverting the distributions 
 $\Exalpha$ and $\Exbeta$ and fixing the transverse distribution $\Exc_Y$.
 In other words for these both models, the structures $(\Exalpha,\Exbeta,\Exc_Y)$ and $(\Exbeta,\Exalpha,\Exc_Y)$ are isomorphic,
 so that \emph{a posteriori}, the order of the distributions $\Ealpha$ and $\Ebeta$ 
in the statement of Proposition \ref{propositionSHYstructuresurMdeuxmodeles} 
 does not matter.
 \end{proof}

\section{Completeness of the structure}\label{sectionstructKleinienne}

We will denote by $(\mathcal{H},Y)$ the local model of $\mathcal{S}$,
which is either $(H_\torus,Y_\torus)$ or $(H_\affine,Y_\affine)$, 
and by $\delta\colon\tilde{M}\to Y$ and $\rho\colon\piun{M}\to \mathcal{H}$ the developping map and holonomy morphism of
the $(\mathcal{H},Y)$-structure of $M$.
The goal of this section is to prove that:
\begin{proposition}\label{propositionHYstructurecomplete}
The developping map $\delta$ is a covering map from $\tilde{M}$ to $Y$.
\end{proposition}

\par It is a known fact that a local diffeomorphism satisfying the path-lifting property is a covering map
(the reader can for example look for a proof in \cite[\S 5.6, Proposition 6 p. 383]{docarmo}).
According to the following statement, it will actually be sufficient
to prove the path-lifting property in the $\alpha$, $\beta$, and central directions
to prove that $\delta$ is a covering map.
\begin{lemma}\label{lemmerelevementchemins}
Let $h\colon N\to B$ be a local diffeomorphism between two smooth three-dimensional manifolds, $B$ being connected.
We assume that there is a smooth splitting $E_1\oplus E_2 \oplus E_3=\Fitan{B}$ of the tangent bundle
of $B$ into three one-dimensional smooth distributions, such that
for any $i\in\{1,2,3\}$, $x\in\Image(h)$, and $\tilde{x}\in h^{-1}(x)$,
any path tangent to $E_i$ and starting from $x$ entirely lifts through $h$ to a path starting from $\tilde{x}$.
Then $h$ is a covering map from $N$ to $B$ (and in particular, $h$ is surjective).
\end{lemma}
\begin{proof}
Since $h$ is a local diffeomorphism, it suffices to prove that our hypothesis implies the lift of any path.
By compactness, it is sufficient to locally lift the paths in $B$, around any point. 
We choose $x\in B$ and a sufficiently small open neighbourhood $U$ of $x$, such that 
there are three smooth vector fields 
$X$, $Y$ and $Z$ generating $E_1$, $E_2$ and $E_3$ on $U$, and $\varepsilon>0$ such that
$(t,u,v)\in\intervalleoo{-\varepsilon}{\varepsilon}^3\mapsto \phi(t,u,v)\coloneqq \varphi^t_X\circ\varphi^u_Y\circ\varphi^v_Z(x)\in U$ 
is well-defined, and is a diffeomorphism (this exists according to Inverse mapping theorem).
Let us choose $\tilde{x}\in h^{-1}(x)$.
Then, denoting by $\tilde{X}=h^*X$, $\tilde{Y}=h^*Y$ and $\tilde{Z}=h^*Z$ the pullbacks, 
the property of path-lifting in the directions $E_1$, $E_2$ and $E_3$, and from any point, implies that
$\tilde{\phi}(t,u,v)\coloneqq \varphi^t_{\tilde{X}}\circ\varphi^u_{\tilde{Y}}\circ\varphi^v_{\tilde{Z}}(\tilde{x})$
is well-defined on $\intervalleoo{-\varepsilon}{\varepsilon}^3$. 
If $\gamma\colon\intervalleff{0}{1}\to U$ is a continuous path starting from $x$ and contained in $U$, there are 
three continuous maps $t$, $u$ and $v$ from $\intervalleff{0}{1}$ to $\intervalleoo{-\varepsilon}{\varepsilon}$ such that
$\gamma(s)=\phi(t(s),u(s),v(s))$.
Since $h\circ\tilde{\phi}=\phi$ by construction, $\tilde{\gamma}(s)\coloneqq \tilde{\phi}(t(s),u(s),v(s))$ is then a lift of $\gamma$
starting from $\tilde{x}$, which finishes the proof.
\end{proof}

\begin{remark}\label{remarkconditionsuffisanterelevementchemins}
 In our case, proving that the paths in $\delta(\tilde{M})$ in the $\alpha$-direction (respectively $\beta$ or central direction)
 lift to $\tilde{M}$ is equivalent to prove that for any $x\in\delta(\tilde{M})$ and $\tilde{x}\in\delta^{-1}(x)$,
 we have the following equality: 
 \[
 \delta(\tilde{\mathcal{F}}^\alpha(\tilde{x}))=\Calpha(x)\cap \delta(\tilde{M}),
 \]
 (respectively the same equality for $\beta$-leaves and circes, or for central leaves).
\end{remark}

We start by proving that the image of any $\alpha$ (respectively $\beta$) leaf in $\tilde{M}$ miss
exactly one point in the associated $\alpha$-circle (respectively $\beta$-circle).
We recall that $\partial Y=\X\setminus Y$, as explained before Lemma \ref{lemmadegenerescence}.

\begin{lemma}\label{lemmeimagefeuillesstables}
For any $\tilde{x}\in\tilde{M}$, denoting $x=\delta(\tilde{x})$,
there exists 
$x^*\in \Cbeta(x)\cap \partial Y$ such that 
$\delta(\tilde{\mathcal{F}}^\beta(\tilde{x}))=\Cbeta(x)\setminus\{x^*\}=\Cbeta(x)\cap Y$.
The same happens for $\alpha$-leaves and their associated $\alpha$-circles.
\end{lemma}
\begin{proof}
We will only write the proof for $\beta$-leaves and $\beta$-circles as in the statement, the case of the $\alpha$-direction
being the same.
Denoting $\bar{x}=\pi_M(\tilde{x})$, and
possibly replacing $f$ by $f^{-1}$, we have
$\underset{n\to+\infty}{\lim}\norme{\Diff{\bar{x}}{f^n\restreinta_{\Ealpha(\bar{x})}}}_M=0$
for some fixed Riemannian metric on $M$. 
\par The description of the open subsets $Y_\torus$ and $Y_\affine$ in 
Paragraphs \ref{soussoussectionYtorus} and \ref{soussoussectionYaffine}
easily shows that
in these both cases, the intersection of any $\beta$-circle (respectively $\alpha$-circle) with $Y$ miss exactly one point of 
the circle.
In other words, the intersection $\Cbeta(x)\cap\partial Y$ is a single point $\{x^*\}$, and as a consequence 
$\delta(\tilde{\mathcal{F}}^\beta(\tilde{x}))\subset \Cbeta(x)\setminus\{x^*\}=\Cbeta(x)\cap Y$.
To finish the proof of the lemma, we only have to prove that $\delta(\tilde{\mathcal{F}}^\beta(\tilde{x}))$
cannot miss more than one point of $\Cbeta(x)$. 
To achieve this, we assume by contradiction that there exists $x^-\neq x^+ \in \Cbeta(x)\setminus\{x,x^*\}$ such that:
\begin{equation}\label{absurdhypothesis2lemmeimagecerclepastropgrande}
\delta(\tilde{\mathcal{F}}^\beta(\tilde{x}))=\intervalleoo{x^-}{x^+}\subsetneq \Cbeta(x)\setminus\{x^*\},
\end{equation}
where $\intervalleoo{x^-}{x^+}$ is the connected component of 
$\Cbeta(x)\setminus\{x^-,x^+\}$ that contains $x$.
\par Since $M$ is compact, there exists a strictly increasing sequence $(n_k)$ of positive integers such that,
denoting $\bar{x}=\pi_M(\tilde{x})$, $\bar{x}_k=f^{n_k}(\bar{x})$ converges to a point $\bar{x}_\infty\in M$,
and as $M=\pi_1(M)\backslash\tilde{M}$, there furthermore exists a sequence $\gamma_k\in\pi_1(M)$ such that
$\tilde{x}_k\coloneqq \gamma_k\cdot\tilde{f}^{n_k}(\tilde{x})$ converges to a lift $\tilde{x}_\infty$ of $\bar{x}_\infty$.
As $\gamma_k\tilde{f}^{n_k}$ is an automorphism of the Lagrangian contact structure $\Ltilde$ and $\delta$
a local isomorphism from $\Ltilde$ to $\Lx$, Theorem \ref{theoremLiouvilleXG} 
implies the existence of a unique sequence $g_k\in\G$ satisfying
\[
\delta(\gamma_k\cdot\tilde{f}^{n_k}(\tilde{x}))=g_k\cdot\delta(\tilde{x}).
\]
We denote
$x_k=\delta(\tilde{x}_k)=g_k(x)$, that converges to $x_\infty\coloneqq\delta(\tilde{x}_\infty)$.
Denoting $x_k^-=g_k(x^-)$ and $x_k^+=g_k(x^+)$,
$x_k$, $x_k^-$ and $x_k^+$ are three distincts points of $\Cbeta(x_k)$ for any $k$.
By compactness of $\X$, we can assume up to extraction that $x_k^-$ and $x_k^+$ respectively converge to points 
$x_\infty^-$ and $x_\infty^+$ of $\Cbeta(x_\infty)$, and the hypothesis \eqref{absurdhypothesis2lemmeimagecerclepastropgrande} 
allows us to obtain the following crucial statement.
\begin{fact}\label{factargumentsubtil}
$x_\infty \neq x_\infty^-$, and $x_\infty \neq x_\infty^+$.
\end{fact}
\begin{proof}
Let us assume by contradiction that $x_\infty^- = x_\infty$.
Considering a neighbourhood $U$ of $\tilde{x}_\infty$ such that $\delta\restreinta_U$ is injective,
we can choose $\tilde{y}_\infty\in(\tilde{\mathcal{F}}^\beta(\tilde{x}_\infty)\cap U)\setminus\{\tilde{x}_\infty\}$.
There exists a sequence $\tilde{y}_k\in\tilde{\mathcal{F}}^\beta(\tilde{x}_k)$ converging to $\tilde{y}_\infty$, and 
possibly changing $\tilde{y}_\infty$, we can moreover assume that $\delta(\tilde{y}_k)\in\intervalleoo{x_k^-}{x_k}$,
implying that $\delta(\tilde{y}_\infty)\in\intervalleff{x_\infty^-}{x_\infty}$ by continuity.
But as $x_\infty^- = x_\infty$, $\intervalleff{x_\infty^-}{x_\infty}=\{x_\infty\}$, 
and therefore $\delta(\tilde{y}_\infty)=x_\infty=\delta(\tilde{x}_\infty)$, 
implying $\tilde{y}_\infty=\tilde{x}_\infty$ by injectivity of $\delta\restreinta_U$, which contradicts our hypothesis on $\tilde{y}_\infty$.
In the same way, we obtain $x_\infty \neq x_\infty^+$.
\end{proof}

\par The subgroup $\SO{3}$ of $\G$ acts transitively on $\X$, and we can thus choose $\phi\in\SO{3}$ and a sequence $(\phi_k)$
in $\SO{3}$,
satisfying $\phi(x)=o$ and $\phi_k(x_k)=o$ for any $k$ (we recall that $o=([e_1],[e_1,e_2])$).
Since $\Stab_{\SO{3}}(\Cbeta(o))=
\left[
\begin{smallmatrix}
\SO{2} & 0 \\
0 & 1
\end{smallmatrix}
\right]
$ acts transitively on $\Cbeta(o)$,
we can moreover assume that $\phi(x^+)=o^+$ and $\phi_k(x^+_k)=o^+$, where $o^+=([e_2],[e_1,e_2])\in\Cbeta(o)$.
For any $k$, $\phi_k\circ g_k \circ \phi^{-1}$ is an element of 
$\Stab_{\G}(o)\cap\Stab_\G(o^+)$, \emph{i.e.} is of the form
$
\left[
\begin{smallmatrix}
1 & 0 & x \\
0 & \lambda_k & y \\
0 & 0 & \mu_k
\end{smallmatrix}
\right]
$.
But 
$
\left[
\begin{smallmatrix}
1 & 0 & *  \\
0 & 1 & * \\
0 & 0 & *
\end{smallmatrix}
\right]
$ acts trivially in restriction to $\Cbeta(o)$, and
$
A_k \coloneqq
\left[
\begin{smallmatrix}
1 & 0 & 0 \\
0 & \lambda_k & 0 \\
0 & 0 & 1
\end{smallmatrix}
\right]
$ satisfies thus:
\[
g_k\restreinta_{\Cbeta(x)}=\phi_k^{-1}\circ A_k \circ \phi \restreinta_{\Cbeta(x)}.
\]
The following commutative diagram summarizes the situation.
\begin{equation}\label{equationdiagrammecommutatif}
\begin{tikzcd}
\Cbeta(o) \arrow[d,"A_k"] & 
\Cbeta(x) \arrow[l,"\phi"] \arrow[d,"g_k"] & 
\widetilde{\mathcal{F}}^\beta(\tilde{x}) \arrow[l,"\delta"] \arrow[d,"\gamma_k{\tilde{f}}^{n_k}"] \arrow[r,"\pi_M"] &
\mathcal{F}^\beta(\bar{x}) \arrow[d,"f^{n_k}"] \\
\Cbeta(o)&                 
\Cbeta(x_k) \arrow[l,"\phi_k"] & 
\widetilde{\mathcal{F}}^\beta(\tilde{x}_k) \arrow[l,"\delta"] \arrow[r,"\pi_M"]                                  
& \mathcal{F}^\beta(\bar{x}_k)
\end{tikzcd}
\end{equation}

\par The action of $A_k\in\G$ on $\Cbeta(o)$ is conjugated to the action of the projective transformations
$
\left[
\begin{smallmatrix}
1 & 0 \\
0 & \lambda_k
\end{smallmatrix}
\right]\in\PGL{2}
$ on $\RP{1}$, \emph{i.e.} to the action of the homotheties of ratio $\lambda_k$ on $\R\cup\{\infty\}$.
By this conjugation, $o$ corresponds to 0, $o^+$ to $\infty$, and 
$o^-\coloneqq\phi(x^-)\in\Cbeta(o)\setminus\{o,o^+\}$ corresponds to a non-zero point of $\R$.
Fact \ref{factargumentsubtil} implies that 
$A_k(o^-)=\phi_k(x^-_k)\in\Cbeta(o)$ stays bounded away from $o$ (since $\phi_k\in\SO{3}$), and therefore that
$\lambda_k$ is bounded away from $0$.
\par On the other hand, endowing $\tilde{M}$ with the pullback of the Riemannian metric of $M$, the diagramm 
\eqref{equationdiagrammecommutatif}
implies $\underset{k\to+\infty}{\lim}\norme{\Diff{\tilde{x}}{(\gamma_k\tilde{f}^{n_k})}\restreinta_{\Etildebeta(\tilde{x})}}_{\tilde{M}}=0$
(since $\piun{M}$ acts by isometries).
Fixing any Riemannian metric on $\X$, as $(\tilde{x}_k)$ is relatively compact we also have
$\lim\norme{\Diff{x}{g_k}\restreinta_{\Exbeta(x)}}_\X=0$, and
since $(\phi_k)$ and $(x_k)$ are relatively compact as well, we finally obtain
$\lim\norme{\Diff{o}{A_k}\restreinta_{\Exbeta(x)}}_\X=0$.
\par This contradicts the fact that $\lambda_k$ is bounded away from 0, and
this contradiction concludes the proof of the lemma.
\end{proof}

Lemma \ref{lemmeimagefeuillesstables} allows us to easily infer the path-lifting property in the $\alpha$ and $\beta$-directions.
\begin{corollary}\label{corollairerelevementcerclesalphaetbeta}
\begin{enumerate}
 \item For any $x\in \delta(\tilde{M})$, $\Calpha(x)\cap \delta(\tilde{M})=\Calpha(x)\cap Y$
and $\Cbeta(x)\cap \delta(\tilde{M})=\Cbeta(x)\cap Y$.
\item The paths in $\delta(\tilde{M})$ in the $\alpha$ and $\beta$-directions lift to $\tilde{M}$ from any point.
\end{enumerate} 
\end{corollary}
\begin{proof}
We only write the proof of the statements for the $\alpha$-direction, the case of the $\beta$-direction being formally the same. \\
1. For any $\tilde{x}\in\tilde{M}$, denoting $\delta(\tilde{x})=x$, we know that $\partial Y\cap\Calpha(x)$ 
is equal to a single point $\{x^*\}$ that satisfies $\Calpha(x)\setminus\{x^*\}=\Calpha(x)\cap Y$.
Furthermore, $\delta(\tilde{\mathcal{F}}^\alpha(\tilde{x}))=\Calpha(x)\setminus\{x^*\}=\Calpha(x)\cap Y$ according to 
Lemma \ref{lemmeimagefeuillesstables}.
As $\Calpha(x)\cap\delta(\tilde{M})\subset\cup_{\tilde{x}\in\delta^{-1}(x)}\delta(\tilde{\mathcal{F}}^\alpha(\tilde{x}))=
\Calpha(x)\cap Y$,
we finally obtain $\Calpha(x)\cap \delta(\tilde{M})=\Calpha(x)\cap Y$. \\
2. Together with Lemma \ref{lemmeimagefeuillesstables}, 
we finally have $\delta(\tilde{\mathcal{F}}(\tilde{x}))=\Calpha(x)\cap \delta(\tilde{M})$, for any 
$x\in \delta(\tilde{M})$ and $\tilde{x}\in\delta^{-1}(x)$.
According to the remark \ref{remarkconditionsuffisanterelevementchemins}, this proves
that any path starting from $x$ in the $\alpha$-direction lifts to $\tilde{M}$ from $\tilde{x}$.
\end{proof}

The accessibility property of Lagrangian contact structures allows us to deduce that:
\begin{corollary}\label{corollaireimagedelta}
The developping map is surjective: $\delta(\tilde{M})=Y$.
\end{corollary}
\begin{proof}
Let $x$ be a point of the non-empty subset $\delta(\tilde{M})$, and $y$ be any point in $Y$.
Restricting the Lagrangian contact structure $\Lx=(\Exalpha,\Exbeta)$ of $\X$ to the connected open subset $Y$, 
\cite[Theorem 4.1]{sussmann} implies the existence of
a finite number $x=x_1,\dots,x_n=y$ of points of $Y$ such that for any $i=1,\dots,n-1$,
$x_{i+1}\in \Calpha(x_i)\cap Y$ or $x_{i+1}\in \Cbeta(x_i)\cap Y$.
Applying the first statement of Corollary \ref{corollairerelevementcerclesalphaetbeta}, we deduce by a direct 
finite recurrence that for any $i$, $x_i\in \delta(\tilde{M})$, so that $y\in\delta(\tilde{M})$. 
\end{proof}

We finally prove that the central paths also lift, by a specific method for each model.
\begin{lemma}\label{lemmerelevementcheminscentrauxYt}
In the case of $Y_\torus$, any central path starting at any point $x\in Y_\torus$ lifts in $\tilde{M}$
from any point $\tilde{x}\in\delta^{-1}(x)$.
\end{lemma}
\begin{proof}
Let us recall that
$H_\torus=
\left[
\begin{smallmatrix}
\GL{2} & 0 \\
0 & 1
\end{smallmatrix}
\right]=\Aut(Y_\torus,\Sm_\torus)
$ 
and $o_\torus=([1,0,1],[(1,0,1),e_2])\in Y_\torus$.
Since $Z=\left(\begin{smallmatrix}
 1 & 0 & 0 \\
 0 & 1 & 0 \\
 0 & 0 & -2
 \end{smallmatrix}\right)$ is central in $\mathfrak{h}_\torus$, the Killing field 
 $Z^\dag$ of $\Sm_\torus$ associated to $Z$ is $H_\torus$-invariant.
As $\Exc_\torus(o_\torus)=\R Z^\dag(o_\torus)$ (see Paragraph \ref{soussoussectionYtorus}) and $\Exc$ is $H_\torus$-invariant as well, 
$Z^\dag$ actually generates the transverse distribution on $Y_\torus$.
At any point $x\in Y_\torus$, we thus have $\mathcal{F}^c_\torus(x)=\exp(\R Z)\cdot x$.
Now, as the holonomy group $\rho(\pi_1(M))$ is contained in $H_\torus$ according to 
Proposition \ref{propositionSHYstructuresurMdeuxmodeles}, it leaves $Z^\dag$ invariant,
and the pullback $\tilde{X}\coloneqq\delta^*Z^\dag$ is thus preserved by the fundamental group $\pi_1(M)$.
This allows us to push $\tilde{X}$ down on $M$, to a Killing field $X$ generating the central direction $E^c$.
As $M$ is compact, $X$ is a complete vector field, and as $\pi_M\colon\tilde{M}\to M$ is a covering map,
the pullback $\pi_M^*X=\tilde{X}$ is also complete,
implying that for any $\tilde{x}\in\tilde{M}$, the central leaf at $\tilde{x}$ is simply
the integral curve of $\tilde{X}=\delta^*Z^\dag$ at $\tilde{x}$.
For any $x\in Y_\torus$ and $\tilde{x}\in\delta^{-1}(x)$ (which is non-empty because $\delta(\tilde{M})=Y_\torus$ 
according to Corollary \ref{corollaireimagedelta}) we thus have
$\delta(\tilde{\mathcal{F}}^c(\tilde{x}))=\enstq{\delta(\varphi_{\tilde{X}}^t(\tilde{x}))}{t\in\R}
=\exp(\R Z)\cdot x=\mathcal{F}^c_\torus(x)$. This finishes the proof of the lemma according to 
Remark \ref{remarkconditionsuffisanterelevementchemins}.
\end{proof}

\begin{lemma}\label{lemmerelevementcheminscentrauxYa}
In the case of $Y_\affine$, any central path starting at any point $x\in Y_\affine$ lifts in $\tilde{M}$
from any point $\tilde{x}\in\delta^{-1}(x)$.
\end{lemma}
\begin{proof}
Let us first emphasize that the argument used in the previous lemma for the case of $Y_\torus$ 
does not work here, because the center of $\mathfrak{h}_\affine$ is trivial.
\par We identify $Y_\affine$ with $\R^3$ through
$(x,y,z)\in\R^3\mapsto([x,y,1],[(x,y,1),(z,1,0)])\in Y_\affine$,
and we consider the following vector fields of $Y_\affine$ in these global coordinates:
\[
X^\alpha(x,y,z)=e_3,X^\beta(x,y,z)=(z,1,0),X^c(x,y,z)=e_1.
\]
These vector fields are complete and
generate the enhanced Lagrangian contact structure $\Sm_\affine=(\Exalpha,\Exbeta,\Exc_\affine)$ on $Y_\affine$
(see Paragraph \ref{soussoussectionYaffine}).
Since the paths tangent to the $\alpha$ and $\beta$-distributions
entirely lift to $\tilde{M}$ according to Corollary \ref{corollairerelevementcerclesalphaetbeta}, we deduce
that the pullbacks $\tilde{X}^\alpha=\delta^*X^\alpha$ and $\tilde{X}^\beta=\delta^*X^\beta$ are complete as well.
We can furthermore realize the flow of the central vector field $X^c$ 
by $\alpha-\beta$ curves through the following equalities:
\[ 
\begin{cases}
 \varphi_{X^\beta}^{-t}\circ\varphi_{X^\alpha}^{-t}\circ\varphi_{X^\beta}^t\circ\varphi_{X^\alpha}^t(x)
=x+t^2e_1=\varphi_{X^c}^{t^2}(x), \\
 \varphi_{X^\beta}^{t}\circ\varphi_{X^\alpha}^{-t}\circ\varphi_{X^\beta}^{-t}\circ\varphi_{X^\alpha}^t(x)
=x-t^2e_1=\varphi_{X^c}^{-t^2}(x).
\end{cases}
\]
The same equalities are thus true for the pullbacks $\tilde{X}^\alpha$, $\tilde{X}^\beta$, and 
$\tilde{X}^c=\delta^*X^c$,
and since the flows of $\tilde{X}^\alpha$ and $\tilde{X}^\beta$ are defined for all times,
these equalities show that $\tilde{X}^c$ is also complete.
The completeness of $\tilde{X}^c$ 
allows us to lift any central path of $Y_\affine$ from any point of $\tilde{M}$, and concludes the proof of the lemma.
\end{proof}

\begin{proof}[End of the proof of Proposition \ref{propositionHYstructurecomplete}]
According to Corollary \ref{corollairerelevementcerclesalphaetbeta} and to Lemmas 
\ref{lemmerelevementcheminscentrauxYt} and \ref{lemmerelevementcheminscentrauxYa}, the local diffeomorphism
$\delta$ satisfies the path-lifting property on $Y$ in the $\alpha$, $\beta$, and central directions,
and is thus a covering map from $\tilde{M}$ to $Y$ according to 
Lemma \ref{lemmerelevementchemins}.
\end{proof}

\section{Conclusion}\label{sectionconclusion}

\subsection{End of the proof of Theorem \ref{theoremeoptimal}}\label{soussectionfinpreuveTheoremeB}

The work that has been done so far tells us that for one of the two models 
$(\mathcal{H},Y,\Sx)=(H_\torus,Y_\torus,\Sm_\torus)$
or $(H_\affine,Y_\affine,\Sm_\affine)$, $M$ is a $(\mathcal{H},Y)$-manifold whose developping map $\delta\colon\tilde{M}\to Y$
is a covering map satisfying $\delta^*\Sx=\Stilde$.
With these informations, we will finish the proof of Theorem \ref{theoremeoptimal}.
We will use the link between the geometrical and algebraic point of views on the models $(Y_\torus,\Sm_\torus)$ and $(Y_\affine,\Sm_\affine)$,
explained in Paragraphs \ref{soussectionexemplealgebriquesgeometrie} and \ref{soussectionouvertshomogenes}.

\subsubsection{Case of $(Y_\affine,\Sm_\affine)$}\label{soussoussectionconclusioncasYa}

We first assume that $(M,\Sm)$ is locally isomorphic to $(Y_\affine,\Sm_\affine)$.
Since $Y_\affine$ is simply connected (because homeomorphic to $\Heis{3}$), the covering map 
$\delta\colon\tilde{M}\to Y_\affine$ is actually a diffeomorphism in this case.
Since the developping map conjugates the action of $\piun{M}$ on $\tilde{M}$
to the action of the holonomy group $\Gamma=\rho(\piun{M})\subset H_\affine$ on $Y_\affine$,
we can assume without lost of generality that $M$ is a compact quotient 
$\Gamma\backslash Y_\affine$, with $\Gamma$ a discrete subgroup of $H_\affine$
acting freely, properly, and cocompactly. Since $f$ is an automorphism of $(M,\Sm)$, we moreover deduce from Proposition
\ref{propositiongroupeautomorphismeetLiouvillequatremodeles} that $f\in\Nor_{H_\affine}(\Gamma)$.
\par We saw in Paragraph \ref{soussoussectionYaffine} that the identification between $\Heis{3}$ and $Y_\affine$
given by the orbital map at $o_\affine$ conjugates the action of $H_\affine$ on $Y_\affine$,
and the action of the semi-direct product $\Heis{3}\rtimes\mathcal{A}$ of affine automorphisms of $\Heis{3}$
preserving its left-invariant structure.
We can thus
assume that $M$ is a quotient $\Gamma\backslash\Heis{3}$, with 
$\Gamma$ a discrete subgroup of $\Heis{3}\rtimes\mathcal{A}$ acting freely, properly, and cocompactly on $\Heis{3}$,
and that $f\in\Nor_{\Heis{3}\rtimes\mathcal{A}}(\Gamma)$.
\par Denoting 
$[x,y,z]=
\left(\begin{smallmatrix}
1 & x & z \\
0 & 1 & y \\
0 & 0 & 1
\end{smallmatrix}\right)
$, the identification $[x,y,z]\in\Heis{3}\mapsto (x,y,z)\in\R^3$ of $\Heis{3}$ with $\R^3$
is equivariant for the following injective morphism from $\Heis{3}\rtimes\mathcal{A}$ to the affine transformations of $\R^3$:
\[
\Theta\colon
([x,y,z],\varphi_{\lambda,\mu})\in\Heis{3}\rtimes\mathcal{A}\mapsto
\begin{pmatrix}
\lambda & 0 & 0 \\
0 & \mu & 0 \\
0 & \mu x & \lambda\mu
\end{pmatrix}
+
\begin{bmatrix}
x \\
y \\
z
\end{bmatrix}\in\Aff{3}.
\] \\
\par $M$ is thus diffeomorphic to the quotient $\Lambda\backslash\R^3$, where $\Lambda\coloneqq\Theta(\Gamma)$ is a discrete subgroup of
affine transformations of $\R^3$ contained in $S\coloneqq\Theta(\Heis{3}\rtimes\mathcal{A})$, 
acting freely, properly and cocompactly on $\R^3$.
Since $S$ is solvable (because $\Heis{3}\rtimes\mathcal{A}\simeq\Pmin$ is), the work of Fried and Goldmann in \cite{friedgoldmann} 
(more precisely Theorem 1.4, Corollary 1.5 and Paragraphs 3 and 4 of this paper)
implies the existence of a so-called \emph{crystallographic hull} $C$ of $\Lambda$.
This group $C$ is a closed subgroup of $S$ 
containing $\Lambda$, and whose identity component $C^0$ satisfies the following assumptions:
$\Lambda\cap C^0$ has finite index in $\Lambda$ and is cocompact in $C^0$,
$C^0$ acts simply transitively on $\R^3$, and $C^0$ is isomorphic to $\R^3$, $\Heis{3}$, or $Sol$.
One can easily check that $S$ does not contain any subgroup isomorphic to $\R^3$, that the subgroups of $S$ isomorphic to $Sol$
do not act simply transitively on $\R^3$,
and that $\Theta(\Heis{3})$ is the only subgroup of $S$ isomorphic to $\Heis{3}$.
Finally, $C^0$ is equal to $\Theta(\Heis{3})$, and therefore, $\Lambda\cap\Theta(\Heis{3})$ has finite index
in $\Lambda$ and is cocompact in $\Theta(\Heis{3})$.
As a consequence, $\Gamma_0\coloneqq\Gamma\cap\Heis{3}$ 
has finite index in $\Gamma$ and is a cocompact lattice of $\Heis{3}$. \\

\par The morphism $p\colon(g,\varphi)\in\Heis{3}\rtimes\mathcal{A}\mapsto \varphi\in\mathcal{A}$ having a kernel equal to $\Heis{3}$,
$\Gamma/\Gamma_0$ is isomorphic to $p(\Gamma)\subset\mathcal{A}$.
But $\mathcal{A}$ is isomorphic to $(\R^*)^2$,
and a finite subgroup of $\mathcal{A}$ is thus contained in the subgroup $\{\varphi_{\pm1,\pm1}\}$ of cardinal 4, implying that
$\Gamma_0$ is a subgroup of $\Gamma$ of index at most 4.
Let us denote $f=(g,\varphi)\in\Nor_{\Heis{3}\rtimes\mathcal{A}}(\Gamma)$.
Then we have 
$g\varphi(\Gamma_0)g^{-1}=\Gamma_0$, and the affine automorphism 
$x\mapsto g\varphi(x)$ induces therefore a diffeomorphism $\check{f}$ of $\check{M}\coloneqq\Gamma_0\backslash\Heis{3}$ through 
$\check{f}(x\Gamma_0 )=g\varphi(x)\Gamma_0$. 
The canonical projection $\check{\pi}\colon\check{M}=\Gamma_0\backslash\Heis{3} \to M=\Gamma\backslash\Heis{3}$
is a covering of finite order equal to the index of $\Gamma_0$ in $\Gamma$,
and we have $\check{\pi}\circ\check{f}=f\circ\check{\pi}$. 
\par We know that $\varphi$ is equal to $\varphi_{\lambda,\mu}$ for some $(\lambda,\mu)\in(\R^*)^2$ (see \eqref{equationdefinitionautomorphismHeis3}), and
it only remains to show that $\abs{\lambda}<1$ and $\abs{\mu}>1$, or the contrary,
to conclude that $\check{f}$ is a partially hyperbolic affine automorphism of $\Heis{3}$.
Let us assume by contradiction that $\abs{\lambda}<1$ and $\abs{\mu}<1$. 
Choosing a left-invariant volume form $\nu$ on $\Heis{3}$, we have $((\Diff{e}{\varphi})^*\nu)_e=\lambda^2\mu^2\nu_e$,
and $\nu$ induces a volume form $\bar{\nu}$ on $\check{M}=\Gamma_0\backslash\Heis{3}$
such that $\check{f}^*\bar{\nu}=\lambda^2\mu^2\bar{\nu}$, because $L_g$ preserves $\nu$.
As $\check{f}$ is a diffeomorphism of the compact manifold $\check{M}$, 
we must have $\int_{\check{M}}\bar{\nu}=\int_{\check{M}}\check{f}^*\bar{\nu}=\lambda^2\mu^2\int_{\check{M}}\bar{\nu}$, which is a 
contradiction because $\int_{\check{M}}\bar{\nu}\neq 0$ and $\lambda^2\mu^2<1$. 
The same argument shows that we cannot have $\abs{\lambda}>1$ and $\abs{\mu}>1$ neither, which 
finishes the proof of Theorem 
\ref{theoremeoptimal} in the case of the local model $(Y_\affine,\Sm_\affine)$. 

\subsubsection{Case of $(Y_\torus,\Sm_\torus)$}
We now assume that $\Sm$ is locally isomorphic to $(Y_\torus,\Sm_\torus)$.
Identifying $Y_\torus$ with $\SL{2}$ as explained in Paragraph \ref{soussoussectionYtorus},
we can lift the developping map $\delta\colon\tilde{M}\to Y_\torus$ to a map $\tilde{\delta}\colon\tilde{M}\to\SLtilde{2}$
through the universal cover morphism $\pi_{\SL{2}}\colon\SLtilde{2}\to\SL{2}$.
As $\delta$ is a covering map according to Proposition \ref{propositionHYstructurecomplete}, 
$\tilde{\delta}$ is a diffeomorphism because $\SLtilde{2}$ is simply connected.
As $M$ is supposed to be orientable, $\piun{M}$ preserves its orientation, implying
that the holonomy group $\rho(\piun{M})$ is contained in the subgroup $H_\torus^+=\GLplus{2}$ of elements of positive determinant.
We saw in Paragraph \ref{soussoussectionYtorus} that the diffeomorphism $\theta_{o_\torus}\circ\iota\colon\SL{2}\to Y_\torus$ 
conjugates the action of $\GLplus{2}$ on $Y_\torus$ and the action of $\SL{2}\times A$ on $\SL{2}$.
As $\pi_{\SL{2}}$ is equivariant for the projection $\SLtilde{2}\times\tilde{A}\to\SL{2}\times A$,
we finally conclude that the diffeomorphism $\tilde{\delta}\colon\tilde{M}\to\SLtilde{2}$ is equivariant for a morphism 
$\tilde{\rho}\colon\piun{M}\to\SLtilde{2}\times\tilde{A}$.
We can thus assume that $M$ is a quotient $\tilde{\Gamma}\backslash\SLtilde{2}$, with $\tilde{\Gamma}$
a discrete subgroup of $\SLtilde{2}\times\tilde{A}$ acting freely, properly, and cocompactly on $\SLtilde{2}$.
Possibly replacing $f$ by $f^2$, we can assume that $f$ preserves the orientation of $M$, 
and Proposition \ref{propositiongroupeautomorphismeetLiouvillequatremodeles} implies then
that $f=L_g\circ R_{a^t}$ with $(g,a^t)\in\Nor_{\SLtilde{2}\times\tilde{A}}(\tilde{\Gamma})$. \\

\par Denoting by $r_1\colon\SLtilde{2}\times\tilde{A}\to\SLtilde{2}$
the projection on the first factor, and $\tilde{\Gamma}_0\coloneqq r_1(\tilde{\Gamma})\subset\SLtilde{2}$,
we now prove the following result.
\begin{fact}\label{factkulkarniraymond}
$\tilde{\Gamma}_0$ is a cocompact lattice of $\SLtilde{2}$, 
and $\tilde{\Gamma}$ is the graph-group $\gr(\tilde{u},\tilde{\Gamma}_0)$ of
a morphism $\tilde{u}\colon\tilde{\Gamma}\to\tilde{A}$.
\end{fact}
\begin{proof}
Choosing a generator $z$ of the center $\Ztilde$ of $\SLtilde{2}$, 
the finiteness of the level proved by Salein in \cite[Theorem 3.3.2.3]{salein} implies the existence of a non-zero integer $k\in\N^*$ such that
$\tilde{\Gamma}\cap(\Ztilde\times\{e\})=\langle(z^k,e)\rangle$.
We will denote by $\langle g \rangle$ the group generated by an element $g$, and we
introduce the group $\PSLk\coloneqq\SLtilde{2}/\langle z^k \rangle$ and denote by  
$p_k\colon\SLtilde{2}\to\PSLk$ its universal cover. 
Then, denoting $A_k=p_k(\tilde{A})$
and $\Gamma_k\coloneqq (p_k\times p_k)(\tilde{\Gamma})<\PSLk\times A_k$,
$p_k$ induces a diffeomorphism $\bar{p}_k\colon\tilde{\Gamma}\backslash\SLtilde{2}\to\Gamma_k\backslash\PSLk$
(because $\Ker p_k=\langle z^k \rangle$ and $(z^k,e)\in\tilde{\Gamma}$),
implying in particular that $\Gamma_k$ acts freely, properly, and cocompactly on $\PSLk$.
\par We can now apply the work of Kulkarni-Raymond in \cite{kulkarni} to $\Gamma_k$.
We denote by $\pi\colon\SLtilde{2}\to\PSL{2}$ the universal cover morphism of $\PSL{2}$ (of kernel $\Ztilde$),
and by $\pi_k\colon\PSLk\to\PSL{2}$ the induced $k$-fold covering by $\PSLk$.
Then, with $\Gamma=(\pi\times\pi)(\tilde{\Gamma})$ and $\Gamma_0=r_1(\Gamma)<\PSL{2}$ the projection on the first factor,
the form \cite[Lemma 4.3.1]{tholozan} of Kulkarni-Raymond's results proved by Tholozan implies that
$\Gamma_0$ is a cocompact lattice of $\PSL{2}$, and that $\pi_k\circ r_1\restreinta_{\Gamma_k}$ is injective.
\par The first assertion ensures that $\tilde{\Gamma}_0$ is discrete in $\SLtilde{2}$.
The second one implies that $\Gamma=\gr(u,\Gamma_0)$ is the graph-group of a morphism $u\colon\Gamma_0\to A=\pi(\tilde{A})$.
Since $r_1\restreinta_{\tilde{\Gamma}}$ is also injective, this implies that $\tilde{\Gamma}$ is the graph 
of a morphism $\tilde{u}\colon\tilde{\Gamma}_0\to\tilde{A}$, trivial on $\tilde{\Gamma}_0\cap\Ztilde$.
\par Since $\Ztilde\cap\tilde{\Gamma}_0=\langle z^k \rangle$ is finite,
the projection $\tilde{\Gamma}_0\backslash\SLtilde{2}\to\Gamma_0\backslash\PSL{2}$ has finite fibers,
implying that $\tilde{\Gamma}$ is a cocompact lattice as $\Gamma_0\backslash\PSL{2}$ is compact. 
\end{proof}

The projection $\Gamma_0=\pi(\tilde{\Gamma}_0)$ is a cocompact lattice of $\PSL{2}$ according to the proof of 
Fact \ref{factkulkarniraymond}, 
and $\Gamma_0\backslash\Nor_{\PSL{2}}(\Gamma_0)$ is thus finite.
Therefore, $\tilde{\Gamma}_0\backslash\Nor_{\SLtilde{2}}(\tilde{\Gamma}_0)$ is finite as well since
the projection 
$\tilde{\Gamma}_0\backslash\Nor_{\SLtilde{2}}(\tilde{\Gamma}_0)\to\Gamma_0\backslash\Nor_{\PSL{2}}(\Gamma_0)$
has finite fibers ($\Ztilde\cap\tilde{\Gamma}_0=\langle z^k \rangle$ is finite according to the finiteness of the level).
\par Recall that $f=L_g\circ R_a$, where
$(g,a^t)\in\Nor_{\SLtilde{2}\times\tilde{A}}(\tilde{\Gamma})$.
Therefore $g\in\Nor_{\SLtilde{2}}(\tilde{\Gamma}_0)$, and 
since $\tilde{\Gamma}_0\backslash\Nor_{\SLtilde{2}}(\tilde{\Gamma}_0)$ is finite, 
there exists $n\in\N^*$ such that $\gamma\coloneqq g^n\in\tilde{\Gamma}_0$.
Denoting $a\coloneqq a^n \tilde{u}(\gamma)^{-1}$, we have
$f^n=L_{\gamma}\circ R_{a^n}=R_a\circ(L_{\gamma}\circ R_{\tilde{u}(\gamma)})$.
But $L_{\gamma}\circ R_{\tilde{u}(\gamma)}$ acts trivially on the quotient $\tilde{\Gamma}\backslash\SLtilde{2}$, 
and therefore $f=R_a$ is a non-zero time-map of the algebraic contact-Anosov flow $(R_{a^t})$ on $\tilde{\Gamma}\backslash\SLtilde{2}$. 
\par Let us underline that $(R_{a^t})$ is indeed Anosov,
because the work of Zeghib in
\cite[Prop. 4.2 p.868]{zeghib} proves that $(R_{a^t})$ is 
\emph{quasi-Anosov} with the definition of Ma{\~{n}}é,
and Ma{\~{n}}é proves in \cite[Theorem A]{mane} that three-dimensional quasi-Anosov flows 
are Anosov. \\

\par This concludes the proof of Theorem \ref{theoremeoptimal}
in the case where $\Sm$ is locally isomorphic to $(Y_\torus,\Sm_\torus)$,
and concludes thus its whole proof.

\subsection{Proof of Theorem \ref{theoremeprincipalPHdiffeos}}\label{soussectionpreuvetheoremeA}

Theorem \ref{theoremeoptimal} directly implies 
Theorem \ref{theoremeprincipalPHdiffeos} stated in the introduction thanks to an argument of Brin.
More precisely, we obtain the following refined version of Theorem \ref{theoremeprincipalPHdiffeos},
where no domination is required on the central direction, and where the two remaining directions can \emph{a priori}
be both contracted, or both expanded.

\begin{corollary}
\label{corollairesansdomination}
Let $M$ be a closed, connected and orientable three-dimensional manifold, endowed with a smooth splitting
$\Fitan{M}=E^\alpha\oplus E^\beta\oplus E^c$ such that $E^\alpha \oplus E^\beta$ is a contact distribution.
Let $f$ be a diffeomorphism of $M$ that preserves this splitting, 
and such that
\begin{itemize}
 \item each of the distributions $E^\alpha$ and $E^\beta$ is either uniformly contracted, or uniformly expanded by $f$,
 \item and $NW(f)=M$.
\end{itemize}
Then the conclusions of Theorem \ref{theoremeprincipalPHdiffeos} hold.
In particular, $f$ is a partially
hyperbolic diffeomorphism.
\end{corollary}
\begin{proof}
Since $E^\alpha\oplus E^\beta$ is contact and $M$ connected, 
any two points of $M$ are linked by the concatenation of a finite number of paths, tangent either to $E^\alpha$ or to $E^\beta$
(this is for example a consequence of the work of Sussmann in \cite[Theorem 4.1]{sussmann}).
In other words, the pair $(\mathcal{F}^\alpha,\mathcal{F}^\beta)$ of foliations associated
to $(E^\alpha,E^\beta)$ is topologically transitive in the terminology of Brin in \cite{brin}.
Our hypothesis of uniform contraction or expansion of the distributions $E^\alpha$ and $E^\beta$
directly implies that $\mathcal{F}^\alpha$ and $\mathcal{F}^\beta$
are uniformly contracted or expanded in the terminology of \cite{brin}.
Since $NW(f)=M$ by hypothesis, \cite[Theorem 1.1]{brin} implies that $f$ is topologically transitive.
In fact, although Brin states his result assuming that one of the distributions is contracted, and the other one expanded, 
it is easy to see that his proof does in fact not use this assumption, and that the same proof works if both distributions
are expanded, or both contracted.
\par We are now under the hypotheses of Theorem \ref{theoremeoptimal}, and its conclusions hold.
\end{proof}

\bibliographystyle{alpha}
\bibliography{biblio-diffeosPHcontact}

\end{document}